\pdfoutput=1
\documentclass[a4paper, 10pt]{amsart}

\usepackage[protrusion=true,expansion=true]{microtype} 

\usepackage{float}
\usepackage[all]{xy}
\usepackage{physics}
\usepackage{fancyhdr}
\usepackage[utf8]{inputenc}
\usepackage{graphicx} 
\usepackage{wrapfig} 
\usepackage{mathtools} 
\usepackage{adjustbox}
\usepackage{leftidx}
\usepackage{tikz-cd}
\usepackage{enumitem}
\usepackage{stackrel}

\usepackage{mathpazo}
\usepackage{mathrsfs}
\usepackage[T1]{fontenc}
\usepackage{amsmath}
\usepackage{amssymb}
\usepackage{hyperref}
\usepackage{cleveref}
\usepackage{comment}
\usepackage{color}
\usepackage{tikz}
\usepackage{tikz-cd}
\usepackage{xparse}
\usetikzlibrary{matrix,arrows,decorations.pathmorphing}

\newtheorem{example}{Example}[section]
\newtheorem{theorem}{Theorem}[section]
\newtheorem{proposition}{Proposition}[section]

\newtheorem{definition}{Definition}[section]
\newtheorem{lemma}{Lemma}[section]
\newtheorem{remark}{Remark}[section]

\newtheorem{notation}{Notation}[section]

\crefname{lemma}{Lemma}{lemma}
\crefname{remark}{Remark}{remark}
\crefname{corollary}{Corollary}{corollary}
\crefname{theorem}{Theorem}{theorem}
\crefname{proposition}{Proposition}{proposition}
\crefname{example}{Example}{example}
\crefname{definition}{Definition}{definition}
\crefname{notation}{Notation}{notation}
\crefname{appendix}{Appendix}{appendix}
\crefname{section}{Section}{section}

\newcommand{\AAA}{\mathfrak A}
\newcommand{\BBB}{\mathcal B}
\newcommand{\CC}{\mathscr C}
\newcommand{\DDD}{\mathcal D}

\newcommand{\FFF}{\mathscr F}

\newcommand{\HHH}{\mathcal H} 
\newcommand{\LLL}{\mathcal L} 
\newcommand{\MMM}{\mathcal M}
\newcommand{\MM}{\mathfrak M}

\newcommand{\NNN}{\mathcal N}
\newcommand{\NN}{\mathfrak N}
\newcommand{\JJJ}{\mathcal J}

\newcommand{\OOO}{\mathcal O}

\newcommand{\SSS}{\mathscr S}
\newcommand{\SI}{\mathcal S}

\newcommand{\ZZZ}{\mathcal{Z}}

\newcommand{\C}{\mathbb C} 
\newcommand{\R}{\mathbb R}  
\newcommand{\N}{\mathbb N} 
\newcommand{\ut}{\mathbb{1}}

\DeclareDocumentCommand{\sqv}{O{} O{} O{} O{}}{
\begin{small}
  \begin{tikzcd}[ampersand replacement=\&]
     \cdot \arrow{r}{#1} \arrow{d}{#3} \& \cdot \arrow{d}{#2}\\
     \cdot \arrow{r}{#4} \& \cdot
  \end{tikzcd}  
\end{small}
}

\DeclareDocumentCommand{\csqv}{O{} O{} O{} O{} }{
\begin{small}
\overline{
 \begin{tikzcd}[ampersand replacement=\&]
    \cdot \arrow{r}{#1} \arrow{d}{#3} \& \cdot \arrow{d}{#2}\\
      \cdot \arrow{r}{#4} \& \cdot
  \end{tikzcd}}
\end{small}
}
\newcommand*{\bfrac}[2]{\genfrac{[}{]}{0pt}{}{#1}{#2}}

\newcommand{\titleinfo}{Realization of rigid C$^*$-tensor categories via Tomita bimodules}
\newcommand{\titleinfoshort}{Realization of rigid C$^*$-tensor categories}
\newcommand{\authorinfo}{Luca Giorgetti, Wei Yuan} 

\begin{document}

\title{\LARGE\textbf{\titleinfo}} 
\author{\large\textsc{\authorinfo}} 

\address{Dipartimento di Matematica, Universit\`a di Roma Tor Vergata \\
Via della Ricerca Scientifica, 1, I-00133 Roma, Italy}
\email{giorgett@mat.uniroma2.it}
\address{Institute of Mathematics, Academy of Mathematics and Systems Science \\
Chinese Academy of Sciences, Beijing, 100190, China}
\address{School of Mathematical Sciences, University of Chinese Academy of Sciences, 
Beijing 100049, China}
\email{wyuan@math.ac.cn}

\date{}

\begin{abstract}
Starting from a (small) rigid C$^*$-tensor category $\CC$ with simple unit, 
we construct von Neumann algebras associated to each of its objects. 
These algebras are factors and can be either semifinite (of type II$_1$ or II$_\infty$, 
depending on whether the spectrum of the category is finite or infinite) or they can be 
of type III$_\lambda$, $\lambda\in (0,1]$. The choice of type is tuned by the choice of 
Tomita structure (defined in the paper) on certain bimodules we use in the construction. 
Moreover, if the spectrum is infinite we realize the whole tensor category directly as 
endomorphisms of these algebras, with finite Jones index, by exhibiting a fully faithful 
unitary tensor functor $F:\CC \hookrightarrow End_0(\Phi)$ where $\Phi$ is a factor 
(of type II or III).

The construction relies on methods from free probability (full Fock space, amalgamated 
free products), it does not depend on amenability assumptions, 
and it can be applied to categories with uncountable spectrum (hence it provides an 
alternative answer to a conjecture of Yamagami \cite{Y3}).
Even in the case of uncountably generated categories, we can refine the previous 
equivalence to obtain realizations on $\sigma$-finite factors as endomorphisms 
(in the type III case) and as bimodules (in the type II case).

In the case of trivial Tomita structure, we recover the same algebra obtained in \cite{PopaS} and \cite{AMD},
namely the (countably generated) free group factor $L(F_\infty)$ if the given category
has denumerable spectrum, while we get the free group factor with uncountably many 
generators if the spectrum is infinite and non-denumerable.
\end{abstract}


\maketitle

\section{Introduction}

Tensor categories (also called monoidal categories, see \cite{ML}, \cite{PSDV})
are abstract mathematical structures which naturally arise in different ways 
when dealing with operator algebras on a Hilbert space $\HHH$. In this context, 
they are typically unitary and C$^*$ \cite{GLR}, i.e., they come equipped with
an involution $t \mapsto t^*$ and a norm $t\mapsto \|t\|$ on arrows,
describing respectively the adjunction and the operator norm in $\BBB(\HHH)$.
Examples come from categories of endomorphisms of von Neumann algebras with tensor
structure given on objects by the composition of endomorphisms, or from categories
of bimodules with the relative tensor product (Connes fusion), or from representation 
categories of quantum groups, or again from the analysis of superselection sectors 
in Quantum Field Theory (in the algebraic formulation of QFT due to Haag and Kastler)
where the tensor product describes the composition of elementary particle states.
See \cite{Mue10tour} and references therein. In the theory of subfactors \cite{J}, 
\cite{JoSuBook}, tensor categories (or more generally a 2-category) can be associated 
to a given (finite index) subfactor by looking at its fusion graphs. 

One of the most exciting additional structures that can be given on top of an abstract
(C$^*$-)tensor categories is an intrinsic notion of dimension \cite{L-R-dim}, 
see also \cite{DHR}, \cite{DR}, which associates a real number $d_X \geq 1$ to each 
object $X$ of the category. The dimension is intrinsic in the sense that it is formulated 
by means of objects and arrows in the category only, more precisely by means of conjugate 
objects and (solutions of) the conjugate equations (also called zig-zag equations). 
In the case of the category of finite dimensional Hilbert spaces $V$, the intrinsic dimension 
coincides with the usual notion of dimension of $V$ as a vector space. 
The finiteness of the Jones index of a subfactor \cite{Kosaki-Ind}, \cite{L-R-dim}, 
and the finiteness of the statistics of a superselection sector in QFT \cite{L-indexI} 
are as well instances of the intrinsic dimension applied to concrete C$^*$-tensor categories.
A C$^*$-tensor category whose objects have finite dimension (i.e., admit conjugate objects
in the sense of the conjugate equations) is called rigid. Similar notions appearing in the 
literature on tensor categories, sometimes available in slightly more general contexts than ours, 
are the one of pivotality, sphericality, see, e.g., \cite{PSDV}, \cite{Mue10tour}, and
Frobenius duality \cite{Y4}.

The question (motivated by the previous discussion) of how to realize abstract
rigid C$^*$-tensor categories as endomorphisms (or more generally bimodules) 
of operator algebras can be traced back to the seminal work of Jones on the index of 
type II$_1$ subfactors \cite{J} and it has been studied by many authors over the years. In \cite{HY}, 
Hayashi and Yamagami realize categories admitting an \lq\lq amenable"
dimension function as bimodules of the hyperfinite type II$_1$ factor. 
Later on, a different realization of arbitrary categories with countable 
spectrum (the set of isomorphism classes of simple objects) over amalgamated free product 
factors has been given by Yamagami in \cite{Y3}.
Both these works, and more generally the research in the direction of constructing operator algebras
out of finite dimensional combinatorial data, have their roots in the work of Popa on the 
construction of subfactors associated to standard lattices \cite{Pop95-1}, i.e., 
on the reverse of the machinery which associates to a subfactor its standard invariant
(the system of its higher relative commutants). Moreover, the powerful 
deformation-rigidity theory developed by Popa \cite{Popa-betti}, \cite{Pop-rigid-I}, \cite{Pop-rigid-II} makes 
it possible to completely determine the bimodule categories for certain classes of type II$_1$ factors 
(see for example \cite{V07}, \cite{FVaes}, \cite{DS}, \cite{SFS} for some explicit results on the calculation 
of bimodule categories). More recently, Brothier, Hartglass and Penneys proved in \cite{AMD} that every 
countably generated rigid C$^*$-tensor category can be realized as bimodules of free group factors 
(see also \cite{GJS}, \cite{JSW}, \cite{KSunder}, \cite{GJS11}, \cite{Unshade}, \cite{HFree}).

The purpose of the present work is to re-interpret the construction in \cite{AMD} via Tomita bimodules 
(defined in the paper, see \cref{tomita_bimodule_def}), to generalize it in order to obtain different types of factors 
(possibly III$_\lambda$, $\lambda\in(0,1]$), and to prove the universality of free group factors  
for rigid C$^*$-tensor categories with uncountable spectrum (see \cref{section_lambda1}).

Unlike \cite{AMD}, we do not use the language of
Jones' planar algebras \cite{Jplanar} (a planar diagrammatic axiomatization of Popa's standard lattices).
Instead, we prefer to work with the tensor category itself in order to exploit its 
flexibility: our main trick is to
double the spectrum of the given category and consider the set of all \lq\lq letters" 
corresponding to inequivalent simple objects and to their conjugate objects. 
This allows us to define Tomita structures (see \cref{section_CstartoOA})
on certain pre-Hilbert bimodules $H(X)$ that we associate to the objects $X$ of the category,
without having to cope with ambiguities arising from the choice of solutions of the conjugate
equations in the case of self-conjugate objects (depending on their Frobenius-Schur indicator, i.e., depending on the
reality or pseudo-reality of self-conjugate objects, in the terminology of \cite{L-R-dim}). 
The algebra $\Phi(H(X))''$ associated to the Tomita bimodule $H(X)$ (via Fock space construction, see below) 
is of type II or III$_\lambda$,
$\lambda\in(0,1]$, depending on the choice of Tomita structure, 
both in the finite and infinite spectrum case. Even if we choose the \emph{trivial} Tomita structure, 
our construction is different from the one of \cite{AMD} in the sense that for finitely generated 
categories we obtain different algebras, e.g., the trivial category with only one simple object $\ut$ produces 
the free group factor with two generators $L(F_2)$.

The paper is organized as follows. In \cref{section_prelim}, we consider pre-Hilbert C$^*$-bimodules,
a natural non-complete generalization of Hilbert C$^*$-bimodules,
which we use to treat finite-dimensional purely algebraic issues before passing to the norm or
Hilbert space completions. Moreover, we define Tomita structures on a pre-Hilbert $\AAA$-$\AAA$
bimodule (hence we define Tomita $\AAA$-$\AAA$ bimodules), where $\AAA$ is a von Neumann algebra (\cref{tomita_bimodule_def}).
The terminology is due to the fact that every Tomita algebra is a Tomita $\C$-$\C$ bimodule.
We derive the main properties of Tomita bimodules only in the case of semifinite von Neumann algebras
$\AAA$ and choosing a reference normal semifinite faithful tracial weight $\tau$ on $\AAA$. 
We recall the definition of (full) Fock space $\FFF(H)$ (crucial in Voiculescu's free probability theory 
\cite{NVK}) here associated to a Tomita $\AAA$-$\AAA$ bimodule $H$ and express the 
Tomita-Takesaki's modular objects of the von Neumann algebras $\Phi(H)''$,
generated by $\AAA$ and by the creation and annihilation operators, in terms of the Tomita
structure of $H$ and of the chosen tracial weight on $\AAA$ (\cref{module_auto}). This is the main result of the section.

In \cref{section_CstartoOA}, which is the main part of this work, we associate to every object $X$ 
of a rigid C$^*$-tensor category, with simple unit, a semifinite von Neumann algebra
$\AAA(X)$ and a Tomita $\AAA(X)$-$\AAA(X)$ bimodule $H(X)$ (\cref{prop_tomitafromC}). The Tomita structure 
is determined by an arbitrary choice of strictly positive numbers $\lambda_\alpha$
for every $\alpha\in\SSS$, where $\SSS$ is a representative set of simple objects in the category
(i.e., $\SSS$ labels the spectrum of the category). The algebra $\AAA(X)$ is semifinite with \lq\lq canonical"
tracial weight given by the rigidity structure of the category (the standard left and right inverses of 
\cite{L-R-dim} are indeed tracial), moreover $\AAA(X)$ is a (possibly infinite) direct sum of 
matrix algebras (because the category is automatically semisimple) and the weight of minimal 
projections equals the intrinsic dimension of the elements in $\SSS$. The associated von 
Neumann algebra $\Phi(H(X))''$ on Fock space (or better, on its Hilbert space completion 
with respect to the tracial weight) turns out to be a factor for every $X$ (\cref{factor_thm}), and we study the 
type of $\Phi(H(X))''$ depending on the size of the spectrum and on the chosen Tomita structure on 
$H(X)$ (\cref{type_prop}). For ease of exposition we state the results of this section only in the case $X=\ut$ (tensor unit object),
but these can be generalized to an arbitrary $X\in\CC$ without conceptual difficulty (\cref{rmk_from1toX}).

In \cref{section_lambda1}, we study $\Phi(H(\ut))''$ in the case of categories with
infinite spectrum $\SSS$ and assuming $\lambda_\alpha = 1$ for every $\alpha\in\SSS$ 
(trivial Tomita structure). We prove that the free group factor $L(F_\SSS)$,
either with countably many or with uncountably many generators, sits in a corner of $\Phi(H(\ut))''$ (\cref{thm_freecorner}).
Easy examples of rigid C$^*$-tensor categories which are not amenable and have uncountable spectrum
come from the algebraic group algebras of uncountable discrete and non-amenable groups, see \cite[Example 2.2]{HY}. E.g., consider the pointed discrete categories with fusion ring equal to $\C[\R\times F_2]$ or $\C[F_\infty]$, where $\R$ is endowed with the discrete topology and $F_\infty$ is the free group with uncountably many generators. Obtaining a realization result for such categories, which were previously not covered in the literature, was the original motivation of our work.

In \cref{section_endos}, again in the infinite spectrum case, we re-interpret the algebra $\Phi(H(\ut))''$ as a corner of a bigger auxiliary algebra $\Phi(\CC)$, we associate to every object $X\in\CC$ a non-unital endomorphism of $\Phi(\CC)$, and we cut it down to an endomorphism of $\Phi(H(\ut))''$, denoted by $F(X)$. We conclude by showing that $F$ is a fully faithful unitary tensor functor from $\CC$ into $End(\Phi(H(\ut))'')$, hence a realization of the category $\CC$ as endomorphisms with finite index of a factor (\cref{thm_realization}) which can be chosen to be either of type II$_\infty$ or of type III$_\lambda$, $\lambda\in (0,1]$. Moreover, the realization always happens on a $\sigma$-finite factor (e.g., $L(F_\infty)$ if we choose trivial Tomita structure) by composing $F$ with a suitable equivalence of bimodule (or endomorphism) categories coming from the ampliation (\cref{thm_ampliation}).

In \cref{appendix_am_prod}, as we could not find a reference, we define and study the amalgamated free product
of arbitrary von Neumann algebras (not necessarily $\sigma$-finite), as we need in our construction when dealing with uncountably generated categories.

In \cref{section_nonsigma}, we generalize some results concerning the Jones projection and the structure
of amalgamated free products, well-known in the $\sigma$-finite case, to the case of arbitrary von Neumann algebras.

\section{Preliminaries}\label{section_prelim}

Let $\NNN$ be a unital C$^*$-algebra. A \textbf{pre-Hilbert $\NNN$-$\NNN$ bimodule} is an 
$\NNN$-$\NNN$ bimodule $H$ with a sesquilinear $\NNN$-valued inner product 
$\braket{\cdot}{\cdot}_{\NNN}: H \times H \rightarrow \NNN$, fulfilling 
\begin{enumerate}
    \item $\braket{\xi_1}{B \cdot \xi_2 \cdot A}_{\NNN} = \braket{B^* \cdot \xi_1}{\xi_2}_{\NNN} A$,
          for every $\xi_1$, $\xi_2 \in H$ and $A$, $B \in \NNN$,
    \item $\braket{\xi}{\xi}_{\NNN} \geq 0$, and  
          $\braket{\xi}{\xi}_{\NNN}=0$ implies $\xi = 0$,
    \item $\|A \cdot \xi\|_{H} \leq \|A\| \|\xi\|_{H}$,
        where $\|\xi\|_{H} = \|\braket{\xi}{\xi}_{\NNN}\|^{1/2}$. 
\end{enumerate}
Note that (2) implies that $\braket{\xi_1}{\xi_2}_{\NNN} =\braket{\xi_2}{\xi_1}^*_{\NNN}$.
A pre-Hilbert $\NNN$-$\NNN$ bimodule $H$ which is complete in the norm 
$\|\cdot\|_{H}$ is called a \textbf{Hilbert $\NNN$-$\NNN$ bimodule}. 
Let $\LLL(H)$ and $\BBB(H)$ be respectively the set of adjointable and bounded adjointable linear mappings  
from $H$ to $H$. Condition $(3)$ implies that the left action of $\NNN$ on $H$ is a 
*-homomorphism from $\NNN$ into $\BBB(H)$.
If $H$ is a Hilbert $\NNN$-$\NNN$ bimodule, then $\BBB(H) = \LLL(H)$ (see \cite{CL}). 

\begin{remark}
Let $H$ be a pre-Hilbert $\NNN$-$\NNN$ bimodule and $T \in \BBB(H)$.
Note that
\begin{align*}
    \|T^*\xi\|_{H}^2 =\|\braket{T^*\xi}{T^*\xi}_{\NNN}\| = \|\braket{\xi}{TT^*\xi}_\NNN\| \leq
    \|TT^*\xi\|_{H}\|\xi\|_{H} \leq \|T\| \|T^*\xi\|_{H}\|\xi\|_{H}.
\end{align*}
Therefore $T^* \in \BBB(H)$. Thus if $\overline{H}$ is the completion of $H$, then each 
$T \in \BBB(H)$ extends uniquely to an element in $\BBB(\overline{H})$.
But $\BBB(H)$ is not a norm dense subalgebra of $\BBB(\overline{H})$ in general.

\begin{example} \label{density_example}
    Let $\NNN = \C$, $H = l^1(\N)$ and $\overline{H}=l^2(\N)$. 
    Note that $\N = \cup_{n=0}^{\infty}\{k: 2^n \leq k < 2^{n+1}-1\}$.
    Let $\xi_j = 1/\sqrt{2^n} \sum_{i=0}^{2^n-1} e_{2^n+i}$ where
    $\{e_1, e_2, \ldots \}$ is the canonical orthonormal basis of $l^2(\N)$.
    Define a partial isometry $V \in \BBB(l^2(\N))$ by $Ve_j = \xi_j$.

    Assume there exists $T \in \BBB(l^1(\N))$ and $\|T-V\| < 1/2$. Note that $\|Te_n - \xi_n\| < 1/2$, thus 
    \begin{align*}
        \sum_{i=0}^{2^n-1} |\braket{e_{2^n+i}}{Te_n}| 
        \geq& \sqrt{2^{n}}-\sum_{i=0}^{2^n-1} |\braket{e_{2^n+i}}{Te_n} - 1/\sqrt{2^n}|\\
        \geq& \sqrt{2^{n}}(1-\|Te_n - \xi_n\|) \geq \frac{\sqrt{2^{n}}}{2}.
    \end{align*} 
    We now choose inductively two subsequences 
    of non negative numbers $1=n_0 < n_1<  n_2 < \cdots$ and
    $n(1) < n(2) < \cdots$ such that $n(k) \geq k^2$ and 
    \begin{align*}
        \sum_{i \notin [n_{l-1}, n_{l}-1]} |\braket{e_{i}}{Te_{n(l)}}| < 1, \quad
        l = 1, 2, \ldots
    \end{align*}
    Let $k(1) =1$. Since $Te_1 \in l^1(\N)$, we can choose $n_1 \in \N$ that satisfies the above condition.  
    Assume that $\{n_1, \ldots, n_k\}$ and $\{n(1), \ldots, n(k)\}$ are chosen. 
    Recall that $T^*$ also maps $l^1(\N)$ into
    $l^1(\N)$ and $\{T^*e_i\}_{i=1}^{n_k-1} \subset l^1(\N)$.
    We may choose $n(k+1) > \max\{(k+1)^2, n_k, n(k)\}$ such that
    $\sum_{i \in [1, n_{k}-1]} |\braket{e_{i}}{Te_{n(k+1)}}| < 1/2$.
    Now it is clear that we can choose $n_{k+1} > n_k$ such that 
    $\sum_{i \notin [n_{k}, n_{k+1}-1]} |\braket{e_{i}}{Te_{n(k+1)}}| < 1$. 
    Let $\beta = \sum_{k=1}^{\infty} 1/k^2 e_{n(k)} \in l^1(\N)$. Then
    \begin{align*}
        \|T\beta\|_1 = \sum_{l=1}^{\infty} \sum_{i \in [n_{l-1}, n_l -1]}
        |\braket{e_i}{\sum_{k=1}^{\infty} 1/k^2 Te_{n(k)}}|
        \geq \sum_{l=1}^{\infty} \frac{\sqrt{2^{n(l)}}}{2l^2} - 
        \sum_{l=1}^{\infty} \frac{1}{l^2} = +\infty.
    \end{align*}
    Thus $T$ is not in $\BBB(l^1(\N))$, and $\BBB(l^1(\N))$ is not dense in $\BBB(l^2(\N))$.
\end{example}
\end{remark}

In the following, we shall also consider the Hilbert space completion $H_\varphi$ of a pre-Hilbert $\NNN$-$\NNN$ bimodule $H$, 
associated to a choice of weight on $\NNN$. 

\begin{notation}\label{notation_Hvarphi}
Let $H$ be a pre-Hilbert $\NNN$-$\NNN$ bimodule, $\varphi$ a faithful weight on the 
C$^*$-algebra $\NNN$ and denote $\NN(H,\varphi) = \{\xi \in H: \varphi (\braket{\xi}{\xi}_{\NNN}) < +\infty\}$.
The formula 
\begin{align}\label{def_inner_prod}
  \braket{\xi_1}{\xi_2} = \varphi (\braket{\xi_1}{\xi_2}_{\NNN}), \quad
    \xi_1, \xi_2 \in  \NN(H,\varphi)  
\end{align}
defines a positive definite inner product on $\NN(H,\varphi)$.
We denote the norm associated with this inner product by $\| \cdot \|_2$, i.e.,
$\|\xi\|_2 = \varphi(\braket{\xi}{\xi}_{\NNN})^{1/2}$,
and the completion of $\NN(H,\varphi)$ relative to this norm by 
$H_{\varphi}$. 
\end{notation}

By \cite[Proposition 1.2]{CL},
each $T \in \BBB(H)$ corresponds to a bounded operator 
$\pi_\varphi(T)$ on $H_\varphi$ such that $\pi_\varphi(T)\xi= T\xi$, $\xi \in \NN(H,\varphi)$, 
and $T \mapsto \pi_\varphi(T)$ is a *-representation of $\BBB(H)$ on $H_\varphi$.

\begin{lemma}\label{kernal_lemma}
    Let $\AAA$ be a von Neumann algebra and $H$ a pre-Hilbert $\AAA$-$\AAA$ bimodule.
    If $\varphi$ is a normal semifinite faithful (n.s.f.) weight on $\AAA$, then
    $Ker(\pi_\varphi) =\{0\}$.  
\end{lemma}

\begin{proof}
    If $0 \neq T \in \BBB(H)$, then there exists $\xi \in H$ such that 
    $\braket{T\xi}{T\xi}_{\AAA} \neq 0$.
    Since $\varphi$ is semifinite and faithful,
    we can choose a self-adjoint operator $A \in \AAA$ such that $0 < \varphi(A^2) < \infty$ and 
    $A\braket{T\xi}{T\xi}_{\AAA}A \neq 0$. Note that $\xi \cdot A \in \NN(H,\varphi)$, this implies that 
    $\pi_\varphi(T) \neq 0$. 
\end{proof}

Let $\AAA$ be a von Neumann algebra and $\varphi$ a n.s.f. weight on 
$\AAA$. For the rest of this section, $H$ is a pre-Hilbert $\AAA$-$\AAA$ bimodule with a 
distinguished \lq\lq vacuum vector" 
$\Omega\in H$ such that $\braket{\Omega}{\Omega}_{\AAA}=I$ and
$A \cdot \Omega = \Omega \cdot A$ for all $A \in \AAA$ (as we shall be equipped with in the Fock space construction
at the end of this section). Let 
\begin{align}\label{P_def}
    e_{\AAA}(\xi) = \braket{\Omega}{\xi}_{\AAA}\Omega, \quad \xi \in H. 
\end{align}
It is easy to check that $e_{\AAA}(A \cdot \xi \cdot B) = A\cdot e_{\AAA}(\xi) \cdot B$, 
where $A$, $B \in \AAA$.

\begin{proposition} \label{P_props}
    $e_{\AAA}$ is a projection in $\BBB(H)$ whose range is $\AAA \cdot \Omega$.
\end{proposition}

\begin{proof}
    Since $\braket{\xi}{\Omega}_{\AAA}\braket{\Omega}{\xi}_{\AAA} \leq \braket{\xi}{\xi}_{\AAA}$,
    we have $\|e_{\AAA}(\xi)\|_{H} \leq \|\xi\|_H$. It is also clear that
    $\braket{\beta}{e_{\AAA}(\xi)}_{\AAA}= \braket{\beta}{\Omega}_{\AAA}\braket{\Omega}{\xi}_{\AAA}
    =\braket{e_{\AAA}(\beta)}{\xi}_{\AAA}$, thus $e_{\AAA}=e_{\AAA}^*$. For any $A \in \AAA$, we have
    $e_{\AAA}(A \cdot \Omega) = A \cdot \Omega$.
\end{proof}

For the rest of this section, $\MMM \subset \BBB(H)$ is a *-algebra containing $\AAA$ (which we regard as represented on $H$ via its left action). 
Then 
\begin{align}\label{cond_exp_def}
  E_0(T) = \braket{\Omega}{T \cdot \Omega}_{\AAA}, \quad T \in \MMM  
\end{align}
is a conditional expectation from $\MMM$ onto $\AAA$ (\cite[Theorem 4.6.15]{RS}).
It is easy to check that $e_{\AAA}Te_{\AAA} = E_0(T)e_{\AAA}$, $T\in\MMM$.
Also note that $\MMM\Omega = span\{T \cdot \Omega: T \in \MMM\}$ is a pre-Hilbert
$\AAA$-$\AAA$ bimodule and that $L^2(\AAA, \phi)$ is canonically embedded into $H_\phi$ for
any n.s.f. weight $\phi$ on $\AAA$ via $A \in \AAA \mapsto A \cdot \Omega$.

\begin{lemma}\label{project_lemma}
    $e_{\AAA}$ defined in \cref{P_def} extends to an orthogonal projection, 
    still denoted by $e_{\AAA}$, from $H_\varphi$ onto the subspace $L^2(\AAA, \varphi)$.  
    Furthermore, $e_{\AAA} \pi_\varphi(T) e_{\AAA} = \pi_\varphi(E_0(T))e_{\AAA}$, for every $T \in \MMM$.
\end{lemma}

\begin{proof}
    $\braket{\beta}{e_{\AAA} \pi_\varphi(T) e_{\AAA} \xi} = 
    \varphi(\braket{\beta}{\Omega}_{\AAA}E_0(T)\braket{\Omega}{\xi}_{\AAA})
    = \braket{\beta}{\pi_\varphi(E_0(T))e_{\AAA}\xi}$, for all 
    $\xi, \beta \in \NN(H, \varphi)$.
    Thus $e_{\AAA} \pi_\varphi(T) e_{\AAA} = \pi_\varphi(E_0(T))e_{\AAA}$.
\end{proof}

Note that $A \in \AAA \mapsto \pi_\varphi(A)e_{\AAA}$ is a *-isomorphism, namely 
the Gelfand-Naimark-Segal (GNS) representation of $\AAA$ associated with $\varphi$. By \cref{project_lemma}, 
$e_{\AAA} \pi_{\varphi}(\MMM)'' e_{\AAA} = \pi_{\varphi}(\AAA)e_{\AAA}$. 
For any $B \in \pi_{\varphi}(\MMM)''$, let $E_\varphi(B)$ be the unique element in $\AAA$ satisfying 
$\pi_\varphi(E_\varphi(B))e_{\AAA} = e_{\AAA}Be_{\AAA}$.
It is not hard to check that $B \mapsto E_\varphi(B)$ is a normal completely positive
$\AAA$-$\AAA$ bimodule map, i.e., $E_\varphi(\pi_\varphi(A_1)B\pi_\varphi(A_2))
= A_1 E_\varphi(B)A_2$ for $A_1$, $A_2 \in \AAA$. By \cref{project_lemma}
$E_\varphi(\pi_\varphi(T)) = E_0(T)$, $T \in \MMM$.

\begin{lemma}\label{GNS_lemma}
   For any other n.s.f. weight $\phi$ on $\AAA$, $\tilde{\phi}= \phi \circ E_{\varphi}$
   is a normal semifinite weight on $\pi_\varphi(\MMM)''$ such that $\tilde{\phi}(\pi_\varphi(T)) = 
    \phi(E_0(T))$, $T \in \MMM$. If $\MMM\Omega$ is dense in $H$ (w.r.t. the norm $\|\cdot \|_H$), then
    there exists a unitary $U: H_\phi \rightarrow L^2(\pi_\varphi(\MMM)'',  \tilde{\phi})$ such that
    $U^*\pi_{\tilde{\phi}}(\pi_\varphi(T))U = \pi_{\phi}(T)$ for every $T \in \MMM$, where
    $\pi_{\tilde{\phi}}$ is the GNS representation of $\pi_\varphi(\MMM)''$ associated 
    with $\tilde{\phi}$.
\end{lemma}

\begin{proof}
    We claim that $\MMM \Omega \cap \NN(H, \phi)$ is dense in $H_\phi$ w.r.t. the Hilbert space 
    norm. Indeed, let $\xi \in \NN(H,\phi)$ and $K = \braket{\xi}{\xi}_{\AAA}^{1/2}$.
    If $T_m = mK(I+mK)^{-1}$, then $KT_1 \leq KT_2 \leq KT_3 \leq \cdots$ and 
    $KT_m \rightarrow K$ in norm. Note that $\phi(T_m^2)< \infty$.
    Since $\phi$ is normal, for any $\varepsilon > 0$, we can choose $m_0$ such that  
    \begin{align*}
        \|\xi - \xi \cdot T_{m_0}\|_2=
        |\phi(\braket{\xi \cdot (I-T_{m_0})}{\xi \cdot (I-T_{m_0})}_{\MMM})|^{1/2} < \varepsilon. 
    \end{align*}
    Let $C = \phi(T_{m_{0}}^2)$ and 
    $\beta \in \MMM \Omega$ such that $\|\braket{\beta - \xi}{\beta - \xi}_{\AAA}\| < \varepsilon^2/C$. 
    Therefore
    \begin{align*}
        \|\beta \cdot T_{m_0} - \xi \cdot T_{m_0}\|_2 = 
        \phi(T_{m_0}\braket{\beta - \xi}{\beta - \xi}_{\MMM}T_{m_0})^{1/2} < \varepsilon.
    \end{align*}
    Thus $\MMM \Omega \cap \NN(H, \phi)$ is a dense subspace of $H_\phi$. 

    Similarly, for $B \in \pi_\varphi(\MMM)''$ satisfying $\phi(E_\varphi(B^*B)) < \infty$,
    let $K =  E_\varphi(B^*B)^{1/2}$. Repeat the argument above, for any $\varepsilon>0$, we find
    a positive operator $T_{m_0} \in \AAA$ such that $\phi(T_{m_0}^2) < \infty$ and 
    $\phi((I-T_{m_0})E_\varphi(B^*B)(I-T_{m_0}))^{1/2}< \varepsilon$.
    Since $\phi(T_{m_0}E_\varphi(\cdot) T_{m_0})$ is a normal functional on $\pi_\varphi(\MMM)''$,
    we can find $A \in \MMM$ such that 
    $\phi(T_{m_0}E_\varphi([\pi_\varphi(A)-B]^*[\pi_\varphi(A)-B]) T_{m_0})^{1/2} \leq \varepsilon$.
    Thus $\{\pi_\varphi(T): \phi(E(T^*T))< \infty, T \in \MMM\}$ is 
    dense in $L^2(\pi_\varphi(\MMM)', \tilde{\phi})$.
    
    Since $\braket{A\Omega}{B\Omega} = \phi(E_0(A^*B)) = \phi(E_\varphi(\pi_\varphi(A^*B)))= 
    \braket{\pi_\varphi(A)}{\pi_\varphi(B)}$, for $A\Omega$, $B\Omega \in \MMM \Omega \cap \NN(H, \phi)$.
    The map $A\Omega \mapsto \pi_\varphi(A) \in L^2(\pi_\varphi(\MMM)'',\tilde{\phi})$ 
    can be extended to a unitary from $H_\phi$ to $L^2(\pi_\varphi(\MMM)'', \tilde{\phi})$. 
    Finally, note that
    \begin{align*}
        \pi_{\tilde{\phi}}(\pi_\varphi(T))U A\Omega = \pi_\varphi(TA) = U\pi_\phi(T)A\Omega, \quad
        A\Omega \in \MMM\Omega \cap \NN(H, \phi).
    \end{align*}
\end{proof}

\begin{theorem}\label{rep_ind_weight}
   Let $\varphi$ and $\phi$ be two n.s.f. weights on $\AAA$. 
    If $\MMM\Omega$ is dense in $H$ (w.r.t. the norm $\|\cdot\|_H$), then
    the map $\pi_{\varphi}(T) \mapsto \pi_{\phi}(T)$, $T \in \MMM$, extends
    to a *-isomorphism between the two von Neumann algebras $\pi_\varphi(\MMM)''$ and
    $\pi_{\phi}(\MMM)''$.
\end{theorem}

\begin{proof}
    By \cref{GNS_lemma}, we have normal *-homomorphisms $\rho_1:\pi_\varphi(\MMM)'' \rightarrow
    \pi_\phi(\MMM)''$ and $\rho_2:\pi_\phi(\MMM)'' \rightarrow \pi_\varphi(\MMM)''$ such that
    $\rho_1(\pi_\varphi(T)) = \pi_\phi(T)$ and 
    $\rho_2(\pi_\phi(T)) = \pi_\varphi(T)$ for every $T \in \MMM$. Thus $\rho_1$ is a 
    *-isomorphism and $\rho_2 = \rho_1^{-1}$.
\end{proof}

\subsection{Pre-Hilbert bimodules with normal left action}

\begin{definition}
   Let $H$ be a pre-Hilbert $\AAA$-$\AAA$ bimodule. We say that the left action is normal if 
   the map $A \in \AAA \mapsto \braket{\xi}{A \cdot \xi}_{\AAA}$ is normal for each $\xi \in H$,
    i.e., $A \mapsto \braket{\xi}{A \cdot \xi}_{\AAA}$ is continuous from
    $(\AAA)_1$ to $\AAA$ both in their weak-operator topology, where $(\AAA)_1$ denotes the unit ball of $\AAA$.
\end{definition}

\begin{remark}
    \begin{enumerate}
        \item In general, $\pi_\varphi(\AAA)'' \neq \pi_\varphi(\AAA)$. However, if the left 
            action is normal, then $\pi_\varphi(\AAA)'' = \pi_\varphi(\AAA)$ (see 
            \cite[Proposition 7.1.15]{KRII}).
        \item For every $\beta,\xi\in H$, the map $A \in \AAA \mapsto \braket{\beta}{A \cdot \xi}_{\AAA}$ is continuous
            from $(\AAA)_1$ to $\AAA$ both in their weak-operator topology.
        \item For every $\xi\in H$, the map $A \in \AAA \mapsto \braket{\xi}{A \cdot \xi}_{\AAA}$
       is completely positive. Indeed, let $(K_{ij})_{i,j}$ be a positive element in $M_n(\AAA)$,
       then we only need to check that $\sum_{l=1}^n(\braket{K_{li}\cdot\xi}{K_{lj}\cdot\xi}_{\AAA})_{i,j} 
       \geq 0$, and this is true by \cite[Lemma 4.2]{CL}.
    \end{enumerate}
\end{remark}

The following easy fact is used implicitly in the paper and  
the proof is left as an exercise for the reader.

\begin{proposition}
   Let $\DDD$ be a subset of $H$. Assume that $\AAA \cdot \DDD \cdot \AAA= 
    span\{A \cdot \beta \cdot B: A, B \in \AAA, \beta \in \DDD\}$ is dense in $H$. Then 
   the left action is normal if and only if $A \in \AAA \mapsto \braket{\beta}{A \cdot \beta}_{\AAA}$
   is normal for each $\beta \in \DDD$.
\end{proposition}

%
%

\begin{proposition}
    Let $H$ be a pre-Hilbert $\AAA$-$\AAA$ bimodule. If the left action is normal, then
    $A \mapsto \braket{\xi}{A \cdot \xi}_{\AAA}$ is continuous from
    $(\AAA)_1$ to $\AAA$ both in their strong-operator topology.
\end{proposition}

\begin{proof} 
    Assume that $A_\alpha$ tends to $0$ in the strong-operator topology, 
    then $A_\alpha^*A_\alpha \rightarrow 0$ in the weak-operator topology. 
    Therefore $\braket{A_\alpha \cdot \xi}{A_\alpha \cdot \xi}_{\AAA}$ 
    converges to $0$ in the weak-operator topology.
    Since $\braket{A_\alpha \cdot \xi}{A_\alpha \cdot \xi}_{\AAA} \geq 1/\|\xi\|_{H}^2
    \braket{A_\alpha \cdot \xi}{\xi}_{\AAA} \braket{\xi}{A_\alpha \cdot \xi}_{\AAA}$, by
    \cite[Lemma 5.3]{CL}, we have that $\braket{A_\alpha \cdot \xi}{\xi}_{\AAA}$ tends to $0$ in
    the strong-operator topology.
\end{proof}

\begin{proposition}\label{inner_definite_pro}
   Let $H_1$ and $H_2$ be two pre-Hilbert $\AAA$-$\AAA$ bimodules. If the left action on $H_2$
    is normal, then $H_1 \otimes_\AAA H_2$, the algebraic tensor product over $\AAA$, 
    is a pre-Hilbert $\AAA$-$\AAA$ bimodule with $\AAA$-valued inner product given on simple tensors by
    \begin{align*}
        \braket{\xi_1 \otimes \beta_1}{\xi_2 \otimes \beta_2}_{\AAA} = 
        \braket{\beta_1}{\braket{\xi_1}{\xi_2}_{\AAA}\cdot \beta_2}_{\AAA}, \quad \xi_1, \xi_2 \in H_1,\, 
        \beta_1 \beta_2 \in H_2.
    \end{align*}
\end{proposition}

\begin{proof}
    We only need to show that if $\zeta= \sum_{i} \xi_i \otimes \beta_i$ satisfying 
    $\braket{\zeta}{\zeta}_\AAA =0$, then $\zeta =0$ in $H_1 \otimes_\AAA H_2$. 
    
    Since $\braket{\beta}{T \cdot \beta}_{\AAA}=0$ where $\beta=(\beta_1, \ldots, \beta_n)
    \in H_2^n$ and $T = (\braket{\xi_i}{\xi_j}_{\AAA})_{i,j} \in M_n(\AAA)$,
    $T^{1/2}\cdot \beta =0$. Thus $\braket{\beta}{T^{1/2} \cdot \beta}_{\AAA}=0$ and this implies
    $T^{1/4} \cdot \beta = 0$. By induction, we have $T^{1/2^n} \cdot \beta =0$, $n = 1, 2, \ldots$.
    Since $T^{1/2^n}$ tends to the range projection $P=(p_{ij})_{i,j}$ of 
    $T$ in the strong-operator topology, we have that $\braket{\beta}{T^{1/2^n} \cdot \beta}_{\AAA}$
    converges to $\braket{\beta}{P\cdot \beta}_{\AAA}$ in the strong-operator topology. Thus
    $P\cdot \beta = 0$ and $\sum_{k=1}^n p_{i,k} \cdot \beta_k = 0$, $i = 1, \ldots, n$.

    Note that
    \begin{align*}
        0=(I-P)(\braket{\xi_k}{\xi_l}_\AAA)_{k,l}(I-P)=  
        (\braket{\xi_i - \sum_k \xi_k \cdot p_{k,i}}{\xi_j - \sum_{l} \xi_l \cdot p_{l,j}}_{\AAA})_{i,j}.
    \end{align*}
    Therefore $\xi_i = \sum_k \xi_k \cdot p_{k,i}$, $i = 1,\ldots, n$.
    Then it is clear that
    \begin{align*}
        \zeta = \sum_{i} \xi_i \otimes \beta_i = \sum_{i,k} \xi_k \cdot p_{k,i} \otimes \beta_i 
        = \sum_{i,k} \xi_k \otimes p_{k,i} \cdot \beta_i =0.
    \end{align*}
\end{proof}

\subsection{Tomita pre-Hilbert bimodules and Fock space construction}

Throughout the rest of this section, $\AAA$ is a semifinite von Neumann algebra and $\tau$ is a
n.s.f. tracial weight on $\AAA$. Recall that $\NN(\AAA, \tau) = \{A \in \AAA:
\tau(A^*A) < \infty\}$ is an ideal. Let $H$ be a pre-Hilbert $\AAA$-$\AAA$ bimodule $H$ with normal 
left action of $\AAA$. Note that 
\begin{align*}
    \braket{\xi}{\beta \cdot A} = \tau(\braket{\xi}{\beta}_\AAA A) = \tau((\braket{\beta}{\xi}_\AAA A^*)^*) =  
    \braket{\xi \cdot A^*}{\beta}.
\end{align*}
Thus the right action $\xi \mapsto \xi \cdot B$, $\xi \in \NN(H,\tau)$, gives
a (normal) *-representation of the opposite algebra $\AAA^{op}$ on the Hilbert space $H_\tau$ (see \cref{notation_Hvarphi}). 
Moreover, $H_\tau$ is an $\AAA$-$\AAA$ bimodule with scalar-valued inner product (also called correspondence in \cite{ACCG}, \cite{L-indexII}, \cite{LS}). 

\begin{definition}\label{tomita_bimodule_def}
   Let $H$ be a pre-Hilbert $\AAA$-$\AAA$ bimodule with normal left action of $\AAA$.
   We say that $H$ is a \textbf{Tomita $\AAA$-$\AAA$ bimodule} if 
   \begin{enumerate}
       \item $\NN(H,\tau) = H$, and in this case we regard $H$ as a dense subspace of $H_\tau$, 
       \item $H$ admits an involution $S$ such that 
           $S(A \cdot \xi \cdot B) = B^* \cdot S(\xi) \cdot A^*$ and $S^2(\xi) = \xi$, for every $\xi \in H$ and $A,B\in \AAA$,
        \item $H$ admits a complex one-parameter group $\{U(\alpha):\alpha \in \C\}$ of
              isomorphisms satisfying the following properties: 
            \begin{enumerate}
                \item For every $\xi$, $\beta \in H$, the map $\alpha \mapsto U(\alpha)\beta$  
                    is continuous with respect to $\|\cdot\|_{H}$
                    and the function $\alpha \mapsto \braket{\xi}{U(\alpha) \beta}$ is entire,
                \item $S(U(\alpha)\xi)=U(\overline{\alpha})S(\xi)$,
                \item $\braket{\xi}{U(\alpha) \beta} = 
                    \braket{U(-\overline{\alpha}) \xi}{\beta}$,
                \item $\braket{S(\xi)}{S(\beta)} = \braket{\beta}{U(-i)\xi}$.
            \end{enumerate}
   \end{enumerate}
\end{definition}

\begin{remark}
It is clear that every Tomita algebra is a Tomita $\C$-$\C$ bimodule. The same definition of Tomita bimodule could be given for $\AAA$ not necessarily semifinite and for $\tau$ not necessarily tracial. In order to derive the following properties, and for the purpose of this paper, we stick to the semifinite case.
\end{remark}

Note that $\|U(\alpha)\beta\|_2^2 = 
\braket{\beta}{U(2i Im \alpha)\beta}$ is a continuous
function of $\alpha$ for each $\beta \in H$, hence it is locally bounded. 
By \cite[A.1]{TKII}, $\alpha \in \C \mapsto U(\alpha)\beta \in H_\tau$ 
is entire in the (Hilbert space) norm for every $\beta \in H$. Furthermore, we have the
following fact.

\begin{lemma}\label{entire_in_A}
    Let $H$ be a Tomita $\AAA$-$\AAA$ bimodule. For any $\xi$, $\beta \in H$,
    the map $\alpha\mapsto \braket{\xi}{U(\alpha)\beta}_{\AAA} \in \AAA$ is entire in the (operator) norm.
\end{lemma}

\begin{proof}
By \cref{tomita_bimodule_def}(3)(a), the map is bounded on any compact subset of $\C$.
For every $A$, $B \in \NN(\AAA,\tau)$, $\tau(A^* \braket{\xi}{U(\alpha)\beta}_{\AAA} B) = \braket{\xi\cdot AB^*}{U(\alpha)\beta}$ is 
entire. By \cite[A.1]{TKII} we have the statement.
\end{proof}

\begin{remark}
    Let $\overline{H}$ be the closure of $H$ with respect to $\| \cdot \|_H$. 
    Using a similar argument as in the proof of
    \cref{entire_in_A}, we can show that the map $\alpha \mapsto U(\alpha)\xi \in \overline{H}$ is
    entire for every $\xi \in H$.
\end{remark}

Now, $U(\alpha)$ and $S$ can be viewed as densely defined operators with 
domain $H \subset H_\tau$, and we have the following result.

\begin{proposition}\label{core_prop}
    Let $H$ be a Tomita $\AAA$-$\AAA$ bimodule. Then $S$ is preclosed and we use the same symbol 
    $S$ to denote the closure. If $J \Delta^{1/2}$ is the polar decomposition of $S$,
    then $J^2=I$, $J\Delta J= \Delta^{-1}$, $U(\alpha) = \Delta^{i\alpha}|_H$ 
    and $H$ is a core for each $\Delta^{i\alpha}$, $\alpha \in \C$.
\end{proposition}

\begin{proof}
    By \cref{tomita_bimodule_def}(3)(d), $H$ is in the domain of $S^*$. Thus
    $S^*$ is densely defined, $S$ is preclosed and $U(-i) \subset \Delta$.
    Hence \cite[Lemma 1.5]{TKII} implies $J^2=I$, $J\Delta J= \Delta^{-1}$.
    By the same argument used in the proof of \cite[Theorem 2.2(ii)]{TKII}, 
    it can be shown that the closure of $U(it)$ is 
    self-adjoint for every $t \in \R$, $U(\alpha) = \Delta^{i\alpha}|_H$ and  
    $H$ is a core for each $\Delta^{i\alpha}$.

    For the convenience of the reader, we sketch the proof. By condition (3)(c), 
    $\{U(t)\}_{t \in \R}$ can be extended to a one-parameter unitary group on $H_\tau$.
    By Stone's theorem, there exists a non-singular (unbounded) self-adjoint positive operator $K$
    such that $K^{it}|_{H} = U(t)$. By \cite[VI Lemma 2.3]{TKII}, $H \subset \DDD(K^{i \alpha})$ for 
    every $\alpha \in \C$ and $K^{i \alpha}|_{H} = U(\alpha)$. Thus $K|_H = \Delta|_H$.
    By \cite[VI, Lemma 1.21]{TKII}, $H$ is a core for $K^{\alpha}$, $\alpha \in \C$.
    Thus $K \subseteq \Delta$. Since $\Delta^* = \Delta$ and $K^* = K$, we have $K = \Delta$.
\end{proof}

\begin{remark}
    By \cref{core_prop}, $S|_{H} = JU(-i/2)$, thus $JH = H$. 
\end{remark}

%
%

\begin{lemma}\label{bimodule_maps_prop}
    Let $H$ be a Tomita $\AAA$-$\AAA$ bimodule. Then 
    \begin{align*}
        &S(A \cdot \xi \cdot B) = B^* \cdot S(\xi) \cdot A^*,\quad
        S^*(A \cdot \beta \cdot B) = B^* \cdot S^*(\beta) \cdot A^*,\\
        &\Delta^{i \alpha}(A \cdot \zeta \cdot B) = A \cdot \Delta^{i \alpha}(\zeta) \cdot B, \quad
        J(A \cdot \zeta \cdot B) = B^* \cdot J(\zeta) \cdot A^*.
    \end{align*}
    where $\xi$, $\beta$, $\zeta$ are in the domains of $S$, $S^*$ and $\Delta^{i\alpha}$ 
    respectively and $A$, $B \in \AAA$.  
\end{lemma}

\begin{proof}
    By \cref{core_prop}, we only need to show that the equations hold for $\xi \in H$. 
    Note that 
    $\braket{A\cdot \beta \cdot B}{S(\xi)}=
      \tau(\braket{\beta}{A^*\cdot S(\xi)\cdot B^*}_{\AAA})
        =\braket{\xi}{B^*\cdot S^*(\beta) \cdot A^*}$.
    This implies $\Delta^{i \alpha}(A \cdot \xi \cdot B) = A \cdot \Delta^{i \alpha}(\xi) \cdot B$.
    Thus $J(A \cdot \xi \cdot B) = \Delta^{1/2}S(A \cdot \xi \cdot B)=
    B^* \cdot J(\xi) \cdot A^*$.
\end{proof}

\begin{lemma}\label{c_hold_for_mod}
   Let $H$ be a Tomita $\AAA$-$\AAA$ bimodule. Then for any $\xi$, $\beta \in H$, we have
   \begin{align*}
       \braket{\xi}{U(\alpha)\beta}_\AAA = \braket{U(-\overline{\alpha})\xi}{\beta}_\AAA.
   \end{align*}
\end{lemma}

\begin{proof}
    By \cref{bimodule_maps_prop}, we have 
    \begin{align*}
        \tau(\braket{\xi}{U(\alpha)\beta}_\AAA A)=  \tau(\braket{\xi}{U(\alpha)(\beta \cdot A)}_\AAA)
        =\tau(\braket{U(-\overline{\alpha}) \xi}{\beta}_\AAA A), \quad \forall A \in \AAA.
    \end{align*}
\end{proof}

\begin{lemma}\label{tensor_pro_t}
    Suppose that $H_1$ and $H_2$ are two Tomita $\AAA$-$\AAA$ bimodules. Then
    \begin{align*}
        S(\xi_1 \otimes \xi_2)=S_2(\xi_2)\otimes S_1(\xi_1), \quad
        U(\alpha)(\xi_1 \otimes \xi_2) = U_1(\alpha)(\xi_1) \otimes U_2(\alpha)(\xi_2), 
    \end{align*}
    define a conjugate linear map $S$ and a one-parameter group $\{U(\alpha)\}_{\alpha \in \C}$,
    where $S_i$ and $U_i(\alpha)$ are the involution and one-parameter group for $H_i$, $i=1,2$, and 
    $\xi_1 \otimes \xi_2 \in H_1 \otimes_{\AAA} H_2$.
    Furthermore $\{U(\alpha)\}$ satisfies the conditions (3)(a) and (3)(d) in \cref{tomita_bimodule_def}.
\end{lemma}

\begin{proof}
By \cref{inner_definite_pro}, \cref{bimodule_maps_prop}, $S$ and $U(\alpha)$ are well-defined.
It is easy to see that $\alpha \mapsto U(\alpha)(\xi_1 \otimes \xi_2)$ is continuous.
By \cref{bimodule_maps_prop},  
$\braket{\xi_1 \otimes \xi_2}{U_1(\alpha)\beta_1 \otimes U_2(\alpha)\beta_2}
    = \braket{U_2(-\overline{\alpha})\xi_2}{\braket{\xi_1}{U_1(\alpha)\beta_1}_{\AAA}\beta_2}$.
Then \cref{entire_in_A} implies that
$\{U(\alpha)\}$ satisfies the condition (3)(a) in \cref{tomita_bimodule_def}.

For $\xi_i$ and $\beta_i \in H_i$, $i=1,2$, we have
    \begin{align*}
        \braket{S_2(\xi_2) \otimes S_1(\xi_1)}{S_2(\beta_2) \otimes S_1(\beta_1)}
        =&\; \tau(\braket{S_2(\xi_2)}{S_2(\beta_2)}_\AAA \braket{\beta_1}{U_1(-i)\xi_1}_\AAA)\\
        = \tau(\braket{\beta_2}{\braket{\beta_1}{U_1(-i)\xi_1}_\AAA U_2(-i)\xi_2}_{\AAA})
        =& \braket{\beta_1 \otimes \beta_2}{U_1(-i)\xi_1\otimes U_2(-i)\xi_2},
    \end{align*}
    i.e.,
    \begin{align*}
      \braket{S(\xi_1 \otimes \xi_2)}{S(\beta_1 \otimes \beta_2)}=
        \braket{\beta_1 \otimes \beta_2}{U(-i)(\xi_1\otimes\xi_2)}. 
    \end{align*}
\end{proof}

\begin{proposition}\label{tensor_tomita_op}
Given a Tomita $\AAA$-$\AAA$ bimodule $H$, then $H^{\otimes_{\AAA}^n}$, $n\geq 1$, is also a Tomita
$\AAA$-$\AAA$ bimodule with the $\AAA$-valued inner product given by linear extension of  
\begin{align*}
    \braket{\xi_1 \otimes \cdots \otimes \xi_n}{\beta_1 \otimes \cdots \otimes \beta_n}_{\AAA}
    =\braket{\xi_n}{
        \langle \cdots\braket{\xi_2}{\braket{\xi_1}{\beta_1}_{\AAA}\beta_2}_{\AAA} \cdots 
        \rangle_{\AAA}\beta_n}_{\AAA}, \xi_i, \beta_i \in H,
\end{align*}
and the involution $S_n$ and one-parameter group $\{U_n(\alpha)\}_{\alpha \in \C}$ are respectively given by 
$S_n(\xi_1 \otimes \cdots \otimes \xi_n) = S(\xi_n) \otimes \cdots \otimes S(\xi_1)$ and 
$U_n(\alpha)(\xi_1 \otimes \cdots \otimes \xi_n) = U(\alpha)(\xi_1) \otimes \cdots 
\otimes U(\alpha)(\xi_n)$. 
\end{proposition}

\begin{proof}
By \cref{inner_definite_pro}, $H^{\otimes_{\AAA}^n}$ is a pre-Hilbert $\AAA$-$\AAA$ bimodule with normal left
action. Conditions (1), (2), (3)(b) in the definition of Tomita bimodule are trivial to check. By \cref{c_hold_for_mod}, (3)(c) is also clear.
By \cref{tensor_pro_t} and induction on $n$, we get (3)(a) and (3)(d). 
\end{proof}

Let $S_n = J_n \Delta_n^{1/2}$ be the polar decomposition of $S$, $n \geq 1$. By \cref{core_prop}, $H^{\otimes^n_{\AAA}}$ is a common core
for $\Delta_n^{t}$, $t \in \R$, and $\Delta_n^{t}(\xi_1 \otimes \cdots \otimes \xi_n) 
= \Delta^{t}(\xi_1) \otimes \cdots \otimes \Delta^{t}(\xi_n)$.
Therefore $J_n(\xi_1 \otimes \cdots \otimes \xi_n) 
= \Delta_n^{1/2}S_n(\xi_1 \otimes \cdots \otimes \xi_n)=J(\xi_n) \otimes \cdots \otimes J(\xi_1)$.
We recall the following definition, see, e.g., \cite[Section 4.6]{RS} for the same definition given for Hilbert bimodules over a C$^*$-algebra.

\begin{definition}\label{fockspace_def}
Given a Tomita $\AAA$-$\AAA$ bimodule $H$, 
let $\FFF(H) = \oplus_{n\geq 0} H^{\otimes^n_{\AAA}}$ be the associated \textbf{Fock space}, 
where $H^{\otimes^0_{\AAA}} = \AAA$ with the $\AAA$-valued inner product
$\braket{A_1}{A_2}_{\AAA} = A_1^* A_2$ for $A_1$, $A_2 \in \AAA$.
\end{definition}

By \cref{inner_definite_pro}, $\FFF(H)$ is a pre-Hilbert $\AAA$-$\AAA$ bimodule. 
It is clear that the identity $I$ of 
$\AAA$ is a \lq\lq vacuum vector" for $\FFF(H)$, and we will use $\Omega$ to denote this vector 
as in the beginning of this section.

For $\xi \in H$, the creation and annihilation operators $L(\xi)$, $L^*(\xi)$ are defined by
\begin{align*}
    &L(\xi) A = \xi \cdot A, \quad 
    L(\xi) \beta_1 \otimes \cdots \beta_n = \xi \otimes \beta_1 \otimes \cdots \otimes \beta_n,\\
    &L^*(\xi) A = 0, \qquad L^*(\xi) \beta_1 \otimes \cdots \beta_n = 
    \braket{\xi}{\beta_1}_{\AAA} \cdot \beta_2 \otimes \cdots \otimes \beta_n,
\end{align*}
where $A \in \AAA$ and $\beta_i \in H$, $i = 1, \ldots, n$.

\begin{definition}\label{phialgs_def}
Let $\Phi(H)$ be the $^*$-subalgebra of $\BBB(\FFF(H))$ generated by $\AAA$ (acting on $\FFF(H)$ from the left) and 
$\{\Gamma(\xi)=L(\xi) + L^*(S(\xi)): \xi \in H\}$.

Let $\Phi(H)''$ be the von Neumann algebra generated by $\pi_\tau(\Phi(H))$ on the Hilbert space completion $\FFF(H)_\tau$ 
(see \cref{notation_Hvarphi}) and let $\Phi(H)'$ be the commutant of $\Phi(H)''$. 
\end{definition}

\begin{proposition}\label{generate_prop}
    $\Phi(H)\Omega = \FFF(H)$.  
\end{proposition}

\begin{proof}
    By induction, it is not hard to check that 
    $\AAA \oplus \oplus_{n \geq 1}\{\xi_1 \otimes \cdots \otimes \xi_n: \xi_i \in H, i = 1,\ldots,n\} \subset  
    \Phi(H)\Omega$.
\end{proof}

The GNS space $L^2(\AAA, \tau)$ can be viewed as the closure of the subspace spanned by 
$\NN(\AAA, \tau)$ in $\FFF(H)_\tau$. 
Let $e_{\AAA}$ be the orthogonal projection from $\FFF(H)_\tau$ onto $L^2(\AAA, \tau)$.
By \cref{project_lemma}, $T \in \Phi(H)'' \mapsto e_{\AAA}Te_{\AAA}|_{L^2(\AAA, \tau)} \in 
\pi_{\tau}(\AAA)e_{\AAA}|_{L^2(\AAA, \tau)}$,
induces a normal completely positive map $E_\tau$ from $\Phi(H)''$ onto $\AAA$.
Since $A \in \AAA \mapsto \pi_\tau(A)$ is a *-isomorphism, $E = \pi_\tau \circ E_\tau$ is a
conditional expectation from $\Phi(H)''$ onto $\pi_\tau(\AAA)$.

By \cref{tensor_tomita_op}, 
$\NN(\FFF(H),\tau) = \NN(\AAA, \tau) \oplus \oplus_{n \geq 1}H^{\otimes^{n}_{\AAA}}$ is a 
Tomita $\AAA$-$\AAA$ bimodule with involution $S_0 \oplus \oplus_{n\geq 1} S_n$ and 
one-parameter group $Id \oplus \oplus_{n\geq 1}U_n(\alpha)$, where $S_0=J_0$
is the modular conjugation associated with $\tau$ on the Hilbert space $L^2(\AAA,\tau)$ and 
$Id$ is the identity map on $\AAA$.
Let $\JJJ = J_0 \oplus \oplus_{n \geq 1} J_n$. By \cref{tensor_tomita_op},
the polar decomposition of the closure of $S_0 \oplus \oplus_{n\geq 1} S_n$ (which we shall denote 
again by $S_0 \oplus \oplus_{n\geq 1} S_n$) is
$\JJJ(I \oplus \oplus_{n \geq 1} \Delta_n^{1/2})$.

\begin{proposition}\label{faith_ce}
    For each $\xi \in H$, $A \in \AAA$, we have $\JJJ A \JJJ B = BA^*$,
    $\JJJ A \JJJ \beta_1 \otimes \cdots \otimes \beta_n 
    = \beta_1 \otimes \cdots \otimes (\beta_n \cdot A^*)$,
    $\JJJ \Gamma(\xi) \JJJ B = B \cdot J(\xi)$ and  
    \begin{align}\label{right_vect_act}
        \JJJ \Gamma(\xi) \JJJ \beta_1 \otimes \cdots \otimes \beta_n = 
        \beta_1 \otimes \cdots \beta_n \otimes J(\xi) + 
        \beta_1 \otimes \cdots  \beta_{n-1} \cdot \braket{J(\beta_n)}{S(\xi)}_{\AAA},
    \end{align}   
    where $B \in \NN(\AAA, \tau)$. Thus $\JJJ\Phi(H)\JJJ \subset \Phi(H)'$ and 
    the conditional expectation
    \begin{align*}
        E= \pi_\tau \circ E_\tau: \Phi(H)'' \rightarrow \pi_\tau (\AAA) 
    \end{align*}
    is faithful. 
\end{proposition}

\begin{proof}
    All the equations can be checked by easy computation. Let $B \in \NN(\AAA,\tau)$ and $\beta \in H$. 
    Note that $S|_{H} = JU(-i/2)$, \cref{c_hold_for_mod} implies
    \begin{align*}
        \Gamma(\beta)\JJJ \Gamma(\xi) \JJJ (B) &= \beta \otimes B \cdot J(\xi) + 
        \braket{S(\beta)}{B \cdot J(\xi)}_{\AAA}\\
        &= \beta \cdot B \otimes J(\xi) + 
        \braket{J(\beta \cdot B)}{S(\xi)}_{\AAA}=\JJJ \Gamma(\xi) \JJJ \Gamma(\beta)(B).
    \end{align*}
    Similar computation shows that $\JJJ\Phi(H)\JJJ \subset \Phi(H)'$.  
    
    Note that \cref{generate_prop} and \cref{right_vect_act} imply that 
    $\Phi(H)' L^2(\AAA,\tau) = \{T \zeta: \zeta \in L^2(\AAA,\tau), T \in \Phi(H)'\}$ is dense in 
    $\FFF(H)_\tau$. Since $E_\tau(T^*T)=0$, $T \in \Phi(H)''$, if and only if 
    $Te_{\AAA} =0$, then $E_\tau$ and consequently $E$ are faithful.
\end{proof}

Since $E_\tau$ is faithful, $\tilde{\tau} = \tau \circ E_\tau$ is a n.s.f. weight on $\Phi(H)''$. 
By \cref{GNS_lemma}, the map $A \Omega \in \NN(\FFF(H),\tau) \mapsto \pi_\tau(A) 
\in L^2(\Phi(H)'', \tilde{\tau})$ extends to a unitary $U$ 
such that $U^* \pi_{\tilde{\tau}}(\pi_\tau(T))U = \pi_\tau(T)$ for every $T \in \Phi(H)$, where
$\pi_{\tilde{\tau}}$ is the GNS representation of $\Phi(H)''$ associated with $\tilde{\tau}$. 

Let $\SI$ be the involution in the Tomita theory associated with $\tilde{\tau}$.
We claim that $U \NN(\FFF(H), \tau)$ is a core for $\SI$. Indeed, let $T$ be a 
self-adjoint operator in $\Phi(H)''$ such that $\tau(E_\tau(T^2)) < \infty$. 
For any $\varepsilon >0$, there is a $\delta > 0$ such that 
$\tau(E(T^2)(I-P)) < \varepsilon$, where $P$ is the spectral projection of $E(T^2)$ 
corresponding to $[\delta, \|T^2\|]$. Note that $P \in \NN(\AAA, \tau)$,
and $\pi_\tau(P \Phi(H)P)$ is dense in $\pi_\tau(P)\Phi(H)'' \pi_\tau(P)$. Therefore there is a 
self-adjoint operator $K \in \pi_\tau(P \Phi(H)P)$ such that 
$\tau(E_\tau((K-T\pi_\tau(P))^*(K-T\pi_\tau(P)))) < \varepsilon$.
Note that $P \Phi(H)P \Omega \subset \NN(\FFF(H), \tau)$.
We have that the graph of $\SI |_{U \NN(\FFF(H), \tau)}$ is dense in the graph of $\SI$. 

By the definition of $\SI$, it is clear that $U^* \SI U (B)= B^*$ and $U^* \SI U(\xi)= 
\Gamma(\xi)^*\Omega= S(\xi)$, where $B \in \NN(\AAA, \tau)$ and $\xi \in H$. Note that
\begin{align*}
    \xi_1 \otimes \xi_2 = \Gamma(\xi_1) \Gamma(\xi_2)\Omega -\braket{S(\xi_1)}{\xi_2}_{\AAA}.
\end{align*}
Since 
$U^*\SI U(\Gamma(\xi_1) \Gamma(\xi_2)\Omega)=  \Gamma(\xi_2)^*
\Gamma(\xi_1)^* \Omega = S(\xi_2) \otimes S(\xi_1)
+ \braket{\xi_2}{S(\xi_1)}_{\AAA}$,
we have $\SI(\xi_1 \otimes \xi_2)= S_2(\xi_1 \otimes \xi_2)$.
By similar computation and induction on the degree of the tensor, it is not hard to 
check that $U^* \SI U(\xi_1 \otimes \cdots \otimes \xi_n) = S_n(\xi_1 \otimes \cdots \otimes \xi_n)$,
for any $\xi_i \in H$, $i = 1, \ldots, n$. Thus $S_0 \oplus \oplus_{n \geq 1} S_n = U^* \SI U$,
where $S_0=J_0$. 

Therefore we have the following result which shows how the Tomita structure of
$H$ determines the modular objects of the associated von Neumann algebra $\Phi(H)''$. This is the main
result of this section and it is crucial for the determination of the type of factors constructed in
the next section.

\begin{theorem}\label{module_auto}
    $\Phi(H)' = \JJJ \Phi(H)'' \JJJ$. For each $\xi \in H$, 
    $\sigma^{\tau \circ E_\tau}_t(\Gamma(\xi)) = \Gamma(U(t)\xi)$ for every $t\in\R$.
\end{theorem}

\begin{proof}
    For every $t\in\R$, $A \in \NN(\AAA, \tau)$ and $\beta_1 \otimes \cdots \beta_n \in H^{\otimes^n_{\AAA}}$, 
    we have
    \begin{align*}
        \sigma^{\tau \circ E_\tau}_t(\Gamma(\xi)) \beta_1 \otimes \cdots \beta_n=
        &(I \oplus \oplus_{n \geq 1} \Delta_n)^{it}\Gamma(\xi)(I \oplus \oplus_{n \geq 1} \Delta_n)^{-it}
    \beta_1 \otimes \cdots \beta_n\\
        =& U(t)\xi \otimes \beta_1 \otimes \cdots \beta_n + 
        \braket{S(U(t)\xi)}{\beta_1}_{\AAA} \beta_2 \otimes \cdots \otimes \beta_n,
    \end{align*} 
    and $\sigma^{\tau \circ E_\tau}_t(\Gamma(\xi)) A = U(t)\xi \cdot A$.
\end{proof}

\subsection{Fock space and amalgamated free products}

Let $\{(H_i, S_i, \{U_i(\alpha)\}_{\alpha \in \C})\}_{i\in I}$ be a family of Tomita 
$\AAA$-$\AAA$ bimodules. Let $H =  \oplus_{i \in I} H_i = \{(\xi_i)_{i\in I}: \xi_i \in H_i,$
$\xi_i$ is zero for all but for a finite number of indices$\}$.
It is obvious that $H$ admits an involution $S= \oplus_i S_i$ and a one-parameter
group $\{U(\alpha)=\oplus_{i}U_i(\alpha)\}_{\alpha \in \C}$. If we define the $\AAA$-valued inner product by 
$\braket{(\xi_i)}{(\beta_i)}_{\AAA} = \sum_i \braket{\xi_i}{\beta_i}_{\AAA}$, then
$H$ is a Tomita $\AAA$-$\AAA$ bimodule with left and right action given by 
$A \cdot (\xi_i) \cdot B = (A\cdot \xi_i \cdot B)$. 

$\FFF(H_i)$ can be canonically embedded into $\FFF(H)$ as an 
$\AAA$-$\AAA$ sub-bimodule. We view $\FFF(H_i)_\tau$ as a Hilbert subspace of 
$\FFF(H)_\tau$. Let $e_{\AAA}$ and $e_i$ be the projection from 
$\FFF(H)_\tau$ onto $L^2(\AAA, \tau)$ and $\FFF(H_i)$ respectively. 
It is clear that, for every $i\in I$, $e_i \in (\AAA \cup \{\Gamma(\xi): \xi \in H_i\})' \cap \BBB(\FFF(H)_\tau)$.
By \cref{faith_ce}, the central carrier of $e_i \in (\AAA \cup \{\Gamma(\xi): \xi \in H_i\})'$ is 
the identity. Therefore 
\begin{align*}
    (\AAA \cup \{\Gamma(\xi): \xi \in H_i\})'' (\subset \BBB(\FFF(H)_\tau))\cong\, &  
    (\AAA \cup \{\Gamma(\xi): \xi \in H_i\})''e_i \\
    \cong\, & \Phi(H_i)''(\subset \BBB(\FFF(H_i)_\tau)).
\end{align*}
Therefore, $\Phi(H_i)'' $ can be identified with a von Neumann subalgebra
of $\Phi(H)''$.    

Let $E$ be the faithful normal conditional expectation from $\Phi(H)''$ onto $\pi_\tau(\AAA)$
defined in \cref{faith_ce} such that $E(T)e_{\AAA} = e_{\AAA}T e_{\AAA}$. 
Consider $T_1, \ldots, T_l \in \Phi(H)$ such that $T_k \in \Phi(H_{i(k)})$, 
$i(1) \neq i(2) \neq \cdots \neq i(l)$ and $E_0(T_k) =0$ for all $k=1, \ldots, l$.
Then each $T_k$ is a finite sum of elements of the form
$L(\xi_1)\cdots L(\xi_n)L^*(\beta_1) \cdots L^*(\beta_m)$, where $\xi_j$, $\beta_j \in H_{i(k)}$ and
$n+m > 0$. It is not hard to check that $e_{\AAA}\pi_\tau(T_1 \cdots T_l)e_{\AAA} = 0$ 
(see \cite[Theorem 4.6.15.]{RS}) and $e_{s}\pi_\tau(T_1 \cdots T_l) e_{s} \neq 0$ 
only if $l=1$ and $s=i(1)$. Therefore $e_s \Phi(H)'' e_s = \Phi(H_s)'' e_s$ and 
let $E_{s}(T)$ be the unique element in $\Phi(H_s)''$ such that
$e_{s} T e_{s} = E_s(T)e_{s}$, for each 
$T \in \Phi(H)''$. By definition of $E_s$, we have
$E_s(\pi_\tau(A)T\pi_\tau(B))= \pi_\tau(A)E_s(T)\pi_\tau(B)$ and 
\begin{align*}
    E \circ E_s(T)e_{\AAA} = e_{\AAA} e_{s}Te_{s}e_{\AAA} = E(T)e_{\AAA},  
\end{align*}
where $A$, $B \in \Phi(H_s)$ and $T \in \Phi(H)''$.
Thus $E_s$ is a faithful normal conditional expectation
from $\Phi(H)''$ onto $\Phi(H_s)''$ satisfying $E \circ E_s = E$.

Summarizing the above discussion, we have the following proposition 
(cf.\ \cite[Theorem 4.6.15]{RS} and \cref{appendix_am_prod} for the definition of amalgamated free product 
among arbitrary von Neumann algebras).

\begin{proposition}\label{Fockisafunctor_prop}
    Let $\{(H_i, S_i, \{U_i(\alpha)\}_{\alpha \in \C})\}_{i\in I}$ be a family of  
    Tomita $\AAA$-$\AAA$ bimodules. Then $(\Phi(\oplus_{i}H_i)'', E) = 
    *_{\pi_\tau(\AAA)} (\Phi(H_i)'', E|_{\Phi(H_i)''})$.
\end{proposition}

\section{From rigid C$^*$-tensor categories to operator algebras via Tomita bimodules}\label{section_CstartoOA}

Let $\CC$ be a (small) rigid C$^*$-tensor category with simple (i.e., irreducible) unit $\ut$, finite 
direct sums and subobjects, see \cite{BKLR}, \cite{PSDV}, \cite{L-R-dim}. We may assume $\CC$ to be 
strict by \cite[XI.3, Theorem 1]{ML}. For objects $X$ in $\CC$ we write with abuse of notation $X\in\CC$, we denote by $t\in Hom(X,Y)$ the arrows in $\CC$ between $X,Y\in\CC$ and by $I_X$ the identity arrow in $Hom(X,X)$. Moreover, we write $XY$ for the tensor product of objects (instead of $X\otimes Y$), and $t\circ s$, $t\otimes s$, respectively, for the composition product and the tensor product of arrows.
We use $\overline{X}$ to denote a \emph{conjugate object} of $X$ (also called dual object) equipped with unit and counit  
\begin{align*}
    \eta_{X} \in Hom(\ut, \overline{X}X), \quad \varepsilon_{X} \in Hom(X\overline{X}, \ut) 
\end{align*}
satisfying the so-called conjugate equations, namely\footnote{$\otimes$ is always evaluated before $\circ$.}
\begin{align*}
    \varepsilon_{X} \otimes I_X \circ I_X \otimes \eta_X = I_X , \quad  
    I_{\overline{X}} \otimes \varepsilon_{X} \circ \eta_{X} \otimes I_{\overline{X}}= I_{\overline{X}}.
\end{align*}
Hence conjugation is specified by ordered pairs $(X,\overline X)$, or better by four-tuples $(X,\overline X, \eta_X, \varepsilon_X)$.
Furthermore, we always assume that $(\eta_X, \varepsilon_X^*)$ is a \emph{standard} solution 
of the conjugate equations, see \cite[Section 3]{L-R-dim}.
Important consequences of these assumptions on $\CC$, for which we refer to \cite{L-R-dim}, are
semisimplicity (every object is completely reducible into a finite direct sum of simple ones),
the existence of an additive and multiplicative dimension function on objects
$d_X := \eta_X^* \circ \eta_X = \varepsilon_{X} \circ \varepsilon_{X}^* \in \R$, $d_X \geq 1$, and 
a left (= right) trace defined by
\begin{align}\label{trace_def}
    \eta_{X}^* \circ I_{\overline{X}} \otimes t \circ \eta_X = 
    \varepsilon_{X} \circ t \otimes I_{\overline{X}} \circ \varepsilon_{X}^*, \quad
    \forall t \in Hom(X, X)
\end{align}
on the finite dimensional C$^*$-algebra $Hom(X,X)$. Moreover, for every $X$, $Y \in \CC$ 
$Hom(X, Y)$ is a finite dimensional Hilbert space with inner product given by
\begin{align}\label{hom_inner_prod}
    \braket{\xi_1}{\xi_2} =  \eta_{X}^* \circ I_{\overline{X}} 
    \otimes (\xi_1^* \circ \xi_2) \circ \eta_{X}, \quad \xi_1, \xi_2 \in Hom(X, Y).
\end{align}

\begin{notation}\label{notation_sec4}
    \begin{enumerate}
        \item Let $\SSS$ be a representative set of simple objects in $\CC$ such that $\ut\in\SSS$.
            Let $\Lambda = \SSS \cup \overline{\SSS}$ (disjoint union) where $\overline{\SSS} 
            = \{\overline{\alpha}: \alpha \in \SSS\}$. 
            We shall always consider $\alpha\in\SSS$ and 
            $\overline{\alpha}\in\overline{\SSS}$ as distinct elements (\lq\lq letters") in $\Lambda$, 
            even if $\alpha = \overline{\alpha}$ as objects in $\CC$. 
            In particular $\ut \neq \overline{\ut}$ in $\Lambda$.
        \item We label the fundamental \lq\lq fusion" $Hom$-spaces of $\CC$ by triples in $\SSS\times\SSS\times\Lambda$, namely
            \begin{align*}
                (\beta_1, \beta_2, \alpha) \in \SSS \times \SSS \times \Lambda \mapsto 
                \bfrac{\beta_2 * \alpha}{\beta_1} = Hom(\beta_1, \beta_2 \alpha),
            \end{align*}
            where $\beta_2 \alpha = \beta_2 \otimes \alpha$ and we identify $\alpha\in\Lambda$ with its embedding in $\CC$. 
            Note that the set of letters $\Lambda$ can be mapped (non-injectively) into the objects of $\CC$. 
        \item For every $\alpha \in \Lambda$, let $r_{\alpha} = \eta_{\alpha} \in Hom(\ut, \overline{\alpha} \alpha)$ if $\alpha\in\SSS$,
            and let $r_{\overline{\alpha}} = \varepsilon_{\alpha}^* \in Hom(\ut, \alpha \overline{\alpha})$ for the corresponding $\overline\alpha\in\overline
            \SSS$, where $(\eta_\alpha, \varepsilon_{\alpha}^*)$ is a standard solution of the conjugate equations as above.  
        \item For every $\alpha \in \SSS$, regarded as an element of $\Lambda$, let $\overline{\overline{\alpha}} = \alpha$. Thus 
            $\alpha \in \Lambda \mapsto \overline{\alpha} \in \Lambda$ is an involution on $\Lambda$.
    \end{enumerate}
    In the following, we shall use $\alpha,\beta,\gamma \ldots$ to denote elements in 
    $\Lambda$, or in $\SSS$, and $X,Y,Z,\ldots$ to denote objects in $\CC$. 
    For a set $S$, we denote its cardinality by $|S|$.
\end{notation}

We define a von Neumann algebra $\AAA(\ut)$ associated to the tensor unit $\ut$ by
$$\AAA(\ut) = \bigoplus_{\beta \in \SSS} \BBB(Hom(\beta, \beta)) \cong \bigoplus_{\beta \in \SSS}\C I_{\beta}$$ 
with a (non-normalized, semifinite) trace $\tau$ given by \cref{trace_def}, i.e., the trace values of minimal projections equals the categorical
dimension $\tau(I_{\beta}) = d_{\beta}$.

Let 
\begin{align*}
   H(\ut) = \bigoplus_{\alpha \in \Lambda, \beta_1, \beta_2 \in \SSS} 
    \bfrac{\beta_2 * \alpha}{\beta_1}, 
\end{align*}
where $\oplus$ stands for algebraic direct sum. 
Furthermore, $H(\ut)$ is an $\AAA(\ut)$-$\AAA(\ut)$ bimodule with left and right action given by 
\begin{align*}
    I_{\gamma_2} \cdot \xi \cdot I_{\gamma_1} = \delta_{\gamma_2, \beta_2} \delta_{\gamma_1, \beta_1} \xi, \quad
    \xi \in \bfrac{\beta_2 * \alpha}{\beta_1}, \alpha \in \Lambda, \beta_i, \gamma_i \in \SSS, i = 1,2.
\end{align*}
We define an $\AAA(\ut)$-valued inner product on $H(\ut)$ by
\begin{align*}
    \braket{\xi_2}{\xi_1}_{\AAA(\ut)} = \delta_{\alpha_1, \alpha_2}\delta_{\omega_1, \omega_2} \xi_2^*\xi_1 \in
    Hom(\beta_1, \beta_2), \quad \xi_i \in \bfrac{\omega_i * \alpha_i}{\beta_i}, i = 1, 2. 
\end{align*}
It is easy to check that $H(\ut)$ is a pre-Hilbert $\AAA(\ut)$-$\AAA(\ut)$ bimodule with normal left action, as we defined in \cref{section_prelim}.

Let $\alpha \in \Lambda$ and $\beta_1$, $\beta_2 \in \SSS$.
By Frobenius reciprocity \cite[Lemma 2.1]{L-R-dim}, we can define a bijection 
\begin{align}\label{conj_def}
    \xi \in \bfrac{\beta_2 * \alpha}{\beta_1} \mapsto 
    \overline{\xi}=  \xi^* \otimes I_{\overline{\alpha}} \circ  I_{\beta_2} \otimes r_{\overline{\alpha}}
    \in \bfrac{\beta_1 * \overline{\alpha}}{\beta_2}. 
\end{align}
It is clear that $\overline{\overline{\xi}} = \xi$.
Although $\beta_2 \alpha$ and $\beta_1 \overline{\alpha}$ can be the same object in $\CC$, 
we have that $\xi \neq \overline{\xi}$ as vectors in $H(\ut)$ because $\overline{\alpha}\neq\alpha$
in $\Lambda$ by our assumption.

We now define an involution $S$ on $H(\ut)$ and a complex one-parameter group $\{U(z):z \in \C\}$ of 
isomorphisms that endow $H(\ut)$ with the structure of a Tomita $\AAA(\ut)$-$\AAA(\ut)$
bimodule (see \cref{tomita_bimodule_def}).

For each $\alpha \in \Lambda$, we choose a real number $\lambda_{\alpha} > 0$ if $\alpha\in\SSS$. 
For the corresponding $\overline{\alpha}\in\overline{\SSS}$, let $\lambda_{\overline{\alpha}} = 1/\lambda_{\alpha}$.
We define the action of $S$ and $U(z)$ on every $\xi \in \bfrac{\beta_2 * \alpha}{\beta_1}$, 
$\alpha\in\Lambda$, $\beta_1,\beta_2\in\SSS$ by
\begin{align}\label{tomita_struc}
    S : \xi \mapsto \sqrt{\lambda_{\alpha}}\, \overline{\xi}, \quad
    U(z) : \xi \mapsto \lambda_{\alpha}^{iz}\, \xi,\quad z\in\C,
\end{align}
where $\overline{\xi}$ is defined by \cref{conj_def}. Summing up, we have that

\begin{proposition}\label{prop_tomitafromC}
$H(\ut)$ is a Tomita $\AAA(\ut)$-$\AAA(\ut)$ bimodule, for every choice of $\lambda_{\alpha}$ and 
$\lambda_{\overline{\alpha}} = 1/\lambda_{\alpha}$ associated with the elements $\alpha\in\Lambda$.
\end{proposition}

\begin{remark}\label{rmk_from1toX}
Similarly, one can define a semifinite (not necessarily abelian) von Neumann algebra $\AAA(X)$ and a Tomita bimodule $H(X)$ associated with an arbitrary object of the category $X\in\CC$, where semifiniteness of $\AAA(X)$ is again due to the rigidity of $\CC$, and the Tomita structure depends as before on the choice of $\lambda_{\alpha}$, $\alpha\in\Lambda$. Namely
\begin{align*}
    \AAA(X) = \bigoplus_{\beta_1\in\SSS} 
    \BBB(\oplus_{\beta_2\in\SSS} Hom(\beta_1, X\beta_2)) 
\end{align*}
and
\begin{align*}
    H(X) = \bigoplus_{\alpha\in\Lambda, \beta_1,\beta_2\in\SSS} Hom(X\beta_1,X\beta_2\alpha),
\end{align*}
where the semifinite tracial weight on $\AAA(X)$, the bimodule actions of $\AAA(X)$ and the $\AAA(X)$-valued inner product on $H(X)$ are defined similarly to the case $X=\ut$. The Tomita structure can be defined on $H(X)$, for every $\xi \in Hom(X\beta_1, X\beta_2\alpha)$ for 
$\alpha\in\Lambda$, $\beta_1,\beta_2\in\SSS$, again by \cref{tomita_struc} but replacing the conjugation on vectors $\xi \mapsto \overline \xi$ defined by \cref{conj_def} on $H(\ut)$ with the following defined on $H(X)$ by 
\begin{align*}
    \xi \in Hom(X\beta_1,X\beta_2\alpha) \mapsto 
    \overline{\xi}=  \xi^* \otimes I_{\overline{\alpha}} \circ  I_{X\beta_2} \otimes r_{\overline{\alpha}}
    \in Hom(X\beta_2,X\beta_1\overline\alpha). 
\end{align*}
For ease of exposition, we state the results of this section in the case $X=\ut$, namely we will show that $\Phi(H(\ut))''$ is a factor (\cref{factor_thm}) and we will characterize its type depending on the chosen Tomita structure (\cref{type_prop}). The same statements hold for an arbitrary object $X\in\CC$, with suitable modification of the proofs. 

In section 6 we define other algebras denoted by $\Phi(X)$ (\cref{amg_free_lemma})
which play the same role of $\Phi(H(X))''$ but live inside a bigger auxiliary
algebra $\Phi(\CC)$. We will only show that $\Phi(H(\ut))'' \cong \Phi(\ut)$ (\cref{lem_phione}) 
and that $\Phi(\ut) \cong \Phi(X)$ for every $X \in\CC$ (\cref{X_equl},
assuming that the spectrum of the category is infinite), but we omit the
proof of $\Phi(H(X))'' \cong \Phi(X)$ showing the equivalence of the two
constructions.
\end{remark}

In the following, we choose and write 
$\OOO^{\beta_2 * \alpha}_{\beta_1} = \{\xi_i\}_{i=1}^{n}$ for a fixed orthonormal basis of isometries in 
$\bfrac{\beta_2 * \alpha}{\beta_1}$, i.e., $\xi_i^* \circ \xi_j = \delta_{i,j} I_{\beta_1}$, 
where $n = dim Hom(\beta_1, \beta_2 \alpha)$, for every $\beta_1, \beta_2 \in \SSS$ and 
$\alpha \in \Lambda$. Note that $\OOO^{\beta_2 * \alpha}_{\beta_1} = \emptyset$ if 
$n = 0$.
In the case $\alpha = \ut$, we let $\OOO^{\beta * \ut}_{\beta} = \{I_\beta\}$. Thus we have that
\begin{align}\label{H(X)_eq}
    H(\ut) = \bigoplus_{\alpha,\beta_1, \beta_2 \in \SSS} \bigoplus_{\xi \in 
    \OOO^{\beta_2 * \alpha}_{\beta_1}} \{a \xi + b \overline{\xi}: a, b \in \C\}.
\end{align}

From now on, we identify $\AAA(\ut)$ with the von Neumann subalgebra $\pi_{\tau}(\AAA(\ut))$ of 
$\Phi(H(\ut))''$ acting on the Fock (Hilbert) space $\FFF(H(\ut))_{\tau}$ (see \cref{fockspace_def} and \cref{phialgs_def}), 
and we denote by $E$ the normal faithful conditional expectation from $\Phi(H(\ut))''$ onto $\AAA(\ut)$ such that 
\begin{align}\label{def_E}
    E(\Gamma(\xi_1) \cdots \Gamma(\xi_n)) = \braket{\Omega}{\Gamma(\xi_1) \cdots \Gamma(\xi_n) \Omega}_{\AAA(\ut)}, 
\end{align}
where $\Omega = I$ is the vacuum vector in $\FFF(H(\ut))$.

\begin{lemma}\label{unit_case}
    Suppose that $\OOO^{\beta * \alpha}_{\beta} \neq \emptyset$ and let 
    $\xi\in\OOO^{\beta * \alpha}_{\beta}$, where $\beta\in\SSS, \alpha\in\Lambda$.
    If $\lambda_{\alpha} \leq 1$ set $\eta =  \xi$. Otherwise $\lambda_{\overline{\alpha}} = 1/\lambda_{\alpha} \leq 1$ 
    and set $\eta = \overline{\xi}$.
    Denote by $\NNN_\beta$ the von Neumann algebra generated by $\Gamma(\eta)$, then
    $$\Phi(\{a \xi + b \overline{\xi}: a, b \in \C\})'' \;\cong\; 
    \NNN_\beta \;\oplus \bigoplus_{\gamma \in \SSS \setminus \{\beta\}} \C I_{\gamma}$$ 
    and 
    \begin{align*}
        \NNN_\beta\; \cong\; 
        \begin{cases}
            L(F_2) & \mbox{ if $\lambda_{\alpha} = 1$,}\\
            \mbox{free Araki-Woods type III$_{\lambda}$ factor} & \mbox{ otherwise},
        \end{cases}
    \end{align*}
    where $\lambda = \min\{\lambda_{\alpha}, \lambda_{\overline{\alpha}}\}$.
    Furthermore, in the polar decomposition $\Gamma(\eta) = VK$, we have $Ker K = \{0\}$ and   
    $V$, $K$ are *-free with respect to $E$. The distribution of $K^2 = \Gamma(\eta)^* \Gamma(\eta)$ 
    with respect to the state $d_{\beta}^{-1} \tau \circ E|_{\NNN_\beta}$ is 
    \begin{align*}
        \frac{\sqrt{4\lambda - (t-(1+\lambda))^2}}{2\pi \lambda t} dt, \quad t \in 
        ((1-\sqrt{\lambda})^2, (1+\sqrt{\lambda})^2).
    \end{align*}
\end{lemma}

\begin{proof}
    By exchanging the roles of $\xi$ and $\overline{\xi}$, hence of $\alpha$ and $\overline{\alpha}$, 
    we can assume $\lambda_{\alpha} \leq 1$. It is clear that 
    $\Phi(\{a \xi + b \overline{\xi}\})''$ is generated by $\AAA(\ut)$ and 
    $\Gamma(\eta) = L(\xi) + \sqrt{\lambda_{\alpha}} L^*(\overline{\xi})$.
    Since $I_{\beta}$ is a minimal projection in $\AAA(\ut)$ and $I_\beta \Gamma(\eta) = \Gamma(\eta) 
    = \Gamma(\eta) I_\beta$, 
    we have that $E(T) = d_{\beta}^{-1} \tau(E(T))I_{\beta}$ for any $T \in \NNN_\beta$. 
    Note that $\braket{\xi}{\overline{\xi}}_{\AAA(\ut)} = 0$. Then \cite[Remark 4.4, Theorem 4.8 and Theorem 6.1]{DSFQ}
    imply the result (see also \cite[Example 2.6.2]{HP} for the case $\lambda_\alpha = 1$).
\end{proof}

\begin{remark}\label{central_re}
    With the notation of \cref{unit_case}, we have
    $\sigma_t^{\tau \circ E}(\Gamma(\xi)) = \lambda_{\alpha}^{it}\Gamma(\xi)$
    for every $t\in\R$ by \cref{module_auto}. Therefore $K = (\Gamma(\xi)^*\Gamma(\xi))^{1/2}$ 
    is in the centralizer of $\tau \circ E$, denoted by $\Phi(H(\ut))''_{\tau \circ E}$.
\end{remark}

\begin{lemma}\label{diff_case}
   Let $\xi \in \OOO^{\beta_2 * \alpha}_{\beta_1}$ for $\alpha \in \Lambda$ and $\beta_1$, $\beta_2 \in \SSS$. If 
   $\lambda = \lambda_{\alpha} d_{\beta_1}/d_{\beta_2} \leq 1$,
   then\footnote{We adopt the notation used in \cite{DF} to specify the (non-normalized) trace on 
       a direct sum of algebras. Let $\AAA_1$ and $\AAA_2$ be two finite von Neumann algebra 
       with two given states $\omega_{1}$ and $\omega_2$ respectively. We use 
       $\stackrel[t_1]{p}{\AAA_1} \oplus \stackrel[t_2]{q}{\AAA_1}$
       to denote the direct sum algebra with distinguished positive linear functional $t_1 \omega_1(a) + t_2 \omega_2(b)$, 
       $a \in \AAA_1$, $b \in \AAA_2$,
       where $p$ and $q$ are projections corresponding to the identity elements of $\AAA_1$ and $\AAA_2$ respectively.
       A special case of this is when $\omega_{1}$, $\omega_2$ and $t_1 \omega_1(\cdot) + t_2 \omega_2(\cdot)$ are tracial.}
   \begin{align*}
       \Phi(\{a \xi + b \overline{\xi}:a , b \in \C\})'' \;\cong\; 
        (\stackrel[d_{\beta_1}+ \lambda_{\alpha}d_{\beta_1}]{
        I_{\beta_1} + p_{\beta_2}}{L^{\infty}([0,1]) \otimes M_2(\C)}) 
     \;\;\oplus \stackrel[d_{\beta_2}-\lambda_{\alpha}d_{\beta_1}]{q_{\beta_2}}{\C}
     \oplus \bigoplus_{\gamma \in \SSS \setminus \{\beta_1, \beta_2\}} \C I_\gamma,
    \end{align*}
    where $p_{\beta_2}$ and $q_{\beta_2}$ are two projections such that 
    $p_{\beta_2} + q_{\beta_2} = I_{\beta_2}$ and $\tau \circ E(p_{\beta_2}) = \lambda_\alpha d_{\beta_1}$,
    $\tau \circ E(q_{\beta_2}) = d_{\beta_2} - \lambda_\alpha d_{\beta_1}$.
\end{lemma}

\begin{proof}
    As in the previous lemma, $\Phi(\{a \xi + b \overline{\xi}\})''$ is the von Neumann subalgebra of $\Phi(H(\ut))''$ 
    generated by $\AAA(\ut)$ and $\Gamma(\xi)$.
    By \cite[Lemma 4.3 and Remark 4.4]{DSFQ}, we have that the distributions of $\Gamma(\xi)^*\Gamma(\xi)$ and 
    $\Gamma(\xi)\Gamma(\xi)^*$ with respect to $d_{\beta_1}^{-1}\tau \circ E$ and 
    $d_{\beta_2}^{-1}\tau \circ E$ are
    \begin{align*}
        \frac{\sqrt{4\lambda - (t-(1+ \lambda))^2}}{2\pi \lambda t}dt \mbox{ and } 
        \frac{\sqrt{4\lambda - (t-(1+ \lambda))^2}}{2\pi t}dt + (1-\lambda)\delta_0,
    \end{align*}
    where $\delta_0$ is the Dirac measure concentrated in $0$.
    
    Let $\xi_1 = \xi$ and $\xi_2 = \sqrt{1/\lambda}\, S(\xi) \in \bfrac{\beta_1 * \overline{\alpha}}{\beta_2}$. 
    Then $\xi_2^* \xi_2 = I_{\beta_2}$, because $\tau(\xi_2^* \xi_2) = d_{\beta_2}$, and the operator 
    $\Gamma(\xi) = L(\xi_1) + \sqrt{\lambda} L(\xi_2)^*$ fulfills $I_{\beta_2}\Gamma(\xi) = \Gamma(\xi) = \Gamma(\xi) I_{\beta_1}$.
    Let $\Gamma(\xi) = UH$ be the polar decomposition of $\Gamma(\xi)$.
    By the proof of \cite[Lemma 4.3]{DSFQ}, we know that
    $H\in I_{\beta_1}\Phi(H(\ut))''I_{\beta_1}$ and $KerH = \{0\}$ (even when $\lambda = 1$), and 
    $U^*U = I_{\beta_1}$. Since $\sigma_{t}^{\tau \circ E}(U) = \lambda_{\alpha}^{it}U$, $t\in\R$,
    we have
    \begin{align*}
        d_{\beta_1} = \tau \circ E(U^*U) = 1/\lambda_{\alpha}\tau \circ E(UU^*),
    \end{align*}
    thus $\tau \circ E(UU^*) = \lambda_{\alpha}d_{\beta_1} \leq d_{\beta_2}$.
    Since $\tau \circ E$ is faithful, we know that $UU^*$ is a subprojection of $I_{\beta_2}$. Let 
    $p_{\beta_2} = UU^*$ and $q_{\beta_2} = I_{\beta_2} - UU^*$. 
    Note that $p_{\beta_2} = I_{\beta_2}$ if and only if $\lambda = 1$.
    The fact that distribution of $H^2$ is non-atomic implies that the von Neumann algebra
    generated by $H$ is isomorphic to $L^{\infty}([0,1])$. Thus the von Neumann algebra 
    generated by $I_{\beta_1}, I_{\beta_2}$ and 
    $\Gamma(\xi)$ is isomorphic to 
    \begin{align*}
        (\stackrel[d_{\beta_1}+ \lambda_{\alpha}d_{\beta_1}]{
        I_{\beta_1} + p_{\beta_2}}{L^{\infty}([0,1]) \otimes M_2(\C)}) 
        \;\;\oplus \stackrel[d_{\beta_2}-\lambda_{\alpha}d_{\beta_1}]{q_{\beta_2}}{\C}.
    \end{align*}
\end{proof}

Now consider the sets of isometries $\OOO^{\beta * \alpha}_{\beta}$ in the special case $\alpha=\ut$.
For every $\beta \in \SSS$, let $\xi_{\beta} = I_\beta$ be the only element in $\OOO^{\beta * \ut}_{\beta}$.
By \cref{unit_case}, we have that
\begin{align}\label{M1_eq}
    \Phi(\bigoplus_{\beta \in \SSS}\, \{a \xi_\beta + b \overline{\xi}_\beta: a, b \in \C\})''
    \cong \bigoplus_{\beta \in \SSS} \NNN_\beta
\end{align}
where $\NNN_\beta$ is $L(F_2)$ or the free Araki-Woods type III$_{\lambda}$ factor, 
$\lambda = \min\{\lambda_\ut,\lambda_{\overline{\ut}}\}$, and 
$\NNN_\beta$ is generated by a partial isometry $V_\beta$ and a positive operator $K_\beta$ in the centralizer of 
$\tau \circ E$ (see \cref{central_re}). If we let $U_\beta = f(K_{\beta}^2)$ where
$f(x) = e^{2\pi i h(x)}$ and $h(x) = \int_{(1-\sqrt{\lambda})^2}^x 
\sqrt{4\lambda - (t-(1+\lambda))^2}/(2\pi \lambda t) dt$,  
we get a Haar unitary, i.e., $E(U_\beta^n)=0$ for every $n \neq 0$, which generates $\{K_\beta\}''$.
Moreover, $U_\beta$ and $V_\beta$ are *-free with respect to $E$.

Denote by $\MMM_1$ the l.h.s.\ of \cref{M1_eq} and by $\tilde{\MM}_1$ the strongly dense *-subalgebra of $\MMM_1$
generated by $\AAA(\ut)$ and $\{U_\beta, V_\beta: \beta \in \SSS \}$. By \cref{Fockisafunctor_prop} and \cref{H(X)_eq} 
we have an amalgamated free product decomposition of $\Phi(H(\ut))''$, namely
\begin{align}\label{M2_eq}
(\Phi(H(\ut))'', E) = (\MMM_1, E |_{\MMM_1}) *_{\AAA(\ut)} (\MMM_2, E |_{\MMM_2})
\end{align}
where $\MMM_2 = \Phi(\oplus_{\alpha,\beta_1, \beta_2 \in \SSS, (\beta_1 = \beta_2, 
\alpha\neq\ut \text{ or } \beta_1 \neq \beta_2)} \oplus_{\xi \in \OOO^{\beta_2 * \alpha}_{\beta_1}} 
\{a \xi + b \overline{\xi}: a, b \in \C\})''$.

\begin{lemma}\label{popa_unitary}
    With the notations above, there is a unitary $U \in \Phi(H(\ut))_{\tau \circ E}'' \cap \MMM_1$
    such that $E(t_1 U^n t_2) \rightarrow 0$ for $n\rightarrow+\infty$ in the strong operator topology
    for every $t_1, t_2  \in \tilde{\MM}_1$.
\end{lemma}

\begin{proof}
    By \cref{unit_case} and the discussion above, $U = \sum_{\beta\in\SSS} U_{\beta}$ satisfies the conditions.
\end{proof}

\begin{theorem}\label{factor_thm}
    $\Phi(H(\ut))''$ is a factor.  
\end{theorem}

\begin{proof}
    Let $U$ be the unitary in \cref{popa_unitary}. If $T \in \Phi(H(\ut))'' \cap \Phi(H(\ut))'$, then
    $T \in \{U\}' \cap \MMM_1$ by \cref{relative_commutant}.
    Thus \cref{diff_case} implies $\Phi(H(\ut))'' \cap \Phi(H(\ut))' 
    = \Phi(H(\ut))' \cap \AAA(\ut) = \C I$.
\end{proof}

\begin{lemma}\label{center_of_centralizer}
    $\ZZZ(\Phi(H(\ut))''_{\tau \circ E}) \subseteq \AAA(\ut)$, where $\ZZZ(\cdot)$ denotes the center.
\end{lemma}

\begin{proof}
    Proceed as in the proof of \cref{factor_thm}, we have $\ZZZ(\Phi(H(\ut))_{\tau \circ E}'') 
    \subset \MMM_1$. If $\lambda_\ut = 1$, then $\MMM_1 \subset \Phi(H(\ut))_{\tau \circ E}''$
    and we have $\ZZZ(\Phi(H(\ut))_{\tau \circ E}'' \subseteq \AAA(\ut)$.
    By \cite[Corollary 4]{LB}, \cite[Corollary 5.5]{DSFQ} and \cref{unit_case}, it is also 
    clear that $\ZZZ(\Phi(H(\ut))_{\tau \circ E}'' )\subseteq \AAA(\ut)$ if $\lambda_\ut \neq 1$. 
\end{proof}

Recall that the Connes' invariant $S(\MMM)$ of a factor $\MMM$ is given by 
$\cap\{Sp(\Delta_{\varphi}): \varphi \text{ is a n.s.f. weight on $\MMM$}\}$ 
and that $\R^{+}_{*} \cap S(\MMM)= \R^{+}_{*}\cap\{Sp(\Delta_{\varphi_e}): 0 \neq e \in Proj(\ZZZ(\MMM_{\varphi})\}$
is a closed subgroup of $\R^{+}_{*} = \{t\in\R, t>0\}$, where $\varphi_{e} = \varphi|_{e\MMM e}$. See \cite[Chapter III]{AT}, \cite[Chapter VI]{SS}. Here $Proj(\cdot)$ denotes the set of orthogonal projections of a von Neumann algebra and $Sp(\cdot)$ the spectrum of an operator.

\begin{lemma}\label{connes_s}
    Let $G$ be the closed subgroup of $\R^{+}_{*}$ generated by $\{\lambda_\alpha \lambda_{\beta_2}/\lambda_{\beta_1}:
    \alpha \in \Lambda, \beta_1, \beta_2 \in \SSS \mbox{ such that } \bfrac{\beta_2 * \alpha}{\beta_1} \neq \{0\}\}$.
    Then $\R^{+}_{*} \cap S(\Phi(H(\ut))'') = G$.  
\end{lemma}

\begin{proof}
    We use $\varphi$ to denote $\tau \circ E$. By \cref{center_of_centralizer} and 
    \cite[Section 15.4]{SS}, we know that
    $\R^{+}_{*} \cap S(\Phi(H(\ut))'') = \R^{+}_{*} \cap \{Sp(\Delta_{\varphi_{I_\beta}}): \beta \in \SSS\}$.
    Note that $(I_\beta \Phi(H(\ut))'' I_\beta)_{\varphi_{I_\beta}} = 
    I_\beta(\Phi(H(\ut))'')_{\varphi} I_\beta$. Thus the center of 
    $(I_\beta \Phi(H(\ut))''I_\beta)_{\varphi_{I_\beta}}$ is trivial by
    \cite[Proposition 5.5.6]{KRII}. Then \cite[Section 16.1, Corollary 16.2.1]{SS} implies 
    $\R^{+}_{*} \cap S(\Phi(H(\ut))'') = \R^{+}_{*} \cap Sp(\Delta_{\varphi_{I_\ut}})$. 

    Let $\xi_1 \in \bfrac{\beta_1 * \alpha_1}{\ut},\xi_2 \in \bfrac{\beta_2 * \alpha_2}{\beta_1}, \dots, 
    \xi_n \in \bfrac{\ut *\alpha_n}{\beta_{n-1}}$ and $\alpha_i \in \Lambda, \beta_i \in \SSS$, $n=1,2,\ldots$.  
    By \cref{module_auto}, for every $t\in\R$, we have 
    \begin{align*}
        \sigma_{t}^{\varphi}(\Gamma(\xi_n) \cdots \Gamma(\xi_1))
        = (\frac{\lambda_{\ut}\lambda_{\alpha_n}}{\lambda_{\beta_{n-1}}})^{it} \hspace{-1mm}\dots
        (\frac{\lambda_{\beta_2}\lambda_{\alpha_2}}{\lambda_{\beta_1}})^{it}
        (\frac{\lambda_{\beta_1}\lambda_{\alpha_1}}{\lambda_{\ut}})^{it}
        \,\Gamma(\xi_n)\cdots \Gamma(\xi_1).
    \end{align*}

    Also note that $\sigma_{t}^{\varphi}(\Gamma(\zeta_2) \Gamma(\xi) \Gamma(\zeta_1))
    = (\frac{\lambda_{\beta_2}\lambda_{\alpha}}{\lambda_{\beta_{1}}})^{it}
    \Gamma(\zeta_2) \Gamma(\xi) \Gamma(\zeta_1)$, 
    where $\zeta_1 \in \bfrac{\beta_1 * \overline{\beta}_1}{\ut}$, $\xi \in \bfrac{\beta_2 * \alpha}{\beta_1}$, 
    $\zeta_2 \in \bfrac{\ut * \beta_2}{\beta_2}$.  
    Since the linear span of 
    $\{p_\ut\} \cup \{\Gamma(\xi_n) \cdots \Gamma(\xi_1)\}$ as above is dense in the space 
    $L^2(I_\ut \Phi(H(\ut))'' I_\ut, \varphi_{I_\ut})$, we have $Sp(\Delta_{\varphi_{I_\ut}}) \setminus \{0\} = G$.
\end{proof}

\begin{proposition}\label{type_prop}
    $\Phi(H(\ut))''$ is semifinite if and only if 
    $\lambda_{\beta_1} = \lambda_{\beta_2}\lambda_{\alpha}$ for every $\alpha\in\Lambda$ and $\beta_1,\beta_2\in\SSS$
    such that $\bfrac{\beta_2 * \alpha}{\beta_1} \neq \{0\}$. 
    If $\Phi(H(\ut))''$ is not semifinite, then
    it is a type III$_\lambda$ factor for $\lambda \in (0, 1]$, where $\lambda$ depends on the choice of (non-trivial) Tomita structure. 
\end{proposition}

\begin{proof}
    By \cite[Proposition 28.2]{SS}, $\Phi(H)''$ is semifinite if and only if $S(\Phi(H(\ut))'') = \{1\}$. 
    If $\Phi(H(\ut))''$ is semifinite, \cref{connes_s} implies that 
    $\lambda_{\beta_1} = \lambda_{\beta_2}\lambda_{\alpha}$ for every $\alpha \in \Lambda$ and 
    $\beta_1,\beta_2\in\SSS$ such that $\bfrac{\beta_2 * \alpha}{\beta_1} \neq \{0\}$. 

    Conversely, if $\lambda_{\beta_1} = \lambda_{\beta_2}\lambda_{\alpha}$, by
    \cite[VIII, Theorem 2.11]{TKII}, we have
    \begin{align*}
        \sigma_{t}^{\tau_{K} \circ E}(\Gamma(\xi))= K^{it}\sigma^{\tau \circ E}_t(\Gamma(\xi))K^{-it}
        =\Gamma(\xi),
    \end{align*}
    for every $\xi \in \bfrac{\beta_2 * \alpha}{\beta_1}$, $\alpha\in\Lambda$, $\beta_1, \beta_2 \in \SSS$,
    where $K = \sum_{\beta \in \SSS} \lambda_{\beta} I_\beta$ and 
    $\tau_{K}(\cdot) = \tau(K^{1/2} \cdot K^{1/2})$. Thus $\Phi(H(\ut))''$ is semifinite.
\end{proof}

\begin{remark}\label{decompose_remark}
If $\Phi(H(\ut))''$ is semifinite, then $\lambda_{\ut}^2 = \lambda_{\ut}$ implies $\lambda_\ut =1$.
If $\alpha\in\SSS$ is self-conjugate, i.e., $\alpha \cong \overline{\alpha}$ in $\CC$, 
then $\lambda_{\alpha}^2 = 1$, which implies $\lambda_\alpha =1$.
If $\alpha$ is not self-conjugate, then there is $\beta \in \SSS$ which is a conjugate of $\alpha$
and we have $\lambda_{\alpha} \lambda_{\beta} = 1$.

If the spectrum $\SSS$ is finite, then $\Phi(H(\ut))''$ is a type II$_1$ factor.
If $\SSS$ is infinite, then $\tau_{K} \circ E(I) = 
\sum_{\beta\in\SSS} \lambda_\beta d_{\beta} = +\infty$, where $\tau_{K}$ is defined in the proof of
\cref{type_prop}. Thus $\Phi(H(\ut))''$ is a II$_\infty$ factor. 
Observe that $\tau_K \circ E$ is a trace (tracial weight), whereas $\tau\circ E$ is a trace if and only 
if $\lambda_\alpha = 1$ for every $\alpha\in\Lambda$. 
\end{remark}

\section{The case $\lambda_\alpha =1$}\label{section_lambda1}

In this section, we assume that the spectrum $\SSS$ of the category $\CC$ is infinite (not necessarily denumerable)
and study the algebra $\Phi(H(\ut))''$ constructed before,
in the special case $\lambda_\alpha = 1$ for every $\alpha \in \SSS$ (hence for every $\alpha\in\Lambda$), i.e., 
when the $\AAA(\ut)$-$\AAA(\ut)$ pre-Hilbert bimodule $H(\ut)$ has trivial Tomita bimodule structure. The main result is that, 
in the trivial Tomita structure case, the free group factor $L(F_\SSS)$ (with countably infinite or uncountably 
many generators) sits in a corner of $\Phi(H(\ut))''$, namely $I_\ut \Phi(H(\ut))'' I_\ut \cong L(F_\SSS)$
(\cref{thm_freecorner}).

All von Neumann algebras in this section are semifinite and all
conditional expectations are trace-preserving. Since the conditional expectation onto a von Neumann
subalgebra with respect to a fixed tracial weight is unique (see \cite[Theorem 4.2, IX]{TKII}), 
we sometimes omit the conditional expectations from the expressions of amalgamated free products to simplify the
notation, when there is no ambiguity.

Recall that $\AAA(\ut) = \oplus_{\beta \in \SSS} \C I_{\beta}$ is the abelian von Neumann algebra with 
tracial weight $\tau$ such that $\tau(I_{\beta}) = d_{\beta}$, where $I_\beta$ is the identity
of $Hom(\beta, \beta)$, and $d_\beta$ is the categorical dimension of $\beta \in \SSS$.
By \cref{factor_thm} and \cref{type_prop},
\begin{align*}
    (\Phi(H(\ut))'', E) = *_{\AAA(\ut),\, \alpha, \beta_1, \beta_2 \in \SSS,\, \xi 
    \in \OOO^{\beta_2 * \alpha}_{\beta_1}} (\Phi(\{a \xi + b \overline{\xi}: a, b \in \C\})'', E)
\end{align*}
is a II$_\infty$ factor (semifiniteness is due to the assumption $\lambda_\alpha = 1, \alpha\in\SSS$).

We use $L(F_s)$, $1 < s < \infty$, to denote the interpolated free group factor \cite{Radu}, \cite{DFI}. 
In general, for every non-empty set $S$, let $L(F_S)$ be free group factor of the free group over $S$ with
the canonical tracial state $\tau_S$ \cite[Section 6]{Pop}.  

We let $\MMM(F_S) = \oplus_{\beta \in \SSS} L(F_S)$ and identify $\AAA(\ut)$ 
with the subalgebra $\oplus_{\beta \in \SSS} \C$ of $\MMM(F_S)$. Let $E_S$ be the faithful normal 
conditional expectation from $\MMM(F_S)$ onto $\AAA(\ut)$ such that
$E_S(\oplus_{\beta \in \SSS} a_\beta) = \sum_{\beta \in \SSS} \tau_S(a_\beta)I_{\beta}$,
where $a_\beta \in L(F_S)$. By \cref{unit_case}, we have 
\begin{align*}
    (\MMM(F_2), E_2) =  *_{\AAA(\ut),\, \beta \in \SSS, \xi 
    \in \OOO^{\beta * \ut}_{\beta}} (\Phi(\{a \xi + b \overline{\xi}: a, b \in \C\})'', E).
\end{align*}

The following fact is well-known at least when $S$ is denumerable \cite{RF} and it shows that the fundamental group 
of an uncountably generated free group factor is $\R^{+}_{*} = \{t\in\R, t>0\}$ as well.

\begin{lemma}\label{fun_grp_free_group}
    Let $S$ be an infinite set (not necessarily denumerable). Then
    $L(F_S) \cong pL(F_S)p$ for every non-zero projection $p \in L(F_S)$.
\end{lemma}

\begin{proof}
    Let $S = \cup_{i \in \Theta}S_i$ be a partition of $S$ such that $|S_i| = |\N|$ and 
    $|\Theta| = |S|$. Let $(\MMM, \tau)$ be a W*-probability space with a normal faithful tracial state $\tau$, containing a copy of the 
    hyperfinite II$_1$-factor $R$ and a semicircular family $\{X_{ij}:i \in \Theta, j \in S_i\}$ (see 
    \cite[Definition 5.1.1]{NVK}) such that $\{X_{ij}:i \in \Theta, j \in S_i\}$ and $R$ are free.
    Consider the von Neumann algebra $\NNN = (R \cup \{qX_{ij}q:i \in \Theta, j \in S_i\})''$, where
    $q \in R$ is a projection with the same trace value as $p$.

    Let $\NNN_i = (R \cup \{qX_{ij}q:j \in S_i\})''$.
    For each $i \in \Theta$, by \cite[Proposition 2.2]{DFI}, there exists a 
    semicircular family $\{Y_{ij}:j \in S_i\} \subset \NNN_i$
    such that $\NNN_i = (R \cup \{Y_{ij}:j \in S_i\})''$ and $\{Y_{ij}:j \in S_i\}$ and $R$ are free.

    Note that $\{Y_{i_0 j}:j \in S_{i_0}\}$ and 
    $R \cup \{X_{ij}: i \in \Theta \setminus\{i_0\}, j \in S_i\}$ are free for every $i_0 \in \Theta$.
    Thus $\{Y_{ij}: i \in \Theta, j \in S_i\}$ is a semicircular family which is free with $R$.

    By \cite[Theorem 4.1]{DFM}, $L(F_S) \cong (R \cup \{Y_{ij}: i \in \Theta, j \in S_i\})''= \NNN$. 
    Since $q\NNN q = (qRq \cup \{qX_{ij}q:i \in \Theta, j \in S_i\})''$ and $qRq \cong R$, we have 
    $L(F_S) \cong p L(F_S) p$ by \cite[Theorem 1.3]{DFI}.
\end{proof} 

\begin{lemma}\label{step_one}
    For every set $S$ such that $1 \leq |S| \leq |\SSS|$, let
    \begin{align*}
        (\MMM, \mathcal{E}) = (\MMM(F_S), E_S) *_{\AAA(\ut),\, \alpha \in \SSS \setminus \{\ut\}, 
    \xi \in \OOO_{\alpha}^{\ut * \alpha}}
        (\Phi(\{a \xi + b \overline{\xi}: a, b \in \C\})'', E).
    \end{align*}
    Then $\MMM$ is a II$_\infty$ factor with a tracial weight 
    $\tau \circ \mathcal{E}$ and $I_\ut \MMM I_\ut \cong L(F_{\SSS})$.
\end{lemma}

\begin{proof} 
    \cref{modular_auto} in the appendix implies that $\tau \circ \mathcal{E}$ is a tracial weight.
    By \cref{diff_case}, we have
    \begin{align*}
        \MMM = (\stackrel[d_\ut]{I_\ut}{L(F_S)} 
        \oplus_{\alpha \in \SSS \setminus \{\ut\}} 
        \stackrel[d_\alpha]{I_\alpha}{\C})*_{\AAA(\ut),\, \beta \in \SSS \setminus \{\ut\}}\MMM_\beta,
    \end{align*}
    where $\MMM_\beta$ is 
    \begin{align*}
        \MMM_\beta = (\stackrel[d_\beta]{I_\beta}{L(F_S)}\oplus_{\alpha \in \SSS \setminus \{\beta\}} 
        \stackrel[d_\alpha]{I_\alpha}{\C})*_{\AAA(\ut)}
        ( (\stackrel[2]{
            I_{\ut} + p_{\beta}}{L(F_1) \otimes M_2(\C)}) 
        \; \oplus \hspace{-1mm} \stackrel[d_{\beta}-1]{q_{\beta}}{\C}
        \; \oplus_{z \in \SSS \setminus \{\ut, \beta\}} 
        \stackrel[d_z]{I_z}{\C}  ),
    \end{align*}
    and $p_\beta + q_{\beta} = I_\beta$. It is clear that the central carrier of $I_\beta$ in $\MMM_\beta$ is
    $I_\ut + I_\beta$. By \cite[Lemma 4.3]{DI}, 
    \begin{align*}
        I_\beta \MMM_\beta I_\beta = \stackrel[d_\beta]{I_\beta}{L(F_S)} * \stackrel[1]{p_\beta}{(L(F_1)}\oplus 
        \stackrel[d_\beta -1]{q_\beta}{\C}).
    \end{align*} 
    By \cite[Proposition 1.7 (i)]{DF}, $I_\beta \MMM_\beta I_\beta \cong L(F_{t_{\beta}'})$ where 
    $t_{\beta}'$ is a real number bigger than $1$ if $S$ is a finite set (see \cite{DFI}) or 
    $t_{\beta} = S$ if $S$ is an infinite set. 
    Then \cite[Theorem 2.4]{DFI} and \cref{fun_grp_free_group} imply 
    $I_\ut \MMM_\alpha I_\ut \cong L(F_{t_{\beta}})$ where 
    $t_{\beta}$ is a real number bigger than $1$ if $S$ is a finite set or 
    $t_{\beta} = S$ if $S$ is an infinite set.

    Note that $I_\ut \MMM I_\ut = L(F_S) *_{\beta \in \SSS \setminus \{\ut\}} I_\ut \MMM_\beta I_\ut$. 
    Let $\SSS = \cup_{j \in \Theta} S_j$ where $\Theta$ is a set such that $|\Theta| = |\SSS|$ 
    and $\{S_j\}_{j \in J}$ are disjoint countably infinite subsets of $\SSS$. 
    If $S$ is a finite set, by \cite[Proposition 4.3]{DF}, $I_\ut \MMM I_\ut \cong *_{j \in J}L(F_\N) \cong
    L(F_\SSS)$. Otherwise, $I_\ut \MMM I_\ut \cong *_{j \in J}L(F_S) \cong
    L(F_\SSS)$. Finally, $\MMM$ is a type II$_\infty$ factor because the central carrier of $I_\ut$ is $I$. 
\end{proof}

\begin{lemma}\label{simplify_step_one}
    For every $S$ such that $1 \leq |S| \leq |\SSS|$, let
    \begin{align*}
        (\MMM, \mathcal{E}) = (\MMM(F_S), E_S) *_{\AAA(\ut),\, \alpha \in \SSS \setminus \{\ut\}, 
        \xi \in \OOO_{\alpha}^{\ut * \alpha}}
    (\stackrel[2]{
        I_{\ut}+p_{\alpha}}{M_2(\C)} 
        \oplus \stackrel[d_{\alpha}-1]{q_{\alpha}}{\C}
        \oplus_{\beta \in \SSS \setminus \{\ut, \alpha\}} 
        \stackrel[d_{\beta}]{I_{\beta}}{\C}),
    \end{align*}
    where $p_\alpha + q_\alpha = I_\alpha$. Then $\MMM$ is a II$_\infty$ factor 
    with a tracial weight $\tau \circ \mathcal{E}$ and $I_\ut \MMM I_\ut \cong L(F_{\SSS})$.
\end{lemma}

\begin{proof}
    Invoke \cite[Lemma 2.2]{DF} instead of \cite[Proposition 1.7 (i)]{DF}, and 
    the lemma can be proved by the same argument used in the proof of \cref{step_one}. 
\end{proof}

The following fact is well-known, we sketch the proof for the readers' convenience.

\begin{lemma}\label{projection_case}
    Let $0 < t < 1$ and $\MMM = *_{i \in \N} (\stackrel[t]{}{\C} \oplus \stackrel[1-t]{}{\C})$.
    Then $\MMM \cong L(F_\N)$.
\end{lemma}

\begin{proof}
    We assume $t \geq 1/2$. Then by \cite[Theorem 1.1]{DF},
    \begin{align*}
        (\stackrel[t]{}{\C} \oplus \stackrel[1-t]{}{\C})*(\stackrel[t]{}{\C} \oplus \stackrel[1-t]{}{\C})
        \cong (\stackrel[2-2t]{}{L(F_1)\otimes M_2(\C)})\oplus \stackrel[\max\{2t-1,0\}]{}{\C}.
    \end{align*}
    By \cite[Theorem 4.6]{DF},  
    \begin{align*}
        ((\stackrel[2-2t]{}{L(F_1)\otimes M_2(\C)})\; \oplus \stackrel[\max\{2t-1,0\}]{}{\C}) &*
        ((\stackrel[2-2t]{}{L(F_1)\otimes M_2(\C)})\; \oplus \stackrel[\max\{2t-1,0\}]{}{\C})\\
        &\cong \stackrel[1-\max\{4t-3,0\}]{}{L(F_{s_1})} \oplus  \stackrel[\max\{4t-3,0\}]{}{\C},
    \end{align*}
    where $s_1 \geq 1$. By \cite[Proposition 2.4]{DF}, there exist $n \in \N$ such that 
    \begin{align*}
        *_{i =1 \ldots n} (\stackrel[t]{}{\C} \oplus \stackrel[1-t]{}{\C}) \cong L(F_{s_2}),
    \end{align*}
    where $s_2 \geq 1$. Then \cite[Propositions 4.3 and 4.4]{DF} imply the result.         
\end{proof}

\begin{lemma}\label{finite_free_product}
    Let $(\MMM, \tau)$ be a W$^*$-noncommutative probability space with normal faithful tracial state $\tau$ and let
    \begin{align*}
        \AAA = \; \stackrel[2/(1+d)]{e_{11}+ e_{22}}{M_2(\C)} \; \oplus \stackrel[(d - 1)/(1+ d)]{q}{\C}
    \end{align*}  
    be a unital subalgebra of $\MMM$, and $d > 1$.
    We use $\{e_{ij}\}_{i,j}$ to denote the canonical matrix units of $M_2(\C)$ in $\AAA$. 
    Let $\AAA_0$ be the abelian subalgebra generated by $e_{11}$ and $P = e_{22}+q$ and 
    $E$ be the trace-preserving conditional expectation from $\MMM$ onto $\AAA_0$.

    Assume that $\{e_{22}\} \cup \{p_s\}$ is a *-free family of projections in $(e_{22}+q)\MMM (e_{22}+q)$
    such that $\tau(p_s) = 1/(1+ d)$. We denote by $\NNN$ the von Neumann subalgebra of 
    $(e_{22}+q)\MMM (e_{22}+q)$ generated by $\{e_{22}\} \cup \{p_s\}_{s \in S}$.
    Then there exist partial isometries $v_s$ in $\NNN$ such that 
    $v_s^*v_s= p_s$ and $v_s v_s^* = e_{22}$.
    Suppose that $\{w_s\}_{s \in S}$ is a free family of Haar unitaries in 
    $e_{22}\MMM e_{22}$ which is *-free with $e_{22}\NNN e_{22}$.  
    Then 
    \begin{align*}
        ([\AAA \cup (\cup_{s \in S} \{e_{11}, e_{12}w_sv_s, v_s^* w_s^*e_{21}, p_s,P\})]'', E)
        \cong (\AAA, E)*_{\AAA_0} (*_{\AAA_0, s \in S} (\AAA, E)).
    \end{align*}
\end{lemma}

\begin{proof}
    The existence of $v_s$ is implied by \cite[Theorem 1.1]{DF}.
    It is clear that 
    \begin{align*}
        \{e_{11}, e_{12}w_sv_s, v_s^* w_s^*e_{21}, p_s,P\}'' \cong \AAA.
    \end{align*} 
    We could assume that $0 \notin S$ and denote by $Z_0 = \{e_{12}, e_{21}, de_{22} - P\}$ and
    $Z_s = \{e_{12}w_sv_s, v_s^* w_s^*e_{21}, dp_s - P\}$ for every $s\in S$.
    Note  that $E(a)= (1+d)\tau(e_{11}a e_{11})e_{11} + [(1+d)/d] \tau(PaP)P$.
    We only need to show that $E(t_1 \cdots t_n) = 0$. 
    and $t_i \in Z_{k(i)}$ such that $k(1) \neq k(2) \neq \cdots \neq k(n)$, where
    $k(i) \in \{0\} \cup S$, $i\in1,\ldots,n$, $n\geq 1$.

    Note that $E(t_1 \cdots t_n)=0$ if $t_1 \cdots t_n \in e_{11} \MMM P$ or $P\MMM e_{11}$.
    Therefore we assume that $0 \neq t_1 \cdots t_n$ is in $e_{11}\MMM e_{11}$ or in $P\MMM P$ and show that
    $\tau(t_1 \cdots t_n) = 0$.

    Case 1: $t_i \notin \cup_{s \in S} \{e_{12}w_sv_s, v_s^* w_s^* e_{21}\}$. 
    If $e_{12}$ and $e_{21}$ are not in $\{t_1, \ldots, t_n\}$, then it is clear that
    $\tau(t_1 \cdots t_n) = 0$.
    If $e_{12}$ or $e_{21}$ is in $\{t_1, \ldots, t_n\}$, then $t_1 = e_{12}$, $t_n = e_{21}$.
    Note that $t_i = d e_{22} - P$ or $t_i \in \{d p_{s}-P\}_{s\in S}$, $i = 2, \ldots, n-1$
    and $t_{n-1} \neq d e_{22}-P$, $i = 2, \ldots, n-1$.
    Since $\{p_s\}_{s \in S}$ and $e_{22}$ are free, then
    \begin{align*}
        \tau(t_1 \cdots t_n) = \tau(t_2 \cdots t_{n-1}e_{22}) = 0. 
    \end{align*}

    Case 2: If the number of the elements in 
    $\{i: 1\leq i\leq n, t_{i} \in  \cup_{s \in S} \{e_{12}w_sv_s\}\}$ and
    $\{i: 1\leq i\leq n, t_{i} \in \cup_{s \in S} \{v_s^* w_s^* e_{21}\}\}$ are different, 
    then $\tau(t_1 \cdots t_n) = 0$ since $\{w_s\}_{s \in S}$ and 
    $e_{22} \NNN e_{22} = e_{22} (\AAA \cup \NNN)'' e_{22}$ are free. 
   
    Therefore, if $\tau(t_1 \cdots t_n) \neq 0$, there must exist $i_1 < i_2$ such that
    \begin{align*}
        t_{i_1} = v_s^* w_s^* e_{21}, t_{i_2} = e_{12}w_sv_s \mbox{ or }
        t_{i_1} = e_{12}w_sv_s, t_{i_2} = v_s^* w_s^* e_{21}. 
    \end{align*}
    and $t_{i} \notin \cup_{s \in S} \{e_{12}w_sv_s, v_s^* w_s^* e_{21}\}$, $i_1 < i < i_2$.

    Assume that $t_{i_1} = v_s^* w_s^* e_{21}, t_{i_2} = e_{12}w_sv_s$. 
    Since $t_{i_1} \cdots t_{i_2} \neq 0$, $t_{(i_1 + 1)} = e_{12}$,
    $t_{(i_2 - 1)} = e_{21}$, $t_{i} \in \{d p_s - P\}_{s \in S} \cup \{de_{22} - P\}$
    $i_1 + 1 < i < i_2 -1$ and $t_{(i_1 + 2)} \neq de_{22} - P$, $t_{(i_1 -2)} \neq de_{22} - P$,
    then by Case 1, we have
    \begin{align*}
        \tau(e_{22}t_{(i_1 +2)} \cdots t_{(i_2 -2)}e_{22}) = 0.
    \end{align*}
    Now assume that $t_{i_1} = e_{12}w_sv_s, t_{i_2} = v_s^* w_s^* e_{21}$. 
    Since $t_{i_1} \cdots t_{i_2} \neq 0$, $t_{j} \notin \{e_{12}, e_{21}\}$, $i_1 \leq j \leq i_2$.
    Then by Case 1,
    \begin{align*}
        \tau(v_s t_{(i_1 +2)}\cdots t_{(i_2 -2)}v_s^*)
        = \tau(p_s t_{(i_1 +2)}\cdots t_{(i_2 -2)}) = 0.
    \end{align*} 
    This, however, implies that $\tau(t_1 \cdots t_n) = 0$.
\end{proof}

\begin{lemma}\label{ext_lemma}
    Suppose that $S$ is a set such that $|S| = |\SSS|$, let
   \begin{align*}
       \MMM = *_{\AAA_0, s \in S}
           (\stackrel[2]{I_{1}+ p}{M_2(\C)}\oplus \stackrel[d-1]{q}{\C})
   \end{align*} 
    where $\AAA_0= \stackrel[1]{I_1}{\C} \oplus \stackrel[d]{I_2}{\C}$ and $I_2 =p+q$, $d\geq 1$.
    Then $\MMM$ is a II$_1$ factor and $I_1 \MMM I_1 \cong L(F_\SSS)$.
\end{lemma}

\begin{proof}
    Let $s_0 \in S$.

    Case 1: $d = 1$. Then a similar argument as in the proof of \cref{finite_free_product}
    implies that $I_1 \MMM I_1$ is generated by a free family of 
    Haar unitaries $\{u_s\}_{s \in S \setminus \{s_0\}}$. Thus $I_1 \MMM I_1 \cong L(F_\SSS)$.

    Case 2: $d > 1$. By \cref{finite_free_product}, $I_2 \MMM I_2$ is generated by
    a free family of projections $\{p_s\}$ and a family of Haar unitaries
    $\{w_s\} \subset p_{s_0}\MMM p_{s_0}$ such that $\tau(p_s) = 1$, $\{w_s\}$ and 
    $p_{s_0}\{p_s\}_{s\in S}''p_{s_0}$ are free. By \cref{projection_case}, 
    $\{p_s\}_{s\in S}'' \cong L(F_\SSS)$. Then \cref{fun_grp_free_group} implies
    $p_{s_0} \MMM p_{s_0} \cong  L(F_\SSS) * L(F_\SSS) \cong L(F_\SSS)$.
    Note that the central carrier of $p_{s_0}$ is $I_1 + I_2$, thus $\MMM$ is a factor and 
    $I_1 \MMM I_1 \cong L(F_\SSS)$.
\end{proof}

\begin{theorem}\label{thm_freecorner}
    $I_\ut \Phi(H(\ut))'' I_\ut \cong L(F_\SSS)$.
\end{theorem}

\begin{proof}
    First note that $|\SSS \times \SSS| = |\SSS|$ (we assume the Axiom of Choice). Let 
    \begin{align*}
        (\MMM, \mathcal{E}) = (\MMM(F_{\SSS}), E_{\SSS}) *_{\AAA(\ut),\, \alpha \in \SSS \setminus \{\ut\}, 
    \xi \in \OOO_{\alpha}^{\ut * \alpha}}
    (\stackrel[2]{I_{\ut}+ p_{\alpha}}{M_2(\C)}\oplus \stackrel[d_{\alpha}-1]{q_{\alpha}}{\C}
           \oplus_{\beta \in \SSS \setminus \{\ut, \alpha\}}\C, E),
    \end{align*}
    where $p_\alpha +q_\alpha = I_\alpha$. Note that
    \begin{align*}
        (\MMM, \mathcal{E}) = (\MMM(F_{\SSS}), E_{\SSS})*_{\AAA(\ut),\, \alpha \in \SSS \setminus \{\ut\}} \MMM_\alpha, 
    \end{align*}
    where
    \begin{align*}
        \MMM_\alpha = (\stackrel[d_\alpha]{I_\alpha}{L(F_\SSS)}
        \oplus_{\beta \in \SSS \setminus \{\alpha \}} 
        \stackrel[d_\beta]{I_\beta}{\C})*_{\AAA(\ut)}
         ((\stackrel[2]{
            I_{\ut} + p_{\alpha}}{M_2(\C)} 
        \; \oplus \hspace{-1mm} \stackrel[d_{\alpha}-1]{q_{\alpha}}{\C})
        \; \oplus_{\gamma \in \SSS \setminus \{\ut, \alpha\}} \stackrel[d_\gamma]{I_\gamma}{\C}).
    \end{align*}
    Note that $\MMM_\alpha \cong (\stackrel[d_\ut + d_\alpha]{I_\ut + I_\alpha}{L(F_\SSS)}
        \oplus_{\beta \in \SSS \setminus \{\ut, \alpha \}} 
        \stackrel[d_\beta]{I_\beta}{\C})$ by \cite[Lemma 4.3]{DI}.
    By \cref{ext_lemma} and \cref{iso_cor}, we have
    \begin{align*}
        (\MMM, \mathcal{E}) \cong 
        (\MMM(F_{\SSS}), E_{\SSS}) *_{\AAA(\ut),\, \alpha \in \SSS \setminus \{\ut\}, s \in S}
     (\stackrel[2]{I_{\ut}+ p_{\alpha}}{M_2(\C)}\oplus \stackrel[d_{\alpha}-1]{q_{\alpha}}{\C}
           \oplus_{\beta \in \SSS \setminus \{\ut, \alpha\}}\C,E),
    \end{align*}
    where $S$ is an index set such that $|S| = |\SSS|$. 
    Thus $(\MMM, \mathcal{E}) \cong *_{\AAA(\ut),\, s \in S}(\MMM, \mathcal{E})$.

    By \cref{unit_case}, \cref{diff_case}, \cref{step_one}, \cref{simplify_step_one}
    and \cref{iso_cor}, we have  
    \begin{align*}
        (\Phi(H(\ut))'', E) \cong (\MMM, \mathcal{E})
        *_{\AAA(\ut),\, \alpha \in \SSS,\, \beta_1 \neq \beta_2 \in \SSS,\, \beta_2 \neq \ut,\, \xi 
        \in \OOO^{\beta_2 * \alpha}_{\beta_1}} (\Phi(\{a \xi + b \overline{\xi}: a, b \in \C\})'', E).
    \end{align*}
    Thus $\Phi(H(\ut))'' \cong *_{\AAA(\ut),\, \alpha \in \SSS, \beta_1 \neq \beta_2 \in \SSS, \beta_2 \neq \ut, \xi 
    \in \OOO^{\beta_2 * \alpha }_{\beta_1}} \MMM_{\xi}$ where we set 
    \begin{align*}
        \MMM_\xi = (\MMM(F_{\SSS}), E_{\SSS})
        &*_{\AAA(\ut), \, \beta \in \SSS \setminus \{\ut,\beta_2\}}
           (\stackrel[2]{I_{\ut}+ p_{\beta}}{M_2(\C)}\oplus \stackrel[d_{\beta}-1]{q_{\beta}}{\C}
           \oplus_{\eta \in \SSS \setminus \{\ut, \beta\}}\C, E)\\
           &*_{\AAA(\ut)}
           (\stackrel[2d_{\beta_1}]{I_{\beta_1}+ p_{\beta_2}}{(L(F_1) \otimes M_2(\C))}\oplus 
           \stackrel[d_{\beta_2}-d_{\beta_1}]{q_{\beta_2}}{\C} \oplus_
           {\gamma \in \SSS \setminus \{\beta_1, \beta_2\}}\C, E)
    \end{align*}
    if $d_{\beta_1} \leq d_{\beta_2}$, and $\MMM_\xi$ is similarly defined in the case $d_{\beta_2} \leq d_{\beta_1}$.
    Arguing as in the proof of \cref{step_one}, we have that $\MMM_\xi$ is a factor and
    $I_\ut \MMM_\xi I_\ut \cong L(F_\SSS)$.
 
    By \cref{iso_cor}, $\Phi(H(\ut))'' 
    \cong *_{\AAA(\ut),\, \alpha \in \SSS, \beta_1 \neq \beta_2 \in \SSS, \beta_2 \neq \ut, \xi 
    \in \OOO^{\beta_2 * \alpha}_{\beta_1}}(\MMM, \mathcal{E}) \cong (\MMM, \mathcal{E})$
    and \cref{simplify_step_one} implies the result.
\end{proof}

\section{Realization as endomorphisms of $\Phi(H(\ut))''$}\label{section_endos}

Let $\CC$ be a rigid C$^*$-tensor category with simple unit. 
In this section, we assume that the spectrum $\SSS$ is infinite and show that $\CC$ can be realized 
as endomorphisms of the factor $\Phi(H(\ut))''$ constructed in \cref{section_CstartoOA}. 
We first introduce for convenience an auxiliary algebra 
living on a bigger Hilbert space. 

Let $W(\CC)$ be the \lq\lq words" constructed from the objects of $\CC$ (finite sequences of letters in $\CC$).  
We denote by $W_1 * W_2$ the concatenation of two words $W_1, W_2\in W(\CC)$. 
Two words are equal if they have the same length and they are letterwise equal.
Note that, e.g., $\ut*\ut$ and $\ut \ut =\ut \otimes \ut = \ut$ 
are different in $W(\CC)$, since $\ut \ut$ is a one-letter word and $\ut*\ut$ is a two-letter word. 
As in the proof of Mac Lane's coherence theorem \cite[XI. 3]{ML}, we define 
the arrows between words $W_1$ and $W_2$, denoted by $Hom(W_1, W_2)$, to be
$Hom(\iota(W_1), \iota(W_2))$ with the composition just as in $\CC$ and with tensor multiplication 
denoted by $f_1*f_2 = f_1\otimes f_2\in Hom(W_1*Z_1,W_2*Z_2)$ if $f_1\in Hom(W_1,W_2)$, $f_2\in Hom(Z_1,Z_2)$. 
Here $\iota$ is the map
\begin{align*}
    \iota: \emptyset \mapsto \ut, \quad X_1 * X_2 * \cdots * X_n \mapsto X_1 X_2 \cdots X_n,
\end{align*}
where $X_i \in \CC$, $i=1,\ldots,n$, $n\geq 1$, and $\emptyset$ denotes the empty word.
Thus $W(\CC)$ is a tensor category, moreover $\CC$ can be viewed as the full C$^*$-subcategory of $W(\CC)$ 
(not tensor anymore) consisting of all one-letter words.
The reason for which we introduce words is to keep track of the distinction between
$X*Y$ and $X*Z$, when $Y\neq Z$, even if $XY = XZ$ in $\CC$.

From now on, we use $\Phi$ and $L^2(\Phi)$ to denote $\Phi(H(\ut))''$ and $L^2(\Phi(H(\ut))'', \tau \circ E)$,
unless stated otherwise. Recall that $L^2(\Phi)$ is a $\Phi$-$\Phi$ bimodule and $I_\beta\in\Phi$, $\beta\in\SSS$. Let 
$$\HHH(\CC) = \bigoplus_{W \in W(\CC)} \HHH^{W},$$ 
where 
\begin{align*}
    \HHH^{W}= \bigoplus_{\beta,\gamma \in \SSS} Hom(\beta, W*\gamma) \otimes (I_{\beta} L^2(\Phi)). 
\end{align*} 
Here by $\oplus$ and $\otimes$ we mean the Hilbert space completions of the algebraic direct sums and tensor products. 
The inner product on $\HHH(\CC)$ is given by
\begin{align*}
    \braket{u_1\otimes \xi_1}{u_2 \otimes \xi_2} = \frac{1}{d_{\beta}}
    \braket{u_1}{u_2}\braket{\xi_1}{\xi_2}=\braket{\xi_1}{(u_1^*u_2)\xi_2}, 
\end{align*}
where $\xi_1,\xi_2 \in I_{\beta}L^2(\Phi)$, $u_1, u_2 \in Hom(\beta, W*\gamma)$, and 
$\braket{u_1}{u_2}$ is given as in \cref{section_CstartoOA} by \cref{hom_inner_prod}.
To simplify the notation, we will use $u \xi$ to denote $u \otimes I_{\beta} \xi$ for every
$u \in Hom(\beta, W*\gamma)$ and $\xi \in L^2(\Phi)$. Moreover, note that $L^2(\Phi)$ is canonically identified with $\HHH^\emptyset$.

Each $f \in Hom(W_1*\beta_1, W_2 * \beta_2 )$, for $\beta_1,\beta_2 \in \SSS$, $W_1,W_2 \in W(\CC)$, 
acts on $\HHH(\CC)$ as follows. On the linear span of the $u\xi$'s as above
let 
$$f\cdot (u\xi) = \delta_{W_1 * \beta_1,\, W * \gamma}\, (f\circ u)\xi, \quad u \in Hom(\beta, W*\gamma).$$ 
It is routine to check that
\begin{align*}
    \braket{\sum_{i}^{n} (f \circ u_i)\xi_i}{\sum_{j}^{n} (f \circ u_j)\xi_j}
    \leq \|f\|^2\sum_{i,j}^{n}\braket{\xi_i}{(u_i^*u_j)\xi_j},
\end{align*}
since $(u_i^*f^*fu_j)_{i,j} \leq \|f\|^2(u_i^*u_j)_{i,j}$.
Thus the map extends by continuity to a bounded linear operator in $\BBB(\HHH(\CC))$, again denoted by $f$.

Let $f_i \in Hom(W_i*\alpha_i, Z_i* \beta_i)$, for $\alpha_i, \beta_i \in \SSS$, $W_i$, $Z_i \in W(\CC)$, 
$i = 1, 2$. By the above definition, for every $\zeta$, $\zeta_1,\zeta_2 \in \HHH(\CC)$, we have
$$f_1\cdot (f_2 \cdot \zeta) = \delta_{W_1 * \alpha_1,\, Z_2 * \beta_2} (f_1 \circ f_2) \cdot \zeta$$
and 
$$\braket{\zeta_1}{f_1 \cdot \zeta_2} =  \braket{f_1^* \cdot \zeta_1}{\zeta_2}.$$
In particular, for $u \in Hom(\beta, W*\gamma)$ and $\xi \in L^2(\Phi) \subset \HHH(\CC)$, we have
$u \cdot \xi = u \xi$. Therefore, we write $f \cdot \zeta$ as $f \zeta$ from now on. 

It is also clear that 
\begin{align}\label{HC_decomp}
    \HHH(\CC) = \bigoplus_{W \in W(\CC),\, \beta \in \SSS} \bigoplus_{u \in \OOO^{W * \SSS}_{\beta}} u L^2(\Phi),
\end{align}
where $\OOO^{W * \SSS}_{\beta} = \cup_{\gamma\in\SSS} \OOO^{W * \gamma}_{\beta}$ (disjoint union)
and $\OOO^{W * \gamma}_{\beta}$ is a fixed orthonormal basis of isometries in 
$Hom (\beta, W*\gamma)$, for every $\beta, \gamma\in\SSS$, $W\in W(\CC)$, while
$\OOO^{W * \gamma}_{\beta} = \emptyset$ if $dim Hom(\beta, W*\gamma) = 0$.
 
For every $\xi \in \bfrac{\beta_2*\alpha}{\beta_1} \subset H(\ut)$, $\alpha \in \Lambda$, $\beta_1,\beta_2 \in \SSS$, 
the operator $\Gamma(\xi) \in \Phi$ acts 
from the left on $\HHH(\CC)$ in the following way
\begin{align*} 
    \begin{cases}
        \Gamma(\xi)(\zeta)=\Gamma(\xi)\zeta &  \mbox{if $\zeta \in \HHH^{\emptyset}$},\\
        \Gamma(\xi)(u\zeta)=0 &  \mbox{if $u\zeta \notin \HHH^{\emptyset}$},
    \end{cases}
\end{align*}
and we denote again by $\Gamma(\xi)$ the operator acting on $\HHH(\CC)$.

\begin{definition}
Let $\Phi(\CC)$ be the von Neumann subalgebra of $\BBB(\HHH(\CC))$ generated by 
$\{f \in Hom(W_1*\beta_1, W_2*\beta_2): \beta_1,\beta_2 \in \SSS, W_1,W_2 \in W(\CC)\}$ and
    $\{\Gamma(\xi): \xi \in \bfrac{\beta_2 * \alpha}{\beta_1}, \alpha \in \Lambda, \beta_1,\beta_2 \in \SSS\}$.
\end{definition}

Let $P_{W} : \HHH(\CC) \rightarrow \HHH^{W}$ be the orthogonal projection onto $\HHH^{W}$ for $W \in W(\CC)$. 
We have the following easy result and its proof is omitted.

\begin{lemma}\label{lem_phione}
    The algebra $\Phi$ constructed in \cref{section_CstartoOA} from $X=\ut$ is a corner of $\Phi(\CC)$, namely
    $\Phi \cong P_{\emptyset}\Phi(\CC)P_{\emptyset}$ and a canonical *-isomorphism is given by
    \begin{align*}
         I_\beta \mapsto I_\beta|_{\HHH^\emptyset}, \quad \Gamma(\xi) 
        \mapsto \Gamma(\xi)|_{\HHH^{\emptyset}},
    \end{align*}
    for every $\xi \in \bfrac{\beta_2*\alpha}{\beta_1}$, where $\alpha \in \Lambda$, $\beta_1,\beta_2 \in \SSS$. 
\end{lemma}

%
%

From now on, we identify $L^2(\Phi)$ and $\Phi$ with $\HHH^{\emptyset}$ and $P_{\emptyset}\Phi(\CC)P_{\emptyset}$
respectively.

\begin{theorem}\label{X_equl}
    For every word $W \in W(\CC)$ (hence in particular for every object $X\in\CC$ of the category, regarded as a one-letter word), the projection 
    $P_W$ is equivalent to $P_\emptyset$ in $\Phi(\CC)$. 
    Thus $P_W \Phi(\CC) P_W \cong P_\emptyset \Phi(\CC) P_\emptyset = \Phi$. 
\end{theorem}

\begin{proof}
    Note that $P_\emptyset = \sum_{\beta \in \SSS} I_{\beta}$ and 
    $P_W = \sum_{\beta \in \SSS} \sum_{u \in \OOO^{W * \SSS}_\beta} uu^*$.
    Since the central carrier of $P_\emptyset$ is $I$, $\Phi(\CC)$ is either a 
    type II$_\infty$ or a type III factor by \cite[Proposition 5.5.6, Corollary 9.1.4]{KRII}.

    Assume first that $\Phi(\CC)$ is a type III factor. By \cite[Corollary 6.3.5]{KRII},
    $I_\beta$ and $\sum_{u \in \OOO^{W*\SSS}_\beta} uu^*$ are equivalent for every $\beta\in\SSS$, since
    they are both countably decomposable projections. Therefore $P_W$ is equivalent to $P_\emptyset$.

    Assume now that $\Phi(\CC)$ is type II$_\infty$. Let $\SSS = \cup_{i \in \Theta}\SSS_i$
    be a partition of $\SSS$ such that each $\SSS_i$ contains countably many elements and $|\SSS|= |\Theta|$.
    Note that each projection $\oplus_{u \in \OOO_\beta^{W * \SSS}} uu^*$ is a finite projection and
    equivalent to a subprojection of $\oplus_{\gamma \in \SSS_{\theta(\beta)}} I_\gamma$, 
    where $\theta$ is a one-to-one correspondence between $\SSS$ and $\Theta$. 
    This implies that $P_W$ is equivalent to $P_\emptyset$.
\end{proof}

Let $\Delta$ be the modular operator of $\Phi$ w.r.t. $\tau \circ E$. Then we can define a one-parameter 
family of unitaries in $\BBB(\HHH(\CC))$ as follows
\begin{align*}
    U_t: u \zeta \rightarrow u(\Delta^{it}\zeta), \quad t\in\R,
\end{align*}
for every $\zeta \in L^2(\Phi), u \in \OOO^{W * \SSS}_{\beta}$.
It is easy to check that
\begin{align*}
    U_t f U_t^* = f, \quad U_t \Gamma(\xi) U_t^* = \lambda_{\alpha}^{it} \Gamma(\xi), \quad t\in\R,
\end{align*}
where $f \in Hom(W_1 * \beta_1, W_2 * \beta_2)$ and $\xi \in \bfrac{\beta_2 * \alpha}{\beta_1}$, 
for $\alpha \in \Lambda$, $\beta_1,\beta_2 \in \SSS$, $W_1$, $W_2 \in W(\CC)$. 
Therefore $\text{Ad} U_t \in Aut(\Phi(\CC))$.

Note that $\{uI_\beta : u \in \OOO^{W * \SSS}_\beta, I_\beta \in L^2(\Phi), W \in W(\CC), \beta \in \SSS\}$
is a total set of vectors for $\Phi(\CC)'$. Thus
\begin{align}\label{varphi_def}
    \varphi(A) = \sum_{\beta \in \SSS, W \in W(\CC)} \sum_{u \in \OOO^{W * \SSS}_\beta} 
    \braket{u I_{\beta}}{A (u I_\beta)}, \quad A \in \Phi(\CC),\,A \geq 0, 
\end{align}
defines a n.s.f. weight on $\Phi(\CC)$. Moreover, $\varphi(U_t  \cdot U_t^*) = \varphi(\cdot)$ 
for every $t\in\R$, since $U_t(u I_\beta) = u I_\beta$. Now we show that $U_t$, $t\in\R$, implement the modular group of $\Phi(\CC)$
w.r.t. $\varphi$.

\begin{lemma}\label{KMS_aut}
    For every $A \in \Phi(\CC)$ we have $\sigma_t^{\varphi}(A) = U_t A U_t^*$, $t\in\R$. 
\end{lemma}

\begin{proof}
    Note that $\varphi(uA) = \braket{I_\beta}{A (uI_\beta)} = \braket{I_\beta}{(Au)I_\beta}= 
    \varphi(A u)$ for every $u \in \OOO^{W * \SSS}_{\beta}$, $\beta \in \SSS$, $W \in W(\CC)$. 
    Then \cite[Theorem 2.6, VIII]{TKII} implies that the von Neumann subalgebra generated by 
    $\{f\in Hom(W_1*\beta_1, W_2*\beta_2):\beta_1,\beta_2 \in \SSS, W_1,W_2 \in W(\CC)\}$ is contained in
    the centralizer $\Phi(\CC)_{\varphi}$ and $\sigma^{\varphi}_t(f) = f = U_t f U_t^*$.

    Since $E(B) = \sum_{\beta \in \SSS} d_\beta^{-1} \braket{I_\beta}{ B I_\beta}I_\beta$, 
    for every $B \in \Phi$, we have $\tau \circ E(B) = 
    \varphi(B)$. Then \cite[Theorem 1.2, VIII]{TKII}
    implies $\sigma^{\varphi}_t(\Gamma(\xi)) = U_t \Gamma(\xi) U_t^*$ for every $\xi \in 
    \bfrac{\beta_2*\alpha}{\beta_1}$. By the definition of $\Phi(\CC)$, the lemma is proved.
\end{proof}

\begin{lemma}\label{embedding_ut}
    For every word $W \in W(\CC)$ there exists a non-unital injective *-endomorphism $\rho_W$ of $\Phi(\CC)$ such that 
    $$\rho_W: f \in Hom(W_1 * \beta_1 , W_2 * \beta_2) \mapsto I_W * f  
    \in Hom(W* W_1 *\beta_1, W* W_2 *\beta_2)$$ 
    and 
    $$\rho_W: \Gamma(\xi) \mapsto \sum_{\gamma_1, \gamma_2 \in \SSS}
    \sum_{u \in \OOO_{\gamma_2}^{W * \beta_2},\, v \in \OOO^{W * \beta_1}_{\gamma_1}}
    u \Gamma( u^* \otimes I_\alpha \circ I_W \otimes \xi \circ v) v^*$$ 
    for every $\xi \in \bfrac{\beta_2 * \alpha}{\beta_1}$, where $\alpha\in\Lambda$, $\beta_1,\beta_2\in \SSS$, $W_1,W_2\in W(\CC)$.
\end{lemma}

\begin{proof}
    Let $d_W = d_{\iota(W)}$. It is easy to check that $\varphi(f) = 1/d_W \varphi(I_W * f)$.
    Let $\xi_i \in \bfrac{\beta_{i+1}*\alpha_i}{\beta_i}$ where $\beta_i \in \SSS$, $\alpha_i \in \Lambda$, 
    $i = 1, \ldots, m+1$, $m\geq 1$. Note that $\varphi(\Gamma(\xi_m) \cdots \Gamma(\xi_1)) =0= d_W^{-1}
    \varphi(\rho_W(\Gamma(\xi_m)) \cdots \rho_W(\Gamma(\xi_1)))$ unless $m = 2n$ and 
    $\beta_1 = \beta_{2n+1}$.
 
    We claim that $\varphi(\Gamma(\xi_{2n}) \cdots \Gamma(\xi_1)) = d_W^{-1}
    \varphi(\rho_W(\Gamma(\xi_{2n})) \cdots \rho_W(\Gamma(\xi_1)))$ if $\beta_1 = \beta_{2n+1}$. 
    By the definition of $\tau$ and $E$ (see \cref{def_E}), we have
    \begin{align*}
        \tau \circ E(\Gamma(\xi_{2n}) \cdots \Gamma(\xi_1))= 
        d_{\beta_1}\sum_{\substack{(\{j_1, i_1\}, \ldots, \{j_n, i_n\})\\ \in 
        NC_2(2n), i_k < j_k}}\prod_{k=1}^{n} \frac{\sqrt{\lambda_{\alpha_{j_k}}}}{d_{\beta_{i_k}}}
        \braket{\overline{\xi}_{j_k}}{\xi_{i_k}},
    \end{align*}
    where $NC_2(2n)$ stands for the set of non-crossing pairings of $\{1, \ldots, 2n\}$, i.e.,
    there does not exist $1 \leq i_k < i_l < j_k < j_l \leq 2n$ (see \cite[Notation 8.7]{Nica-S} for more details).
    Thus we only need to show that, for every $(\{j_1, i_1\}, \ldots, \{j_n, i_n\}) \in NC_2(2n)$, 
    $d_{\beta_1}d_W\prod_{k=1}^{n}\frac{1}{d_{\beta_{i_k}}}
    \braket{\overline{\xi}_{j_k}}{\xi_{i_k}}$ equals
    \begin{align*}
        \sum_{\gamma_1, \ldots, \gamma_{2n}}d_{\gamma_1} 
        \sum_{\substack{u_i \in \OOO_{\gamma_i}^{W * \beta_i},\\i = 1, \ldots,2n}}
        \prod_{k=1}^{n} \frac{1}{d_{\gamma_{i_{k}}}}
        \braket{u_{j_{k}}^* \otimes I_{\overline{\alpha}_{j_{k}}} \circ I_{W} \otimes
        \overline{\xi}_{j_{k}} \circ u_{j_k+1}}{u_{i_{k}+1}^* \otimes I_{\alpha_{i_k}} 
        \circ I_{W} \otimes \xi_{i_k} \circ u_{i_k}},
    \end{align*}
    here $u_{2n+1} = u_1$. We show that the two expressions are equal by induction on $n$. 

    If $n = 1$, then 
    \begin{align*}
        &\sum_{\gamma_1, \gamma_2} 
        \sum_{u_1 \in \OOO_{\gamma_1}^{W * \beta_1}, u_2 \in \OOO_{\gamma_2}^{W * \beta_2}}
        \braket{u_{2}^* \otimes I_{\overline{\alpha}_{2}} \circ I_{W} \otimes
        \overline{\xi}_{2} \circ u_{1}}{u_{2}^* \otimes I_{\alpha_{1}} 
        \circ I_{W} \otimes \xi_{1} \circ u_{1}}\\
        =&\sum_{\gamma_1}\sum_{u \in \OOO^{W * \beta_1}_{\gamma_1}}\frac{d_{\gamma_1}}{d_{\beta_1}}
        \braket{\overline{\xi}_2}{\xi_1} = d_W \braket{\overline{\xi}_2}{\xi_1}.
    \end{align*}

    Assume the claim holds for $n = 2l$. For $n = 2l + 2$, let $k \in \{1, \ldots, 2l+2\}$ such that $j_k = i_k +1$. Note that
    \begin{align*}
        &\sum_{\gamma_{i_k + 1}} 
        \sum_{u_{i_k +1} \in \OOO_{\gamma_{i_k + 1}}^{W * \beta_{i_k + 1}}}\frac{1}{d_{\gamma_{i_k}}}
        \braket{u_{i_k + 1}^* \otimes I_{\overline{\alpha}_{2}} \circ I_{W} \otimes
        \overline{\xi}_{i_k + 1} \circ u_{i_k + 2}}{u_{i_k + 1}^* \otimes I_{\alpha_{1}} 
        \circ I_{W} \otimes \xi_{i_k} \circ u_{i_k}}\\
        =&\frac{1}{d_{\beta_{i_k}}} \delta_{\beta_{i_k}, \beta_{i_{k}+2}} \delta{u_{i_{k}+2}, u_{i_k}}
        \braket{\overline{\xi}_{j_k}}{\xi_{i_k}}.
    \end{align*} 
    Then the inductive hypothesis implies the claim.
 
    Let $\Phi_0$ be the *-subalgebra of $\Phi(\CC)$ generated by 
    $\{f:Hom(W_1 * \beta_1, W_2 *\beta_2), \beta_1,\beta_2 \in \SSS, W_1, W_2 \in 
    W(\CC)\}$ and $\{\Gamma(\xi): \xi \in \bfrac{\beta_2 * \alpha}{\beta_1}, \beta_1,\beta_2\in \SSS,
    \alpha \in \Lambda\}$. Moreover, let $\rho_W(\Phi_0)$ be the *-subalgebra of $P\Phi(\CC)P$ generated by
    $\{\rho_W(f):Hom(W_1 * \beta_1, W_2 * \beta_2), \beta_1,\beta_2 \in \SSS, W_1,W_2 \in W(\CC)\}$
    and $\{\rho_W(\Gamma(\xi)): \xi \in \bfrac{\beta_2 * \alpha}{\beta_1}, \beta_1,\beta_2 \in \SSS, \alpha \in \Lambda\}$,
    where $P = \sum_{V \in W(\CC)}P_{W*V}$. 

    By \cref{KMS_aut} and \cite[Lemma 2.1]{ILP}, $\Phi_0$ and $\rho_W(\Phi_0)$ are respectively dense in 
    $L^2(\Phi(\CC), \varphi)$ and $L^2(\rho_W(\Phi_0)'', d_W^{-1} \varphi)$. By the discussion above, 
    \begin{align*}
        \varphi(u\Gamma(\xi_n) \cdots \Gamma(\xi_1)v^*) = d_W^{-1}
        \varphi(\rho_W(u)\rho_W(\Gamma(\xi_n)) \cdots \rho_W(\Gamma(\xi_1))\rho_W(v^*)),
    \end{align*}
    where $\xi_i \in \bfrac{\beta_{i+1}*\alpha_i}{\beta_i}$, $u \in \OOO^{W_2 * \SSS}_{\beta_{n+1}}$
    and $v \in \OOO^{W_1 * \SSS}_{\beta_1}$, $\beta_i \in \SSS$, $\alpha_i \in \Lambda$, $i = 1, \ldots, n$, $n\geq 1$, and 
    $W_1$, $W_2 \in W(\CC)$.

    Thus there exists a unitary $U$ from $L^2(\Phi(\CC), \varphi)$ to 
    $L^2(\rho_W(\Phi_0)'', d_W^{-1}\varphi)$ such that 
    $U^*\pi(\rho_W(f))U = f$ and $U^* \pi(\rho_W(\Gamma(\xi)))U = \Gamma(\xi)$ where
    $\pi$ is the GNS representation of $\rho_W(\Phi_0)''$ with respect to $d_{W}^{-1}\varphi$.
\end{proof}

\begin{remark}\label{o_ind_re}
    With the notations used in \cref{embedding_ut}, it is routine to check that
    \begin{align*}
     \sum_{u \in \OOO_{\gamma_2}^{W * \beta_2},\, v \in \OOO^{W * \beta_1}_{\gamma_1}}
      u\Gamma( u^* \otimes I_\alpha \circ I_W \otimes \xi \circ v)v^*   
    \end{align*}
    does not depend on the choice of the orthonormal basis of isometries in
    $Hom(\gamma_i, W * \beta_i)$ where the sum over $u$ and $v$ runs. Then it is 
    not hard to check that $\rho_{W_2} \circ \rho_{W_1} = \rho_{W_2*W_1}$ for every $W_1$, $W_2 \in W(\CC)$.
    
    Note that, even for objects $X,Y\in\CC$ regarded as one-letter words in $W(\CC)$, $\rho_{Y*X}$ and $\rho_{YX}$
    need not coincide as endomorphisms of $\Phi(\CC)$, where $Y*X$ is the concatenation of letters, while $YX$ 
    is the tensor multiplication in $\CC$. 
\end{remark}

\begin{lemma}\label{intre_lemma}
    For every pair of words $W_1$, $W_2 \in W(\CC)$ and arrow $f \in Hom(W_1,W_2)$, we denote 
    $$f \otimes I = \sum_{\beta \in \SSS, W \in W(\CC)} f * I_{W * \beta}\, \in\Phi(\CC),$$ 
    where $f * I_{W * \beta} \in Hom(W_1 * W * \beta, W_2 * W * \beta)$.
    Then the following intertwining relation in $End(\Phi(\CC))$ holds 
    $$(f \otimes I) \rho_{W_1}(A) = \rho_{W_2}(A) (f \otimes I),\quad \forall A \in \Phi(\CC).$$ 
    In particular, 
\begin{align*}
    (I_{X_1 * X_2}^{X_1 X_2} \otimes I) \rho_{X_1} \circ \rho_{X_2}(\cdot) = 
    \rho_{X_1 X_2}(\cdot) (I_{X_1 * X_2}^{X_1 X_2} \otimes I),
\end{align*}
where $I_{X_1 * X_2}^{X_1 X_2} = I_{X_1 X_2} \in Hom(X_1 * X_2, X_1 X_2)$, 
$X_1$, $X_2 \in \CC$.
\end{lemma}

\begin{proof}
    We only need to show that $(f \otimes I) \rho_{W_1}(\Gamma(\xi)) 
    = \rho_{W_2}(\Gamma(\xi))(f \otimes I)$ for every $\xi \in \bfrac{\beta_2*\alpha}{\beta_1}$,
    $\alpha \in \Lambda$, $\beta_i \in \SSS$, $i=1,2$. Without loss of generality, 
    we assume that $f = u_2 u_1^*$ where $u_i \in \OOO_{\beta}^{W_i}$, $\beta \in \SSS$, $i = 1,2$.
    Note that $\tilde{\OOO}^{W_i *\beta_j }_{\gamma} = \{ u \otimes I_{\beta_j} \circ v: u 
    \in \OOO_{\eta}^{W_i}, v \in \OOO^{\eta * \beta_j}_{\gamma}, \eta \in \SSS\}$
    is an orthonormal basis of $Hom(\gamma, W_i \beta_j)$, $i = 1,2$, $j=1,2$.
    Hence by \cref{o_ind_re} we have the desired statement.
\end{proof}

We choose and fix a partial isometry $V_X \in \Phi(\CC)$ 
such that $V_X^*V_X= P_X$ and $V_X V_X^* = P_{\emptyset}$ (whose existence is guaranteed by \cref{X_equl}) for every object $X \in \CC$.
Note that $\rho_X(P_{\emptyset}) = P_X$. 

\begin{definition}
For every object $X\in\CC$, let $F(X)$ be the unital *-endomorphism of $\Phi$
($=P_{\emptyset}\Phi(\CC)P_{\emptyset}$) defined by
\begin{align*}
    F(X)(A) = V_X \rho_X(A) V_X^*, \quad A \in \Phi.
\end{align*}
In particular, if we set $V_\ut = \sum_{\beta\in\SSS} I_{\ut * \beta}^{\beta}$, 
where $I_{\ut * \beta}^{\beta} = I_\beta \in Hom(\ut * \beta, \beta)$, then $F(\ut) = Id_{\Phi}$.

Furthermore, by \cref{intre_lemma}, for every $X,Y \in \CC$ and $f \in Hom(X, Y)$ we have
\begin{align*}
    (V_{Y} (f \otimes I) V_{X}^*)\,F(X)(A) = F(Y)(A)\,(V_{Y} (f \otimes I) V_{X}^*), \quad 
    \forall A \in \Phi.
\end{align*}
\end{definition}

Let $End(\Phi)$ be the (strict) C$^*$-tensor category of endomorphisms of $\Phi$ 
(normal faithful unital and *-preserving), 
with tensor structure given on objects by the composition of endomorphisms and 
C$^*$-norm on arrows given by the one of $\Phi$. The tensor unit $Id_{\Phi}$ of
$End(\Phi)$ is simple because $\Phi$ is a factor.

By \cref{intre_lemma}, it is clear that 
$$X \mapsto F(X), \quad f \mapsto F(f) = V_{Y}(f \otimes I)V_{X}^*$$ 
is a *-functor from $\CC$ into $End(\Phi)$.
In particular, $F(I_X) = I_{F(X)}$ and $F(f^*) = F(f)^*$ hold.

We will show that $F$ is a fully faithful (non-strictly) tensor functor from $\CC$ to $End(\Phi)$, hence an equivalence 
of C$^*$-tensor categories onto its image. 
Recall that $XY = X\otimes Y$ in $\CC$, by our convention, and we denote by $\otimes$ the tensor product
in $End(\Phi)$ as well. The (unitary) associator of the functor 
\begin{align*}
    J_{X_1, X_2}: F(X_1) & \otimes F(X_2) = Ad(V_{X_1} \rho_{X_1}(V_{X_2}))) \circ \rho_{X_1}(\rho_{X_2}(\cdot))\\
    &\mapsto  AdV_{X_1 \otimes X_2} \circ \rho_{X_1 \otimes X_2}=F(X_1 \otimes X_2)  
\end{align*}
is defined for every $X_1,X_2\in\CC$ by
\begin{align*}
    J_{X_1, X_2} = V_{X_1 \otimes X_2}(I_{X_1*X_2}^{X_1 \otimes X_2} \otimes I) 
    \rho_{X_1}(V_{X_2}^*) V_{X_1}^* \, \in \Phi.
\end{align*}

\begin{lemma}\label{naturlity_lemma}
The family of morphisms $\{J_{X_1, X_2}\}$ is natural in $X_1$ and $X_2$, i.e., the following 
diagram commutes in $End(\Phi)$
   \begin{center}
   \adjustbox{scale=1}{
   \begin{tikzcd}[column sep = large, row sep = large]
        F(X_1) \otimes F(X_2) \arrow [r, rightarrow, "J_{X_1, X_2}"] \arrow [d, rightarrow, swap, "F(f_1) \otimes F(f_2)"] & 
        F(X_1 \otimes X_2) \arrow [d, rightarrow, "F(f_1 \otimes f_2)"] \\
        F(Y_1) \otimes F(Y_2) \arrow [r, rightarrow , "J_{Y_1, Y_2}"] & 
        F(Y_1 \otimes Y_2)
    \end{tikzcd}}
    \end{center}
for every $f_i \in Hom(X_i, Y_i)$, $X_i, Y_i\in\CC$, $i = 1,2$. 
\end{lemma}

\begin{proof}
    Since $\rho_X(P_\emptyset) = P_X$, $X \in \CC$, \cref{intre_lemma} implies that
    \begin{align*}
        J_{Y_1, Y_2}(F(f_1) &\otimes F(f_2))\\ 
        &= V_{Y_1 \otimes Y_2}(I^{Y_1 \otimes Y_2}_{Y_1 * Y_2} \otimes I)
        \rho_{Y_1}(V_{Y_2}^*)(f_1 \otimes I)\rho_{X_1}(V_{Y_2}) \rho_{X_1}(f_2 \otimes I) 
        \rho_{X_1}(V_{X_2}^*)V_{X_1}^*\\
        &=V_{Y_1 \otimes Y_2}[(I^{Y_1 \otimes Y_1}_{Y_1 * Y_2} \otimes I)(f_1 \otimes I)
        \rho_{X_1}(\sum_{\beta \in \SSS} f_2 * I_{\beta})]\rho_{X_1}(V_{X_2}^*)V_{X_1}^*
    \end{align*}
    and $F(f_1 \otimes f_2)J_{X_1, X_2} = V_{Y_1 \otimes Y_2}[((f_1 \otimes f_2) \otimes I)
        P_{X_1 \otimes X_2} (I^{X_1 \otimes X_1}_{X_1 * X_2} \otimes I)]\rho_{X_1}(V_{X_2}^*)V_{X_1}^*$.
    
    Moreover
    \begin{align*}
        &(I^{Y_1 \otimes Y_1}_{Y_1 * Y_2} \otimes I)(f_1 \otimes I)\rho_{X_1}(\sum_{\beta \in \SSS} 
        f_2 * I_{\beta})
        = \sum_{\beta \in \SSS}[I^{Y_1 \otimes Y_1}_{Y_1 * Y_2}(f_1 * f_2)] * I_\beta \\
        =& \sum_{\beta \in \SSS}[(f_1 \otimes f_2) I^{X_1 \otimes X_1}_{X_1 * X_2}] * I_\beta = 
        ((f_1 \otimes f_2) \otimes I)P_{X_1 \otimes X_2} (I^{X_1 \otimes X_1}_{X_1X_2} \otimes I),
    \end{align*}
    since $I^{Y_1 \otimes Y_1}_{Y_1 * Y_2}(f_1 * f_2) = (f_1 \otimes f_2) I^{X_1 \otimes X_1}_{X_1 * X_2}$
    and $f_1 * f_2 = f_1 \otimes f_2 \in Hom(X_1*X_2, Y_1*Y_2)$ by definition. 
    Thus $\{J_{X_1, X_2}\}$ is natural.
\end{proof}

\begin{lemma}\label{p_lemma}
   $F$ is a tensor functor. 
\end{lemma}

\begin{proof}
    By \cite[Section 2.4]{PSDV}, we only need to check that the following diagram commutes
   \begin{center}
   \adjustbox{scale=1}{
    \begin{tikzcd}[column sep = large, row sep = large]
       F(X) \otimes F(Y) \otimes F(Z) \arrow [r, rightarrow, "J_{X,Y} \otimes I_{F(Z)}" ] \arrow[d, rightarrow] & 
       F(X \otimes Y) \otimes F(Z) \arrow [r, rightarrow, "J_{X \otimes Y, Z}"] &
       F(X \otimes Y \otimes Z) \arrow [d, rightarrow] \\
       F(X) \otimes F(Y) \otimes F(Z) \arrow [r, rightarrow, "I_{F(X)} \otimes J_{Y, Z}"] &
       F(X) \otimes F(Y \otimes Z) \arrow [r, rightarrow, "J_{X, Y \otimes Z}"] & 
       F(X \otimes Y \otimes Z). 
    \end{tikzcd}}
    \end{center}
    By \cref{intre_lemma} and by $\rho_X(\rho_Y(P_Z)) = P_{X*Y*Z}$, we have 
\begin{align*}
    &J_{X\otimes Y, Z} (J_{X,Y} \otimes I_{F(Z)})=  
    V_{X \otimes Y \otimes Z}(I_{(X \otimes Y) * Z}^{X \otimes Y \otimes Z} \otimes I)
    (I_{X * Y}^{X \otimes Y} \otimes I)P_{X*Y*Z}\rho_{X}(\rho_{Y}((V_{Z}^*)) \rho_{X}(V_Y^*) V_X^*,\\
    &J_{X, Y\otimes Z}(I_{F(X)} \otimes J_{Y,Z}) = 
    V_{X \otimes Y \otimes Z}(I_{X * (Y\otimes Z)}^{X \otimes Y \otimes Z} \otimes I)
    \rho_X(P_{Y \otimes Z}(I_{Y * Z}^{Y\otimes Z} \otimes I))\rho_{X}(\rho_{Y}(V_Z^*))\rho_{X}(V_Y^*)V_X^*.
\end{align*}
Note that
\begin{align*}
    (I_{(X \otimes Y) * Z}^{X \otimes Y \otimes Z} \otimes I)(I_{X * Y}^{X \otimes Y} \otimes I)
    P_{X * Y * Z}
    &= \sum_{\beta \in \SSS} I^{X \otimes Y \otimes Z}_{X * Y * Z} * I_\beta\\ 
    &= (I_{X * (Y\otimes Z)}^{X \otimes Y \otimes Z} \otimes I)
    \rho_X(P_{Y \otimes Z}(I_{Y * Z}^{Y\otimes Z} \otimes I)),
\end{align*}
hence the diagram commutes.
\end{proof}

We denote by $\AAA(\CC)$ be the von Neumann subalgebra of $\Phi(\CC)$ generated by $\{f \in Hom(\beta_1, W* \beta_2):
\beta_1,\beta_2 \in \SSS, W \in W(\CC)\}$. By \cref{KMS_aut}, there exists a faithful normal conditional expectation 
$\tilde{E}: \Phi(\CC) \rightarrow \AAA(\CC)$ with respect to $\varphi$ defined by \cref{varphi_def}.
Note that $\tilde{E}|_{\Phi} = E : \Phi \rightarrow \AAA(\ut)$ defined in \cref{def_E}. 

\begin{lemma}\label{amg_free_lemma}
    For every $X \in \CC$, let $\Phi(X) = P_{X}\Phi(\CC)P_{X}$ and $\AAA(X) = \AAA(\CC) \cap \Phi(X)$. Then
    \begin{align*}
        (\Phi(X), \tilde{E}) = (\MMM_1(X), \tilde{E}) *_{\AAA(X)} (\MMM_2(X), \tilde{E}),
    \end{align*} 
    where $\MMM_j(X)$, $j=1,2$, is the von Neumann subalgebra of $\Phi(X)$ generated by 
    \begin{align*}
        \{u_2 A u_1^*: A \in \MMM_j, u_i \in Hom(\beta_i, X*\gamma_i), \beta_i, \gamma_i \in \SSS, i = 1,2\},
    \end{align*}
    and $\MMM_1$ and $\MMM_2$ are defined by \cref{M1_eq} and \cref{M2_eq} respectively.
\end{lemma}

\begin{proof}
    By \cref{M2_eq}, it is not hard to check that $\Phi(X)$ is generated by 
    $\MMM_1(X)$ and $\MMM_2(X)$. Since $\tilde{E}(u_2 A u_1^*) = u_2 E(A) u_1^*$ and  
    \begin{align*}
        \tilde{E}[(u_1 A_1 v_1^*)(u_2 A_2 v_2^*)\cdots (u_n A_n v_n^*)] 
        = u_1 E[(A_1 v_1^*u_2)(A_2 v_2^* u_3) \cdots A_n]v_n^*= 0,  
    \end{align*}
    whenever $A_k \in KerE \cap \MMM_{i_k}$ with 
    $i_1 \neq i_2 \neq \cdots \neq i_n$, $u_i \in Hom(\beta_i, X \gamma_i)$ and 
    $v_i \in Hom(\beta_i', X \gamma_i')$, the lemma is proved.
\end{proof}

\begin{lemma}\label{X_case}
    The endomorphisms $\rho_X$ of $\Phi(\CC)$ defined in \cref{embedding_ut} map $\Phi$ into $\Phi(X)$, moreover $\rho_X(\Phi)' \cap \Phi(X) \cong Hom(X,X)$ for every $X \in \CC$.
\end{lemma}

\begin{proof}
    With the notation as in the previous lemma, it is easy to check that $\MMM_1(X) \cong \oplus_{\gamma \in \SSS} 
    \BBB(\oplus_{\beta \in \SSS} 
    Hom(\gamma, X*\beta)) \otimes \NNN_\gamma$ (see \cref{M1_eq}) where
    $\NNN_\gamma$ is the factor generated by $\Gamma(\xi_\gamma)$, $\xi_\gamma = I_\gamma \in \OOO^{\gamma * \ut}_\gamma$.
    By the definition of $\rho_X$, we have
    \begin{align*}
        \rho_X(\Gamma(\xi_\beta)) = \sum_{\gamma \in \SSS, u \in \OOO^{X*\beta}_{\gamma}}u
        \Gamma(u^* \otimes I_{\ut} \circ I_X \otimes \xi_{\beta} \circ u)u^*
        = \sum_{\gamma \in \SSS, u \in \OOO^{X*\beta}_{\gamma}}u
        \Gamma(\xi_\gamma)u^*.
    \end{align*}
    Let $\{U_\gamma\}_{\gamma \in \SSS}$ be the unitaries defined before \cref{popa_unitary}, it clear that
    \begin{align*}
        \sum_{\gamma \in \SSS} \sum_{\beta \in \SSS, u \in \OOO^{X*\beta}_{\gamma}}u
        U_\gamma u^* \in \rho_X(\Phi), 
    \end{align*}
    Then \cref{popa_unitary} implies $\rho_X(\Phi)' \cap \Phi(X) \subset \MMM_1(X)$.
    
    Since $\Phi$ is a factor, \cite[Chapter 2, Proposition 2]{Dix} implies that the reduction map by 
    $I_{X * \ut} = \rho_X(I_\ut)$ is an isomorphism
    on $\rho_X(\Phi)' \cap \MMM_1(X) \ni A \mapsto AI_{X * \ut} \in
    (I_{X * \ut}\rho_X(\Phi)I_{X * \ut})' \cap I_{X * \ut}\MMM_1(X)I_{X * \ut}$.

    Note that $I_{X * \ut}\MMM_1(X)I_{X * \ut} \cong \oplus_{\gamma \in \SSS}
    \BBB(Hom(\gamma, X*\ut)) \otimes \NNN_\gamma$, and
    $\rho_X(\Gamma(\xi_\ut)) = \sum_{\gamma \in \SSS, u \in \OOO^{X*\ut}_{\gamma}}u \Gamma(\xi_\gamma)u^*$.
    We have
    \begin{align*}
        (I_{X * \ut}\rho_X(\Phi)I_{X * \ut})' \cap I_{X * \ut}\MMM_1(X)I_{X * \ut}
        \subseteq \bigoplus_{\gamma\in\SSS} 
        \BBB(Hom(\gamma, X*\ut)) \otimes I_{\NNN_\gamma} 
        \cong Hom(X, X), 
    \end{align*}
    and the claim is proved by \cref{intre_lemma}.
\end{proof}

\begin{lemma}\label{faithful_lemma}
    $Hom(F(X_1), F(X_2)) \cong Hom(X_1, X_2)$ for every $X_1, X_2 \in \CC$.
\end{lemma}

\begin{proof}
    Let $f_i \in Hom(X_i, X_1\oplus X_2)$ be isometries such that $f_{i}^* f_{i} = I_{X_i}$ and 
    $e_1 + e_2 = I_{X_1 \oplus X_2}$, where $e_i = f_{i} f_{i}^*$, $i=1,2$. 
    Here $X_1\oplus X_2\in\CC$ denotes the direct sum of objects in $\CC$. Then $T \in Hom(F(X_1), F(X_2))$ if
    and only if
    \begin{align*}
        (f_2 \otimes I)V_{X_2}^* T V_{X_1} \rho_{X_1}(A)(f_1^* \otimes I) 
        = (f_2 \otimes I)\rho_{X_2}(A)V_{X_2}^*TV_{X_1}(f_1^* \otimes I), \quad \forall A \in \Phi.
    \end{align*}
    Thus $(f_2 \otimes I)V_{X_2}^* T V_{X_1}(f_1^* \otimes I) 
    \in (e_2 \otimes I)[\rho_{X_1 \oplus X_2}(\Phi)' \cap \Phi(X_1 \oplus X_2)](e_1 \otimes I)$
    and \cref{X_case} implies the result. 
\end{proof}

Denoted by $End_0(\Phi)$ the subcategory of finite dimensional (i.e., finite index) endomorphisms of $\Phi$, by \cref{naturlity_lemma}, \cref{p_lemma} and \cref{faithful_lemma} we can conclude

\begin{theorem}\label{thm_realization}
$F$ is a fully faithful tensor *-functor (non-strict but unital), 
hence it gives an equivalence of $\CC$ with (the repletion of) 
its image $F(\CC) \subset End_0(\Phi)$ as C$^*$-tensor categories.
\end{theorem}

In the case of categories with infinite and non-denumerable spectrum $\SSS$, the factor $\Phi \cong \Phi(H(\ut))''$ is not $\sigma$-finite (and it is either of type II$_\infty$ or III$_\lambda$, $\lambda\in (0,1]$). 
Indeed, $\pi_\tau(\AAA(\ut)) \subset \Phi(H(\ut))''$ and $\pi_\tau(I_\beta)\in\pi_\tau(\AAA(\ut))$, $\beta\in\SSS$, are uncountably many mutually orthogonal projections summing up to the identity. However, $\Phi \cong \widetilde\Phi \otimes \BBB(\HHH)$ where $\widetilde\Phi$ is a $\sigma$-finite factor and $\HHH$ is a non-separable Hilbert space. Moreover, $Bimod_0(\widetilde\Phi \otimes \BBB(\HHH)) \simeq Bimod_0(\widetilde\Phi)$ where, e.g., $Bimod_0(\widetilde\Phi)$ denotes the category of faithful normal $\widetilde\Phi$-$\widetilde\Phi$ bimodules (in the sense of correspondences) with finite dimension (i.e., with finite index) and $\simeq$ denotes an equivalence of C$^*$-tensor categories. In the type III case, we also have $Bimod_0(\widetilde\Phi) \simeq End_0(\widetilde\Phi)$. In \cite{GFY} we give a proof of these last facts, that we could not find stated in the literature, cf.\ \cite[Corollary 8.6]{Rie}, \cite{Mey}. Recalling \cref{thm_freecorner} and summing up the previous discussion, we can conclude

\begin{theorem}\label{thm_ampliation}
Any rigid C$^*$-tensor category with simple unit and infinite spectrum $\SSS$ (not necessarily denumerable) can be realized as finite index endomorphisms of a $\sigma$-finite type III factor, or as bimodules of a $\sigma$-finite type II factor, such as the free group factor $L(F_\SSS)$ that arises by choosing the trivial Tomita structure.
\end{theorem}

We leave open the questions on whether every countably generated rigid C$^*$-tensor category with simple unit can be realized on a hyperfinite factor, and on whether the free Araki-Woods factors defined by Shlyakhtenko in \cite{DSFQ} are universal as well for rigid C$^*$-tensor categories with simple unit, as our construction (see\ \cref{unit_case}, \cref{connes_s} and \cref{thm_freecorner}, the latter in the trivial Tomita structure case) might suggest.

\appendix \section{Amalgamated free product}\label{appendix_am_prod}

In \cite{YU-def}, the amalgamated free product of $\sigma$-finite 
von Neumann algebras is constructed. The purpose of this appendix 
is to demonstrate that the method used in \cite[Section 2]{YU-def} can be applied 
to construct the amalgamated free product of arbitrary von Neumann algebras.
 
\begin{definition}\label{def_am_prod}
Let $\{\MMM_s\}_{s \in S}$ be a family of von Neumann algebras having a common von Neumann 
subalgebra $\NNN$ such that each inclusion $I \in \NNN \subset \MMM_s$ has a normal faithful
conditional expectation $E_s: \MMM_s \rightarrow \NNN$.   
The \textbf{amalgamated free product} $(\MMM, E) = *_{\NNN, s \in S}(\MMM_s, E_s)$ of the family $(\MMM_s, E_s)$ 
is a von Neumann algebra $\MMM$ with a conditional expectation $E$ satisfying:
\begin{enumerate}
    \item There exist normal *-isomorphisms $\pi_s$ from $\MMM_s$ into $\MMM$
        and $\pi_s|_{\NNN} = \pi_{s'}|_{\NNN}$, $s, s' \in S$. Let $\pi = \pi_s|_{\NNN}$,
    \item $\MMM$ is generated by $\pi_s(\MMM_s)$, $s \in S$,
    \item $E$ is a faithful normal conditional expectation onto $\pi(\NNN)$ such that
        $E(\pi_s(m)) = \pi(E_s(m))$ and the family 
        $\{\pi_{s}(\MMM_s)\}$ is free in $(\MMM, E)$, i.e., 
        \begin{align*}
        E(\pi_{s_1}(m_1) \cdots \pi_{s_k}(m_k)) =  0
        \end{align*}
        if $s_1 \neq s_2 \neq \cdots \neq s_k$ and $m_i \in Ker E_{s_i}$, $i=1,\ldots,k$, where $k\geq 1$.
\end{enumerate}
\end{definition}

\begin{lemma}\label{unique_lemma}
    Let $\{(\MMM_s, E_s)\}_{s \in S}$ and $\{(\tilde{\MMM}_s, \tilde{E}_s)\}_{s \in S}$ 
    be two families of 
    von Neumann algebras having common unital von Neumann subalgebras $\NNN$ and $\tilde{\NNN}$, 
    respectively, and $E_s: \MMM_s \rightarrow \NNN$ and $\tilde{E}_s: \tilde{\MMM}_s \rightarrow 
    \tilde{\NNN}$ normal faithful conditional expectations.
    Let $(\MMM, E) = *_{\NNN,s \in S}(\MMM_s, E_s)$ and $(\tilde{\MMM}, \tilde{E}) 
    = *_{\tilde{\NNN}, s \in S}(\tilde{\MMM}_s, \tilde{E}_s)$ with
    normal *-isomorphisms $\pi_s$ from $\MMM_s$ into $\MMM$ and $\tilde{\pi}_s$ 
    from $\MMM_s$ into $\tilde{\MMM}$.
    If there is a family of *-isomorphisms $\rho_s: \MMM_s \rightarrow \tilde{\MMM}_s$ such that
    $\rho_s(E_s(m)) = \tilde{E}_s(\rho_s(m))$ and $\rho_{s}|_{\NNN} = \rho_{s'}|_{\NNN}$,
    then there exists a unique *-isomorphism $\rho: \MMM \rightarrow \tilde{\MMM}$ such that
    $\tilde{E} \circ \rho = \rho \circ E$ and
    $\rho: \pi_s(m) \mapsto \tilde{\pi}_s(\rho_s(m))$ for every $m \in \MMM_s$, $s \in S$. 
\end{lemma}

\begin{proof}
    Without loss of generality, we can assume that $\MMM_s \subseteq \MMM$ and 
    $\tilde{\MMM}_s \subseteq \tilde{\MMM}$ and $\pi_s(m)=m$, $\tilde{\pi}_s(\tilde{m})=\tilde{m}$
    for every $s \in S$, $m \in \MMM_s$ and $\tilde{m} \in \tilde{\MMM}_s$.

    Let $\varphi$ be a n.s.f. weight on $\NNN$ and $\AAA_0$ be the *-subalgebra of $\MMM$ generated by
$\cup_{s \in S}(\NN(\MMM_s, \varphi \circ E_s) \cap \NN(\MMM_s, \varphi \circ E_s)^*)$.
For each $m \in \NN(\MMM, \varphi\circ E)$, we use $m \varphi^{1/2}$ to denote the 
vector given by the canonical injection $m \in \NN(\MMM, \varphi\circ E) \mapsto 
L^2(\MMM, \varphi \circ E)$.
Note that $\AAA_0 \subset \NN(\MMM, \varphi \circ E) \cap \NN^*(\MMM, \varphi \circ E)$.
We claim that $\AAA_0 \varphi^{1/2}$ is dense in 
$L^2(\MMM, \varphi \circ E)$.

For $a \in \NN(\MMM, \varphi \circ E) \cap \NN^*(\MMM, \varphi \circ E)$ and 
$\varepsilon > 0$, there is a spectral projection $p$ of $E(a^*a)$ such that
$\|(a - ap)\varphi^{1/2}\| < \varepsilon$ and $\varphi(p)< \infty$.
Since $\AAA_0$ is a dense *-subalgebra of $\MMM$ in the strong operator topology, there exists 
$b \in \AAA_0$ such that $\|(b-a)p\varphi^{1/2}\| < \varepsilon$.
Thus $\AAA_0$ is dense in $L^2(\MMM, \varphi \circ E)$.

Note that $\tilde{\varphi} = \varphi \circ \rho^{-1}_s$ is a n.s.f. weight on $\tilde{\NNN}$.
Similarly, let $\tilde{\AAA}_0$ be the *-subalgebra generated by 
$\cup_{s \in S}(\NN(\tilde{\MMM}_s, \tilde{\varphi} \circ \tilde{E}_s) \cap 
\NN(\tilde{\MMM}_s, \tilde{\varphi}\circ \tilde{E}_s)^*)$, then 
$\tilde{\AAA}_0\tilde{\varphi}^{1/2}$ is dense in $L^2(\tilde{\MMM}, \tilde{\varphi} \circ \tilde{E})$.

Let $U$ be the linear map given by linear extension of 
\begin{align*}
U m_1 m_2 \cdots m_k \varphi^{1/2}
= \rho_{s(1)}(m_1) \rho_{s(2)}(m_2) \cdots \rho_{s(k)}(m_k) \tilde{\varphi}^{1/2} 
\end{align*}
where $m_i \in \NN(\MMM_{s(i)}, \varphi\circ E_{s(i)}) \cap \NN(\MMM_{s(i)}, \varphi\circ E_{s(i)})^*$, 
$i(1) \neq i(2) \neq \cdots \neq i(k)$, $k\geq 1$.
By \cref{def_am_prod}(3), it is not hard to check that $U$ can be extended to a
unitary from $L^2(\MMM,\varphi \circ E)$ onto 
$L^2(\tilde{\MMM}, \tilde{\varphi} \circ \tilde{E})$ such that
\begin{align*}
    U \pi_{\varphi}(m) U^* = \pi_{\tilde{\varphi}} \circ \rho_s(m), \quad 
    \forall m \in \MMM_s, s \in S,
\end{align*}
where $\pi_{\varphi}$ and $\pi_{\tilde{\varphi}}$ are GNS representations of $\MMM$ and
$\tilde{\MMM}$ w.r.t. $\varphi \circ E$ and $\tilde{\varphi} \circ \tilde{E}$.
Then $\rho(\cdot) = \pi_{\tilde{\varphi}}^{-1}(U \pi_{\varphi}( \cdot )U^*)$ satisfies 
the conditions in the lemma.
\end{proof}

To construct the amalgamated free product, we first fix a n.s.f. weight $\varphi$ on $\NNN$ and 
regard $\MMM_s$ as a concrete von Neumann algebra acting on 
$\HHH_s = L^2(\MMM_s, \varphi \circ E_s)$ for every $s\in S$. 
Let $\MMM^{\circ}_s = Ker E_s$ and 
$\NN(\MMM_s, \varphi\circ E_s)=\{m \in \MMM_s: \varphi \circ E_s(m^*m) < \infty \}$. 
For each $m \in \NN(\MMM_s, \varphi\circ E_s)$, we use $m \varphi^{1/2}$ to denote the 
vector given by the canonical injection $\NN(\MMM_s, \varphi\circ E_s) \rightarrow \HHH_s$.
Let $J_s$, $\Delta_s$ be the modular operators and $\{\sigma_t^{\varphi \circ E_s}\}_{t\in\R}$ be 
the modular automorphism group of $\MMM_s$ w.r.t. $\varphi \circ E_s$.
Recall that $\HHH_s$ is a $\MMM_s$-$\MMM_s$ bimodule with left and 
right actions given by
\begin{align*}
    m_1 (a\varphi^{1/2}) m_2 = J_s m_2^*J_s m_1a\varphi^{1/2}, \quad m_1, m_2 \in \MMM_s, 
    a \in \NN(\MMM_s, \varphi \circ E_s).
\end{align*}
By \cite[VIII, Lemma 3.18]{TKII}, $m_1(a\varphi^{1/2})m_2 = 
m_1 a \sigma^{\varphi \circ E_s}_{-i/2}(m_2)\varphi^{1/2}$ if $m_2 \in 
\DDD(\sigma^{\varphi \circ E_s}_{-i/2})$.
In the following, let $\AAA_s$ be the maximal Tomita algebra of the left Hilbert algebra
$\NN(\MMM_s, \varphi\circ E_s) \cap \NN(\MMM_s, \varphi\circ E_s)^*$, i.e.,
\begin{align*}
    \AAA_s = \{m \in \MMM_s: m \in \DDD(\sigma^{\varphi \circ E_s}_{z}) \text{ and }
\sigma^{\varphi \circ E_s}_{z}(m) \in \NN(\MMM_s, \varphi\circ E_s) \cap \NN(\MMM_s, \varphi\circ E_s)^*, \\
\forall z \in \C\},
\end{align*}
and $\AAA_{\NNN} = \NNN \cap \AAA_s = E_s(\AAA_s)$.
For $m_1$, $m_2 \in \AAA_s$, we use $m_1 \varphi^{1/2} m_2$ to denote 
$m_1 \sigma^{\varphi \circ E_s}_{-i/2}(m_2)\varphi^{1/2}$. In particular, $\varphi^{1/2}m_2=
\sigma^{\varphi \circ E_s}_{-i/2}(m_2) \varphi^{1/2}$ and 
$\braket{\varphi^{1/2}m_1}{\varphi^{1/2}m_2} = \varphi \circ E_s(m_2 m_1^*)$
(see \cite[Appendix B]{ACCG}, \cite[Section 2]{TF}, \cite[Chapter IX]{TKII}).

Note that $n \in \NN(\NNN, \varphi) \mapsto n \varphi^{1/2} \in \HHH_s$
gives the natural embedding of $L^2(\NNN, \varphi)$ into $L^2(\MMM_s, \varphi \circ E_s)$.
By \cite{TKCD}, we can identify $L^2(\NNN, \varphi)$ as an $\NNN$-$\NNN$ sub-bimodule of  
$L^2(\MMM_s, \varphi \circ E_s)$, moreover $J_s$ and $\Delta_s$ can be decomposed as
\begin{align*}
   J_s =
    \begin{pmatrix}
        J_\NNN & \\
        & J_{s}^{\circ}
    \end{pmatrix}, \quad
    \Delta_s = 
    \begin{pmatrix}
        \Delta_{\NNN} & \\
        & \Delta_s^{\circ} 
    \end{pmatrix}
    \quad \mbox{on 
    $ \begin{pmatrix}
        L^2(\NNN,\varphi)\\
        \HHH_s^{\circ}
    \end{pmatrix}$},
\end{align*}
where $\HHH_s^{\circ} = \HHH_s \ominus L^2(\NNN, \varphi)$. 

Define $\HHH = L^2(\NNN, \varphi) \oplus \oplus_{k=1}^{\infty} (\oplus_{s_1 \neq \cdots \neq s_k}
\HHH_{s_1}^{\circ} \otimes_{\varphi} \cdots \otimes_{\varphi} \HHH_{s_k}^{\circ})$, where
$\otimes_\varphi$ is the relative tensor product \cite{LS}. By \cite[Remarque 2.2(a)]{LS},
\begin{align*}
    \HHH_0 &= \AAA_{\NNN}\varphi^{1/2} \oplus \oplus_{k=1}^{\infty} (\oplus_{s_1 \neq \cdots \neq s_k}
    span\{a_1 \varphi^{1/2} \otimes \cdots \otimes a_k \varphi^{1/2}: a_i \in \AAA^{\circ}_{s_{i}}\})\\
    &=\AAA_{\NNN}\varphi^{1/2} \oplus \oplus_{k=1}^{\infty} (\oplus_{s_1 \neq \cdots \neq s_k}
    span\{\varphi^{1/2} a_1\otimes \cdots \otimes \varphi^{1/2} a_k: a_i \in \AAA^{\circ}_{s_i}\})
\end{align*}
is dense in $\HHH$, where $\AAA^{\circ}_s = \AAA_s \cap \MMM^{\circ}_s$. 
By the definition of relative tensor
product and \cite[Lemme 1.5 (c)]{LS} (or \cite[Proposition 3.1]{TF}), we have 
\begin{align*}
    &\braket{a_1 \varphi^{1/2} \otimes \cdots \otimes a_k\varphi^{1/2}}{b_1 \varphi^{1/2} 
    \otimes \cdots \otimes b_k\varphi^{1/2}}
     \\ & \hspace{20mm} = \varphi (E_{s_k}(a_k^* E_{s_{k-1}}(\cdots a_2^*E_{s_1}(a_1^*b_1)b_2\cdots )b_k)),\\
    &\braket{\varphi^{1/2} a_1 \otimes \cdots \otimes \varphi^{1/2}a_k}{\varphi^{1/2}b_1 
    \otimes \cdots \otimes \varphi^{1/2}b_k}
     \\ & \hspace{20mm} = \varphi (E_{s_1}(b_1 E_{s_2}(\cdots b_{k-1}E_{s_k}(b_k a_k^*)a_{k-1}^*\cdots )a_1^*)).
\end{align*}

For each $s \in S$, let $\HHH(s,l) = L^2(\NNN, \varphi) \oplus \oplus_{k=1}^{\infty}
(\oplus_{s_1 \neq \cdots \neq s_k, s_1 \neq s} \HHH_{s_1}^{\circ} 
\otimes_{\varphi} \cdots \otimes_{\varphi} \HHH_{s_k}^{\circ})$.
By \cite[Proposition 3.5]{TF} (or \cite[Section 2.4]{LS}), we have unitaries
\begin{align*}
U_s: \HHH_s \otimes_{\varphi} \HHH(s,l) = (L^2(\NNN,\varphi)\oplus \HHH_s^{\circ})
\otimes_{\varphi} \HHH(s,l) \rightarrow \HHH. 
\end{align*}   
For $m \in \MMM_s$, let $\pi_s(m) = U_s m \otimes_{\varphi} I_{\HHH(s,l)}U_s^*$ (see
\cite[Corollary 3.4(ii)]{TF}).
If $m_s \in \AAA_s$, it is not hard to check that $\pi_s(m_s) m \varphi^{1/2} 
= E_s(m_s m)\varphi^{1/2} + (m_s m - E_s(m_s m)) \varphi^{1/2}$ and that 
\begin{align}\label{de_def}
    \pi_s(m_s) a_1 \varphi^{1/2} \otimes \cdots \otimes a_k \varphi^{1/2}=
    \begin{cases}
        \begin{aligned}[b]
            (m_s a_1 &- E_s(m_s a_1))\varphi^{1/2} \otimes a_2 \varphi^{1/2} 
            \cdots \otimes a_k \varphi^{1/2}\\
            +&E_s(m_s a_1)a_2\varphi^{1/2} \otimes \cdots \otimes a_k \varphi^{1/2},
            \quad s_1 = s
        \end{aligned}\\
        \begin{aligned}[b]
            (m_s- & E_s(m_s)) \varphi^{1/2} \otimes a_1 \varphi^{1/2} 
            \otimes \cdots \otimes a_k \varphi^{1/2}\\
            +&E_s(m_s)a_1 \otimes \cdots \otimes a_k \varphi^{1/2}, \qquad \quad\quad  s_1 \neq s 
        \end{aligned}
    \end{cases}
\end{align}
where $m \in \AAA_{s}$ and $a_i \in \AAA^{\circ}_{s_i}$.
Note that $L^2(\NNN,\varphi) \oplus \HHH^{\circ}_s \cong L^2(\MMM_s, \varphi \circ E_s)$ and $\pi_s$ is 
a *-isomorphism. If $n \in \NNN$, then it is clear that $\pi_s(n) = \pi_{s'}(n)$, $s$, $s' \in S$.
And we use $\pi$ to denote $\pi_s|_{\NNN}$.
Define $\MMM$ to be the von Neumann algebra generated by $\{\pi_s(\MMM_s): s \in S\}$.

Let $J$ be the conjugate-linear map defined on $\HHH_0$ by
$J n \varphi^{1/2} = J_{\NNN} n\varphi^{1/2}$, $n \in \AAA_{\NNN}$ and 
$J (a_1 \varphi^{1/2} \otimes \cdots \otimes a_k\varphi^{1/2}) = 
J_{s_k}(a_k \varphi^{1/2}) \otimes \cdots \otimes J_{s_1}(a_1\varphi^{1/2})=
\varphi^{1/2}a_k^* \otimes \cdots \otimes \varphi^{1/2}a_1^*$, where $a_i \in \HHH_{s_i}^{\circ}$.
Note that
\begin{equation}
\label{J_rel}
    \begin{aligned}
    & \braket{\varphi^{1/2} a_k^*\otimes \cdots \otimes \varphi^{1/2}a_1^*}{\varphi^{1/2}
        b_k^* \otimes \cdots \otimes \varphi^{1/2}b_1^*)} \\
    & \qquad \qquad =\braket{b_1 \varphi^{1/2} \otimes \cdots \otimes b_k\varphi^{1/2} }{a_1 
        \varphi^{1/2} \otimes \cdots \otimes a_k\varphi^{1/2}}.    
    \end{aligned}
\end{equation}
Thus $J$ can be extended to a conjugate-linear involutive unitary and we still use $J$ to denote 
its extension. 

The following fact is well-known, at least when the von Neumann algebras are $\sigma$-finite (see
\cite[Lemma 2.2]{YU-def}). 

\begin{proposition}\label{commutant_prop}
    $\pi_{s}(m_s)$ and $J\pi_{s'}(m_{s'}^*)J$ commute for every $m_s \in \MMM_s, m_{s'} \in \MMM_{s'}$, $s, s' \in S$. 
    $\MMM L^2(\NNN,\varphi)$ and $J\MMM J L^2(\NNN,\varphi)$ are both dense in $\HHH$.
\end{proposition}

Let $e_{\NNN}$ and $e_{s}$ be the projections from $\HHH$ onto $L^2(\NNN,\varphi)$ and 
$L^2(\NNN,\varphi) \oplus \HHH_{s}^{\circ}$ respectively.

\begin{proposition}\label{cond_proj}
    Let $s\in S$, then
    $e_s \pi_{s_1}(m_{1}) \pi_{s_2}(m_{2}) \cdots \pi_{s_k}(m_{k}) e_s = 0$ 
    if $m_i \in \MMM^{\circ}_{s_i}$, $s_1 \neq s_2 \neq \cdots \neq s_k$ and $k>1$. Moreover,
    $e_s \pi_{s_1}(m_1) e_s = \delta_{s, s_1} \pi_{s}(m_1) e_s$ and
    $e_{\NNN}\pi_{s_1}(m_1) e_{\NNN} =0$.
    If $n \in \NNN$, then $e_s \pi(n) e_s = \pi(n) e_{s}$ and
    $e_\NNN \pi(n) e_\NNN = \pi(n) e_{\NNN}$. 
\end{proposition}

\begin{proof}
    By \cref{de_def}, an easy calculation verifies all the equations 
    when $m_{s_i} \in \AAA^{\circ}_{s_i}$ and $n \in \AAA_\NNN$.
    Since $\AAA_s$ is dense in the strong operator topology in $\MMM_s$, the assertion is 
    proved.
\end{proof}

By \cref{cond_proj}, it is clear that $e_s \in \pi_s(\MMM_s)'$ ($e_\NNN \in \pi(\NNN)'$) and
$e_s \MMM e_s = \pi_s(\MMM_s)e_s$ ($e_\NNN \MMM e_\NNN = \pi(\NNN)e_\NNN$). 
Let $E_{\MMM_s}$ and $E$ be the normal conditional expectations from 
$\MMM$ onto $\pi_s(\MMM_s)$ and $\pi(\NNN)$ respectively defined by
\begin{align*}
    e_s A e_s = E_{\MMM_s}(A)e_s, \quad e_\NNN A e_\NNN = E(A)e_\NNN \quad A \in \MMM. 
\end{align*}
By \cref{commutant_prop}, $E$ and $E_{\MMM_s}$ are faithful. 
Also note that $E(E_{\MMM_s}(A))= E(A)$ and $E(\pi_s(m)) = \pi(E_s(m))$.
Therefore $(\MMM, E)$ is an amalgamated free product of the family $(\MMM_s, E_s)$ as in \cref{def_am_prod}.

\begin{proposition}\label{free_prod_mo_au}
    Let $(\MMM, E) = *_{\NNN, s \in S}(\MMM_s, E_s)$. Then there exist 
    normal conditional expectations $E_{\MMM_s}: \MMM \rightarrow \pi_s(\MMM_s)$ such 
    that $E \circ E_{\MMM_s} = E$.
    Furthermore, for every n.s.f. weight $\phi$ on $\NNN$, we have
    $\sigma^{\tilde{\phi} \circ E}_t(\pi_s(m))
    = \pi_s(\sigma^{\phi \circ E_s}_t(m))$, $m \in \MMM_s$, where $\tilde{\phi} = \phi \circ \pi^{-1}$.
\end{proposition}

\begin{proof}
    Note that $\tilde{\phi} \circ E = (\tilde{\phi} \circ E|_{\pi_s(\MMM_s)}) \circ E_{\MMM_s}$ and
    $\tilde{\phi} \circ E(\pi_s(m)) = \phi \circ E_s(m)$. Then \cite{TKCD} implies the result.
\end{proof}

By \cref{free_prod_mo_au} and \cite[Lemma 2.1]{ILP}, 
we can identity $L^2(\MMM, \varphi \circ \pi^{-1} \circ E)$ with $\HHH$ by using the 
unitary given in \cref{unique_lemma}.
Let $S$ be the involution on the left Hilbert algebra 
$\NN(\MMM,\varphi \circ \pi^{-1} \circ E) \cap \NN(\MMM,\varphi \circ \pi^{-1} \circ E)^*$,
i.e., $S: m  \mapsto m^*$. 

\begin{proposition}\label{modular_auto}
    $\HHH_0$ is a Tomita algebra with involution $S_0 = S|_{\HHH_0}$ and 
    complex one-parameter group $\{U(z):z \in \C\}$ given by
    \begin{align*}
        U(z)(a_1 \varphi^{1/2} \otimes \cdots \otimes a_k \varphi^{1/2})=
        \sigma^{\varphi \circ E_{s_1}}_{z}(a_1) \varphi^{1/2} \otimes \cdots \otimes 
        \sigma^{\varphi \circ E_{s_k}}_{z}(a_k) \varphi^{1/2},
    \end{align*}
    and $U(z)(n\varphi^{1/2}) = \sigma^{\varphi}_z(n)\varphi^{1/2}$,
    where $a_i \in \AAA^{\circ}_s$ and $n \in \AAA_{\NNN}$.
    In particular, the modular operator $\Delta$ associated with $\varphi \circ \pi^{-1} \circ E$ is 
    the closure of $U(-i)$ and the modular conjugation is $J$.
\end{proposition}

\begin{proof}
    For $a_i$, $b_i \in \AAA^{\circ}_{s_i}$, $s_1 \neq \cdots \neq s_k$,
    let $\xi =a_1 \varphi^{1/2} \otimes \cdots a_k\varphi^{1/2}$ and $\beta = 
    b_1 \varphi^{1/2} \otimes \cdots \otimes b_k \varphi^{1/2}$. 
    Note that $S_0\xi =  a_k^* \varphi^{1/2} \otimes \cdots a_1^*\varphi^{1/2}$. Thus we have
    \begin{align*}
        &\braket{\xi}{U(z)\beta}=\varphi (E_{s_k}(a_k^* E_{s_{k-1}}(\cdots a_2^*E_{s_1}
        (a_1^*\sigma_z^{\varphi \circ E_{s_1}}(b_1))\sigma_z^{\varphi \circ E_{s_2}}(b_2)\cdots )
        \sigma_z^{\varphi \circ E_{s_k}}(b_k)))\\
        =&\varphi (E_{s_k}(\sigma_{-\overline{z}}^{\varphi \circ E_{s_k}}(a_k)^* 
        E_{s_{k-1}}(\cdots \sigma_{-\overline{z}}^{\varphi \circ E_{s_1}}(a_2)^*E_{s_1}
        (\sigma_{-\overline{z}}^{\varphi \circ E_{s_1}}(a_1)^*b_1)b_2\cdots )b_k))\\
        =&\braket{U(-\overline{z})\xi}{\beta}.
    \end{align*}
    By the definition of $J$ and \cref{J_rel}, we have
    \begin{align*}
      \braket{S_0 \xi}{S_0 \beta} = \braket{JU(-i/2)\xi}{JU(-i/2)\beta}=\braket{U(-i/2)\beta}{U(-i/2)\xi} 
        =\braket{U(-i)\beta}{\xi}.
    \end{align*}
    The other conditions in the definition of Tomita algebra can also be checked easily.
    Let $\overline{S}_0$ be the closure of $S_0$. By \cite[Theorem 2.2, VI]{TKII}, $\HHH_0$ is 
    a core for both $\overline{S}_0$ and $S_0^*$. Note that $\HHH_0$ is contained in the maximal 
    Tomita algebra in $\NN(\MMM,\varphi \circ \pi^{-1} \circ E) \cap 
    \NN(\MMM,\varphi \circ \pi^{-1} \circ E)^*$.
    Therefore $\HHH_0 \subset \DDD(S) \cap \DDD(S^*)$ and 

    By \cite[Theorem 2.2, VI]{TKII}, $\Delta$ and $\Delta^{1/2}$ are the closure of 
    $U(-i)$ and $U(-i/2)$ respectively. Since $\HHH_0$ is dense in $\HHH$, the modular 
    conjugation is $J$.
\end{proof}

\section{Technical results on non-$\sigma$-finite algebras}\label{section_nonsigma}

The results of this appendix are well-known in the case of \emph{$\sigma$-finite} 
von Neumann algebras (i.e., von Neumann algebras that admit normal faithful states), 
but we could not find them stated in the literature in more general situations. 
In the following we generalize them to the case of arbitrary von Neumann algebras, 
as we shall need later on in our construction (when considering categories with
uncountably many inequivalent simple objects).

Let $\NNN$ be a von Neumann subalgebra of a von Neumann algebra $\MMM$ and 
let $E$ be a normal, faithful conditional expectation from $\MMM$ onto $\NNN$. 
Then $\MMM$ is a pre-Hilbert $\NNN$-$\NNN$ bimodule, as defined in \cref{section_prelim}, with the $\NNN$-valued inner product 
$\braket{m_1}{m_2}_{\NNN} = E(m_1^*m_2)$. The left action of $\NNN$ on $\MMM$ is clearly normal. 
For a n.s.f. weight $\varphi$ of $\NNN$, the Hilbert space $H_\varphi$ defined in \cref{section_prelim} (see \cref{notation_Hvarphi}), 
with $H = \MMM$, coincides with the GNS Hilbert space $L^2(\MMM, \varphi \circ E)$. The inner product $\braket{\cdot}{\cdot}$ associated with $H$ and $\varphi$ is the GNS inner product $\braket{\cdot}{\cdot}_{\varphi\circ E}$.
We regard $\MMM$ and $\NNN$ as a concrete von Neumann algebras acting on $L^2(\MMM, \varphi \circ E)$.
Recall that $e_\NNN \in \NNN'$ is the Jones projection from $L^2(\MMM, \varphi \circ E)$ onto 
the subspace $L^2(\NNN,\varphi)$, and $e_{\NNN}Te_{\NNN} = E(T)e_{\NNN}$ for every $T\in \MMM$.
Moreover, there exists a unique n.s.f. operator-valued weight $E^{-1}$ 
from $\NNN'$ to $\MMM'$ characterized by
\begin{align*}
    \frac{d(\varphi \circ E)}{d\phi} = \frac{d\varphi}{d(\phi \circ E^{-1})},  
\end{align*}
where $\varphi$ and $\phi$ are n.s.f. weights on $\NNN$ and $\MMM'$ respectively
(see \cite{HOW-I}, \cite{HOW-II}, \cite{Kosaki-Ind}).

The following generalizes \cite[Lemma 3.1]{Kosaki-Ind} to the non-$\sigma$-finite case.

\begin{lemma}\label{E_v_e_n}
    With the above notations $E^{-1}(e_\NNN) = I$.   
\end{lemma}

\begin{proof}
    Let $\xi \in L^2(\NNN, \varphi)$ be a $\varphi$-bounded vector (see \cite{AC}) 
    regarded as an element of $L^2(\MMM, \varphi \circ E)$, namely such that the map 
    $R^{\varphi}(\xi): n \in \NN(\NNN,\varphi) \mapsto n\xi \in L^2(\MMM, \varphi \circ E)$
    extends to a bounded linear operator from $L^2(\NNN, \varphi)$ to $L^2(\MMM, \varphi \circ E)$. 
    Note that
    \begin{align*}
        \|m \xi\|_2 = \|E(m^*m)^{1/2}\xi\|_2 \leq \|R^{\varphi}(\xi)\|\|m\|_2, \quad m
        \in \NN(\MMM, \varphi\circ E).
    \end{align*}
    Thus $\xi$ is a $\varphi\circ E$-bounded vector and the map 
    $R^{\varphi \circ E}(\xi):m \mapsto m\xi$ is a bounded operator in $\MMM'$.

    It is clear that 
    $(I-e_\NNN) R^{\varphi \circ E}(\xi)e_\NNN = 0$ for each $\xi \in L^2(\NNN, \varphi)$. 
    Let $m \in \NN(\MMM, \varphi\circ E)$ such that
    $E(m)=0$. Then
    $\braket{\beta}{R^{\varphi \circ E}(\xi)m}_{\varphi\circ E} 
        = \braket{\beta}{E(m)\xi}_{\varphi\circ E} = 0$, for every $\beta \in L^2(\NNN, \varphi)$. 
    Therefore, we have $R^{\varphi \circ E}(\xi)e_\NNN = e_\NNN R^{\varphi \circ E}(\xi)$ and
    $R^{\varphi \circ E}(\xi)e_\NNN|_{L^2(\NNN, \varphi)}=R^\varphi(\xi)$.
 
    By \cite[Proposition 3]{AC}, there exists a family $\{\xi_\alpha\}_{\alpha \in S}$ of 
    $\varphi$-bounded vectors in $L^2(\NNN,\varphi)$ such that
    \begin{align*}
        e_{\NNN} = \sum_{\alpha} R^{\varphi}(\xi_\alpha)R^{\varphi}(\xi_\alpha)^* 
        =\sum_{\alpha} R^{\varphi \circ E}(\xi_\alpha)R^{\varphi \circ E}(\xi_\alpha)^* e_{\NNN}.
    \end{align*}
    Since $\MMM e_\NNN L^2(\MMM,\varphi\circ E) = span\{m\xi: \xi \in L^2(\NNN, \varphi), m \in \MMM\}$ 
    is dense in $L^2(\MMM,\varphi\circ E)$, we have 
    \begin{align*}
    \sum_{\alpha} R^{\varphi \circ E}(\xi_\alpha)R^{\varphi \circ E}(\xi_\alpha)^* = I.   
    \end{align*}
    Therefore
    $E^{-1}(e_\NNN) = \sum_{\alpha} E^{-1}(R^{\varphi}(\xi_\alpha)R^{\varphi}(\xi_\alpha)^*)
    =\sum_{\alpha} R^{\varphi \circ E}(\xi_\alpha)R^{\varphi \circ E}(\xi_\alpha)^* = I$
    (see \cite[Lemma 3.1]{Kosaki-Ind} for more details on this part). 
\end{proof}

Let $J$ be the modular conjugation associated with $\varphi \circ E$ on the Hilbert space
$L^2(\MMM,\varphi\circ E)$. By \cite[IX, Theorem 4.2]{TKII} (or see \cite{TKCD}), $Je_\NNN J = e_\NNN$.
The basic extension of $\MMM$ by $e_\NNN$ is the von Neumann algebra 
generated by $\MMM$ and $e_\NNN$ (\cite{J}), which coincides with $J \NNN' J$. 
For $A \in \BBB(L^2(\MMM,\varphi\circ E))$, let $j(A) = J A^*J$. Then $\hat{E}: A \in J\NNN'J \mapsto
j \circ E^{-1} \circ j(A) \in \MMM$ is a n.s.f. operator-valued weight. By \cref{E_v_e_n}, 
$\hat{E} (e_\NNN) = I$. Moreover, the same arguments as in the proof of \cite[Proposition 2.2]{ILP} lead us to the following:

\begin{lemma}\label{centralizer_lemma}
    $e_\NNN$ is in the centralizer of $\phi=\varphi \circ E \circ \hat{E}$, i.e., 
    $\sigma^{\phi}_t(e_\NNN) = e_\NNN$, $t \in \R$.
\end{lemma}

\begin{proof}
    By \cite[Lemma 1.3]{Kosaki-Ind}, $\hat{E} = (j \circ E \circ j)^{-1}$. Thus
    \begin{align*}
        \frac{d \phi}{d \varphi \circ j}=
        \frac{d \varphi \circ E}{d(\varphi \circ E \circ j)} = \Delta_{\varphi \circ E}.
    \end{align*}
    By \cite[Theorem 9]{AC} and Takesaki's theorem \cite[IX, Theorem 4.2]{TKII} we conclude $\sigma^{\phi}_t(e_\NNN) 
    = \Delta_{\varphi \circ E}^{it}e_\NNN  \Delta_{\varphi \circ E}^{-it} = e_\NNN$.
\end{proof}

%
%

Along the same line of arguments used in the proof of \cite[Proposition 3.3]{YU}, we have 
the following result on the structure of amalgamated free products in the non-finite,
non-$\sigma$-finite setting (see \cref{appendix_am_prod} for their definition).  

\begin{proposition}\label{relative_commutant}
Suppose that $\MMM_1 \supseteq \NNN \subseteq \MMM_2$ are von Neumann algebras with normal faithful
conditional expectations $E_i: \MMM_i \rightarrow \NNN$, $i=1$, $2$.
Let $$(\MMM, E)= (\MMM_1, E_1)*_{\NNN}(\MMM_2, E_2)$$ be their amalgamated free product
(see \cref{def_am_prod}) and $\varphi$ be a n.s.f. weight on $\NNN$.

For $i =1,2$, we use $\MM_i \subset \MMM_i$ to denote the set of analytic elements $A \in \MMM_i$ such that 
    $\sigma_{z}^{\varphi \circ E}(A) \in \NN(\MMM_i, \varphi\circ E_i) \cap \NN(\MMM_i, \varphi\circ E_i)^*$, $z \in\C$.

    Let $\tilde{\MM}_1$ be a strongly dense *-subalgebra of $\MM_1$ which is globally invariant
    under the modular automorphism group $\sigma^{\varphi \circ E_1}_t$ and such that $E_1(\tilde{\MM}_1)
    \subset \tilde{\MM}_1$.
    Let $\AAA$ be a unital von Neumann subalgebra of the centralizer 
    $(\MMM_1)_{\varphi \circ E_1}$ such that 
    \begin{enumerate}
        \item $\varphi \circ E_1|_{\AAA' \cap \MMM_1}$ is semifinite,
        \item there is a net $U_\alpha$ of unitaries in $\AAA$ satisfying 
            $E_1(A U_\alpha B) \rightarrow 0$ in the strong operator topology 
            for all $A$, $B \in \{I\} \cup \tilde{\MM}_1$.
    \end{enumerate}
    Then any unitary $V \in \MMM$ satisfying
    $V\AAA V^* \subset \MMM_1$ must be contained in $\MMM_1$. In particular $\AAA' \cap \MMM =
    \AAA' \cap \MMM_1$.

    Analogous statements hold replacing $\MMM_1$, $\MM_1$, $\tilde \MM_1$ with
    $\MMM_2$, $\MM_2$, $\tilde \MM_2$.
\end{proposition}

\begin{proof}
    We assume that $\MMM$ acts on $L^2(\MMM, \varphi \circ E)$ and take $V$ as above.
    Let $J$ be the modular conjugation associated with $\varphi \circ E$ and 
    $\hat{E} = j \circ E_{\MMM_1}^{-1} \circ j$ be the dual operator-valued weight from 
    $J \MMM_1' J$ onto $\MMM$, where $E_{\MMM_1}$ is the 
    conditional expectation from $\MMM$ onto $\MMM_1$ defined as in the discussion after \cref{cond_proj}.
    We use $e_{\MMM_1}$ to denote the Jones projection from $L^2(\MMM, \varphi \circ E)$ 
    onto $L^2(\MMM_1, \varphi \circ E_1)$. 
    Note that $e_{\MMM_1}T e_{\MMM_1} = E_{\MMM_1}(T)e_{\MMM_1}$, 
    $T \in \MMM$, and $L^2(\MMM_1, \varphi \circ E_1)$ is a separating set 
    in $L^2(\MMM, \varphi \circ E)$ (see \cref{commutant_prop,unique_lemma}).
    Thus we only need to show that $(I-e_{\MMM_1})V^* e_{\MMM_1} V(I-e_{\MMM_1}) = 0$.

    Note that $V \AAA V^* \subset \MMM_1$, $V^* e_{\MMM_1} V \in \AAA' \cap J \MMM_1' J$
    and $\hat{E}(V e_{\MMM_1} V^*) = 1$ thanks to \cref{E_v_e_n}.
    By taking spectral projections, it is sufficient to show that if 
    $F$ is a projection in $\AAA' \cap J \MMM_1' J$ such that
    $F \leq I - e_{\MMM_1}$ and $\|\hat{E}(F)\| < +\infty$, then $F=0$.
 
    Let $\NN = \tilde{\MM}_1 \cap \NNN$, $\tilde{\MM}_1^{\circ} = Ker E_1|_{\tilde{\MM}_1}$,
    $\MM_2^{\circ} = Ker E_2|_{\MM_2}$ and 
    $\MM = \NN + span\Lambda(\tilde{\MM}_1^{\circ}, \MM_2^{\circ})$, where 
    $\Lambda(\tilde{\MM}_1^{\circ}, \MM_2^{\circ})$ is the set of all alternating words in 
    $\tilde{\MM}_1^{\circ}$ and $\MM_2^{\circ}$. By \cite[Lemma 2.1]{ILP}, $\MM$ is dense
    as a subspace of $L^2(\MMM, \varphi \circ E)$.
    Let $H_0 = span\{A e_{\MMM_1} B: A, B \in \MM \}$. 
    By \cref{centralizer_lemma} and \cite[Theorem 4.7]{HOW-I}, $H_0$  is globally invariant 
    under $\sigma_t^{\phi}$, where $\phi = \varphi \circ E \circ \hat{E}$ is a 
    n.s.f. weight on $J\MMM_1 'J$.
    Therefore $H_0$ is dense as a subspace of $L^2(J\MMM_1 ' J, \phi)$
    by \cite[Lemma 2.1]{ILP}.

    Let $P \in \AAA' \cap \MMM_1$ be a projection such that $0 < \gamma^2 =\phi(PFP)$.
    Let $T = \sum_{i=1}^{n} A_i e_{\MMM_1} B_i \in H_0$ satisfying 
    $\phi(T^*T)^{1/2} \leq 3\gamma /2$ and $\phi((T- (PFP)^{1/2})^*(T- (PFP)^{1/2}))^{1/2} \leq \gamma/5$.
    Moreover, we can assume that $(I-e_{\MMM_1})T=T=T(I-e_{\MMM_1})$ by \cref{centralizer_lemma} and
    \cite[VIII, Lemma 3.18 (i)]{TKII}.   
    Therefore $E_{\MMM_1}(A_i) = 0 = E_{\MMM_1}(B_i)$, $i = 1, \ldots, n$, and we 
    could assume that $A_i$, $B_i \in span(\Lambda(\tilde{\MM}_1^{\circ}, \MM_2^{\circ}) 
    \setminus \tilde{\MM}_1^{\circ})$.
 
    For any $U_\alpha$ as in our assumptions, we have
    \begin{align*}
        \gamma^2 - |\phi(T^*U_\alpha T U{_\alpha}^*)| 
        \leq |\phi((PFP)^{1/2} U_\alpha (PFP)^{1/2}U_\alpha^*)- 
        \phi(T^*U_\alpha TU_\alpha^*)|\leq \gamma^2/2.
    \end{align*}
    Thus 
    \begin{align*}
        \gamma^2 &\leq 2\sum_{i,j=1}^n |\phi(B_i^* e_{\MMM_1} A_i^* U_\alpha A_j 
        e_{\MMM_1}B_jU_\alpha^*)|\\
                 &=2 \sum_{i,j=1}^n |\varphi \circ E(\sigma_{i}^{\varphi \circ E}
                 (B_j)U_\alpha^*B_i E_{\MMM_1}(A_i^*U_\alpha A_j))|\\
                 &\leq 2 (\max_{i,j}\|B_i^*\|\|\sigma_{-i}^{\varphi \circ E}
                 (B_j^*)\|_{2})(\sum_{i,j}\|E_{\MMM_1}(A_i^*U_\alpha 
                 A_j)\|_{2}),
    \end{align*}
    the second equality is due to \cite[Proposition 2.17]{SS}. Note that each operator in
    $\Lambda(\tilde{\MM}_1^{\circ}, \MM_2^{\circ}) \setminus \tilde{\MM}_1^{\circ}$ 
    can be written as
    $C\delta$ where $C \in \{1\} \cup \tilde{\MM}_1^\circ$ and $\delta$ is a reduced word 
    that starts with an operator in $\MM_2^{\circ}$. Let $C_1\delta_1$ and $C_2 \delta_2$ be two
    such words. By the fact $\varphi \circ E(\delta_i^* \delta_i) < +\infty$ and 
    \cref{cond_proj}, we have 
    \begin{align*}
        \|E_{\MMM_1}(\delta_1^* C_1^* U_\alpha C_2 \delta_2)\|_{2}
        &=\|E_{\MMM_1}(\delta_1^* E_1(C_1^* U_\alpha C_2) \delta_2)\|_{2}\\
        &\leq \|\delta_1^*\|\| E_1(C_1^* U_\alpha C_2) \delta_2 \|_{2} \rightarrow 0.
    \end{align*}
    Recall that $A_i\in span(\Lambda(\tilde{\MM}_1^{\circ}, \MM_2^{\circ}) 
    \setminus \tilde{\MM}_1^{\circ})$. The above estimation implies that $\gamma =0$ and $F=0$.
\end{proof}

We conclude with two facts that are used in \cref{section_lambda1}.

\begin{lemma}\label{con_exp_pre_iso}
    Let $\varphi_i$ be a n.s.f. weight on a von Neumann algebra $\MMM_i$, and 
    let $\NNN_i$ be a von Neumann subalgebra of $\MMM_i$ such that $\varphi_i|_{\NNN_i}$ 
    is semifinite, $i=1$, $2$. Let $E_i: \MMM_i \rightarrow \NNN_i$ be the normal faithful
    conditional expectation satisfying $\varphi_i \circ E_i = \varphi_i$.
    If $\rho: \MMM_1 \rightarrow \MMM_2$ is a *-isomorphism such that
    $\rho(\NNN_1) = \NNN_2$ and $\varphi_1 = \varphi_2 \circ \rho$, then
    $\rho \circ E_1 = E_2 \circ \rho$.
\end{lemma}

\begin{proof}
    Let $m \in \NN(\MMM_1,\varphi_1)$. For every $n \in \NN(\NNN_1,\varphi_1|_{\NNN_1})$, we have
    \begin{align*}
        \varphi_2(\rho \circ E_1(n^*m))= \varphi_1(n^* E_1(m)) = \varphi_1(n^* m) = 
        \varphi_2(\rho(n)^* E_2(\rho(m))).
    \end{align*}
    Thus $E_2(\rho(m)) = \rho(E_1(m))$. Since $\NN(\MMM_1,\varphi_1)$ is dense in $\MMM_1$,
    $\rho \circ E_1 = E_2 \circ \rho$.
\end{proof}

\begin{proposition}\label{iso_cor}
    Let $\MMM_1 \supseteq \NNN \subseteq \MMM_2$ be von Neumann algebras.
    Suppose that $\MMM_1$ and $\MMM_2$ are semifinite factors and that $\NNN$ is a 
    discrete abelian von Neumann algebra. Let $\{P_s\}_{s \in S}$ be the set of 
    minimal projections of $\NNN$. Assume that $\tau_1(P_s) = \tau_2(P_s) < \infty$ for every $P_s$,
    where $\tau_1$, $\tau_2$ are tracial weights on $\MMM_1$ and $\MMM_2$ respectively.
    If $P_{s_0}\MMM_1 P_{s_0} \cong P_{s_0}\MMM_2 P_{s_0}$, then there exists a 
    *-isomorphism $\rho: \MMM_1 \rightarrow \MMM_2$ such that $\rho(P_s) = P_s$ and 
    $\rho \circ E_1 = E_2 \circ \rho$, where $E_1$ and $E_2$ are the normal conditional
    expectations satisfying $\tau_1 \circ E_1 = \tau_1$ and $\tau_2 \circ E_2 = \tau_2$.
\end{proposition}

\begin{proof}
    Let $Q_i \in \MMM_i$, $i=1$, $2$, be a subprojection of $P_{s_0}$ such that 
    $\tau_1(Q_1) = \tau_2(Q_2)$ and $\MMM_i \cong  Q_i \MMM_i Q_i \otimes \BBB(\HHH)$.
    Since $Q_1 \MMM_1 Q_1 \cong Q_2 \MMM_2 Q_2$, $\rho = \rho_0 \otimes Id$ is a *-isomorphism 
    between $\MMM_1$ and $\MMM_2$, where $\rho_0$ is a *-isomorphism from
    $Q_1 \MMM_1 Q_1$ onto $Q_2 \MMM_2 Q_2$. By the uniqueness of the tracial state on 
    $Q_1 \MMM_1 Q_1$, we have
    $\tau_1|_{Q_1 \MMM_1 Q_1} = \tau_2|_{Q_2 \MMM_2 Q_2} \circ \rho$. Note that  
    $\tau_i = \tau_i|_{Q_i \MMM_i Q_i} \otimes Tr$. Thus $\tau_1 = \tau_2 \circ \rho$.
    Therefore there exists a unitary $U \in \MMM_2$ such that $U\rho(P_s)U^* = P_s$.
    Thus $AdU \circ \rho$ satisfies the conditions in \cref{con_exp_pre_iso}. 
\end{proof}

\bigskip

\noindent\textbf{Acknowledgements.}
We would like to thank Yemon Choi for letting us including \cref{density_example} in the paper, and Florin R\v{a}dulescu for suggesting a reference and informing us about an alternative proof of the computation of the fundamental group of an uncountably generated free group factor contained in \cref{fun_grp_free_group}.
We are also indebted to Arnaud Brothier and Roberto Longo for comments and questions on earlier versions of the manuscript (some of which are reported at the end of \cref{section_endos}), and to Rui Liu and Chi-Keung Ng for constructive discussions.
L.G. acknowledges the hospitality of the Academy of Mathematics and Systems Science (AMSS), Chinese Academy of Sciences (CAS) during a one-month invitation to Beijing in January 2017 that made this research possible.
Research by the authors is supported by the Youth Innovation Promotion Association, CAS and the NSFC under grant numbers 11301511, 11321101 and 11371290, and by the ERC Advanced Grant QUEST 669240.

\end{document}